\documentclass[11pt]{article}
\usepackage{amsmath,amsthm,amssymb,ifsym} 
\usepackage[twoside]{geometry}
\usepackage{graphicx}
\usepackage{epsfig}   
\usepackage{mathpazo} 
\usepackage[scaled]{helvet} 
\usepackage{courier} 
\normalfont
\usepackage[T1]{fontenc}

\newtheorem{theorem}{Theorem}[section]
\newtheorem{proposition}[theorem]{Proposition}
\newtheorem{lemma}[theorem]{Lemma}

\newtheorem{definition}[theorem]{Definition}
\newtheorem{remark}[theorem]{Remark} 
\numberwithin{equation}{section} 
\numberwithin{figure}{section}  
\usepackage{amssymb, amscd, mathrsfs, wasysym} 

\newcommand \barM {\overline M}
\newcommand \barU {\overline U}
\newcommand \underM {\underline M}
\newcommand \tildeV {\widetilde V}
\newcommand \tildeU {\widetilde U}
\newcommand \tildeM {\widetilde M}
\newcommand \trianglerightNEW \triangleright 
\newcommand \hatv {\widehat v}
\newcommand \barr {\overline r}
\newcommand \underr {\underline r} 
\newcommand \sgn {\text{sgn}} 
\newcommand \Uscr {\mathscr U}
\newcommand \rhot {\widetilde \rho}
\newcommand \velt {\widetilde v}
\newcommand \tilder {\widetilde r}
\newcommand \Mscr {\mathscr M}
\newcommand \Omegat {\widetilde \Omega}
\newcommand \auth {\textsc} 

\newcommand \str {\mathcal E}
\newcommand \bei {\begin{itemize}}
\newcommand \eei {\end{itemize}}
\newcommand \be {\begin{equation}}
\newcommand \bel {\begin{equation}\label}
\newcommand \ee {\end{equation}}
\newcommand \del \partial
\newcommand \RR {\mathbb R}

\newcommand \eps \epsilon 
\let\oldmarginpar\marginpar
\renewcommand\marginpar[1]{\-\oldmarginpar[\raggedleft\footnotesize #1]%
{\raggedright\footnotesize #1}}
\begin{document}

\title{\bf \Large Weakly regular fluid flows with bounded variation
\\
on a Schwarzschild background} 
\author{Philippe G. LeFloch\footnote{
\normalsize Laboratoire Jacques-Louis Lions, Centre National de la Recherche Scientifique, Universit\'e Pierre et Marie Curie (Paris VI), 
4, Place Jussieu, 75252 Paris, France. 
\newline
E-mail : {\sl contact@philippelefloch.org, xiang@ljll.math.upmc.fr}}
\, 
and Shuyang Xiang$^*$}

\date{July 2015}

\maketitle

\abstract{We study the global dynamics of isothermal fluids evolving in the domain of outer communication of a Schwarzschild black hole. We first formulate the initial value problem within a class of weak solutions with bounded variation (BV), possibly containing shock waves. We then introduce a version of the random choice method and establish a global-in-time existence theory for the initial value problem within the proposed class of weakly regular fluid flows. The initial data may have arbitrary large bounded variation and can possibly blow up near the horizon of the black hole. Furthermore, we study the class of possibly discontinuous, equilibrium solutions and design a version of the random choice method in which these fluid equilibria are exactly preserved. This leads us to a nonlinear stability property for fluid equilibria under small perturbations with bounded variation. Furthermore, we can also encompass several limiting regimes (stiff matter, non-relativistic flows, extremal black hole) by letting the physical parameters (mass of the black hole, light speed, sound speed) reach extremal values. 
}

\setcounter{tocdepth}{4}

\tableofcontents 


\section{Introduction} 

We are interested in compressible fluids evolving on a curved background and, specifically, on the domain of outer communication of a Schwarzschild black hole spacetime. The fluid flows under consideration may contain shock waves and we must work within a class of weak solutions to the Euler equations. Our main result in this paper is a global-in-time existence theory for the initial value problem, when the fluid data are prescribed on a spacelike hypersurface. We also establish the nonlinear stability of equilibrium fluid solutions and investigate various limiting regimes when the light speed denoted by $c \in (0,+\infty)$, the (constant) sound speed denoted by $k \in [0,+\infty)$, and the mass of the back hole denoted by $M \in [0, +\infty)$ reach extremal values. 

Recall that Schwarzschild spacetime is a spherically symmetric\footnote{that is, invariant under the group of rotations $SO(3)$}  solution to the vacuum Einstein equations of general relativity, and describes a massive body surrounded by a vacuum region. It is one of the simplest non-flat solution to the Einstein equations, but yet the analysis of (linear and) nonlinear waves propagating on this spacetime 
is very challenging and has attracted a lot of attention by mathematicians in recent years. The present paper is part of a program initiated by the first author on the Cauchy problem for the Einstein-Euler equations \cite{BLSS,BL,GL,PLFbook,LR,LS2,LS3}. 

In the so-called Schwarzschild coordinates $t\geq 0$ and $r \in (2M, +\infty)$, the domain of outer communication of Schwarz\-schild spacetime is described by the metric 
\bel{Scharz} 
g = - \Bigg(1 - {2 M \over r} \Bigg) c^2 \, dt^2 + \Bigg(1 - {2 M \over r} \Bigg)^{-1} dr^2 + r^2 \, g_{S^2},
 \qquad r>2M, 
\ee
in which $g_{S^2}:=d\theta^2+(\sin \theta)^2 \, d\varphi^2$ is the canonical metric on the two-sphere $S^2$, 
with $\theta \in [0, 2\pi)$ and $\varphi \in [0, \pi]$. Observe that the metric coefficients are singular as $r \to 2M$, but this boundary is not a genuine singularity of the spacetime and the coefficients would become regular at $r=2M$ by suitably changing coordinates and the metric could be extended beyond this boundary. The boundary $r=2M$ is the horizon of the black hole, and it is natural to study the dynamics of nonlinear waves outside the black hole region. 

The Levi-Civita connection associated with \eqref{Scharz} being denoted by $\nabla$, the Euler equations for a perfect compressible fluid on this spacetime read
\bel{re-Euler}
\nabla_\alpha \big( T^\alpha_ \beta(\rho,u) \big) = 0,
\ee
in which the energy-momentum tensor 
\bel{tensor-form}
T^\alpha_\beta (\rho,u) = \rho  c^2 u^\alpha u_\beta + p(\rho) \Big( g^\alpha_\beta+ u^\alpha u_\beta\Big)
\ee
(with $c> 0$ denoting the speed of light) 
depends on the mass-energy density of the fluid $\rho: M \mapsto (0, +\infty)$ 
and its velocity field $u= (u^\alpha)$, normalized  to be unit and future oriented: 
\be
u^\alpha u_\alpha=-1, \qquad u^0 > 0. 
\ee 
The pressure $p$ is prescribed as a function  $p=p(\rho)$ of the mass energy density and, for the sake of simplicity, we assume that the fluid flow is isothermal, that is, $p(\rho) = k^2 \rho$ where $k \in (0, c)$ represents the speed of sound. We use here standard notation for the metric $g= (g_{\alpha\beta})$ and its inverse $g^{-1} = (g^{\alpha\beta})$ in an arbitrary local coordinate system 
$x= (x^\alpha)$, where the Greek indices describe $0,1,2,3$. 
We raise and lower indices by using this metric and, for instance, we write $u_\alpha = g_{\alpha\beta} u^\beta$ and we have 
$g^\alpha_\beta = \delta^\alpha_\beta$ (the Kronecker symbol).

The content of this paper is as follows. In Section~\ref{sec:2}, we formulate the Euler equations in our context and establish hyperbolicity and genuine nonlinearity properties. In Section~\ref{sec:3}, we formally derive several simpler models, arising when the light speed 
sound speed and/or black hole mass approach extremal values.  Our model takes the form of a nonlinear hyperbolic system of balance laws. such systems were first investigated (for rather different applications) by Dafermos and Hsiao \cite{DH}, Liu \cite{Liu2} and, later, \cite{Glimm-marshal,HL,Dafermos}; see also Dafermos \cite{Dafermos} the references cited therein. We also refer to \cite{BL,Christodoulou1,GroahTemple} for the related problem of self-gravitating fluids in spherical symmetry. 

A systematic study of the class of steady state solutions to the Euler model under consideration is one of the main contribution of the present paper. In Section~\ref{sec:4}, we first study the non-relativistic model, by taking into account the effect of the mass of the black hole. Next, in Section~\ref{sec:5}, we treat the full Euler model on a Schwarzschild background and, in particular, we establish that (smooth) steady state solutions are defined on intervals of the form $(2M, \underr_*)$ or $(\barr_*, +\infty)$. 

Our next task is to study the Riemann problem which is solved in Section~\ref{sec:6}, while the generalized Riemann problem based on prescribing two steady state solutions (rather than constant states) separated by a jump discontinuity is investigated in Section~\ref{sec:7}. 

In Section~\ref{sec:8}, we are then in a position to establish an existence theory for general flows of isothermal fluids evolving in the domain of outer communication \eqref{Scharz}. The technique developed earlier in Grubic and LeFloch \cite{GL} (in a different geometric setup) applies and provides us with the desired global-in-time result. Recall that, according to Nishida \cite{Nishida} and Smoller and Temple \cite{SmollerTemple} who treated fluid flows in flat space, provided all curved geometrical effects are (formally) suppressed, a suitable notion of total variation is available and, specifically, the total variation of the log of the matter density is non-increasing in time. For the fluids on a Schwarzschild background under consideration in the preset paper, we also need to take geometrical terms into account and the total variation may grow, but yet is uniformly controlled on any compact interval of time. Furthermore, an analysis of the solutions near the horizon is also necessary and we observe that no boundary condition is required at $r=2M$ and that solutions need not have finite bounded variation near the horizon, as is the case for some steady state solutions.  

We also propose here a version of the random choice method which we design from piecewise equilibrium solutions and, in turn, preserves equilibria exactly. We then prove that equilibria are nonlinearly stable under small BV perturbations, and the proposed technique provides a possible approach in order to investigate the time-asymptotic behavior of weak solutions.   
Finally, in Section~\ref{sec:9}, we briefly consider the models obtained when the physical parameters take extremal values. Our total variation estimate is uniform with respect to these parameters, so that our main theorem has counterparts for these limiting systems. 


\section{The Euler equations on a Schwarzschild background}
\label{sec:2}

\subsection*{Derivation of the Euler equations} 

By using the subscripts $(0,1,2,3)$ to denote the coordinates $(t,r,\theta,\varphi)$, we can write 
\be
(g_{\alpha\beta}) = 
\left(
\begin{array}{cccc}
 -(1-2M/r)c^2 & 0 & 0 & 0 \\
0 & (1-2M/r)^{-1} & 0 &0 \\
 0 & 0 & r^2&0 \\
 0 & 0& 0& r^2(\sin \theta)^2 \, 
\end{array}
\right),
\ee
with inverse 
\be
(g^{\alpha\beta}) = 
\left(
\begin{array}{cccc}
 -(1-2M/r)^{-1} c^{-2}& 0 & 0 & 0 \\
0 & (1-2M/r) & 0 &0 \\
 0 & 0 & r^{-2}&0 \\
 0 & 0& 0& r^{-2}\sin^{-2}\theta
\end{array}
\right)
\ee
and, by using $\Gamma_{\alpha\beta}^\gamma := {1\over 2}g^{\gamma\theta}(\del_\alpha g_{\beta\theta}+\del_ \beta g_{\alpha\theta} -\del_\theta  g_{\alpha\beta})$, a tedious calculation shows that the non-vanishing Christoffel symbols are  
\bel{Christoffel}
\aligned
& \Gamma_{00}^1= {c^2 M\over r^2}(r-2M),
\qquad 
&&\Gamma_{11}^1= - {M\over r(r-2M)},
\qquad
&&&&\Gamma_{01}^0= {M\over r(r-2M)},
\\
& \Gamma_{12}^2= {1\over r},
\qquad
&&\Gamma_{22}^1=-(r-2M),
\qquad 
&&&&\Gamma_{13}^3= {1\over r},
\\
& \Gamma_{33}^1=-(r-2M)(\sin \theta)^2,
\qquad
&&\Gamma_{33}^2=-\sin\theta\cos\theta, 
\qquad
&&&&\Gamma_{23}^3= {\cos\theta\over \sin\theta}.
\endaligned
\ee

On the other hand, we can express the Euler equations \eqref{re-Euler} in the form 
$$
\aligned
\del_0 T^{0 \beta}+\del_jT^{j\beta}+\Gamma_{00}^0T^{0 \beta}
& +\Gamma_{j0}^jT^{\beta0}+\Gamma_{0j}^0T^{j\beta}+\Gamma_{jk}^jT^{k\beta}
  +\Gamma_{00}^\beta T^{00}+2\Gamma_{j0}^\beta T^{j0}+\Gamma_{jk}^\beta T^{jk} = 0
\endaligned
$$
and, in view of \eqref{Christoffel}, write the Euler equations on a Schwarzschild background as 
\bel{Euler0}
\aligned
&\del_0 \Big( r(r-2M)T^{00} \Big) + \del_1 \Big( r(r-2M)T^{01} \Big) = 0,
\\ 
&\del_0 \Big( r(r-2M)T^{01} \Big) + \del_1  \Big( r(r-2M)T^{11} \Big) = \Omega_1, 
\\
& \Omega_1 := 3MT^{11} - {c^2M \over r^2} (r-2M)^2T^{00} + r(r-2M)^2T^{22} + r(\sin \theta)^2 \, (r-2M)^2T^{33}.
\endaligned
\ee
Here, we have assumed that not only the background geometry but also the fluid flows are {\sl spherically symmetric}, so that the ``transverse'' components of the fluid velocity vanish: 
$
T^{02} =T^{03} = 0.
$
Next, recalling the expression \eqref{tensor-form} of the energy-momentum tensor, we find (with $=p(\rho)$) 
\bel{EulerC}  
\aligned
& \del_0 \Big(  r(r-2M) \big( pu^1u^1+(1-2M/r)^2c^4\rho u^0u^0 \big) \Big)
 + \del_1 \Big(  r(r-2M) (p+c^2\rho)u^0u^1 \Big) = 0,
\\
& \del_0 \Big(  r(r-2M) (p+c^2\rho)u^0u^1 \Big) 
  + \del_1 \Bigg(  r(r-2M) \Big( pu^0 u^0+(1-2M/r)^{-2}\rho u^1u^1\Big) \Bigg) 
= \Omega_1, 
\\
&\Omega_1 = 2r(r-2M)^2p + 3M \Big( pu^1 u^1+(1-2M/r)^2 c^4\rho u^0u^0 \Big)
\\
 & \hskip.8cm   - {c^2M \over r^2} (r-2M)^2 \Big( pu^0u^0+(1-2M/r)^{-2}\rho  u^1u^1\Big).
\endaligned
\ee
Observe that the `first' Euler equation admits a `conservative form', while the second one is a general `balance law'.

By definition, the velocity vector satisfies $(1-2M/r)c^2 u^0 u^0- (1-2M/r)^{-1}u^1 u^1=1$ and $u^0 > 0$, and we find it 
convenient to introduce the {\bf rescaled velocity vector} 
\bel{eq:def-v} 
v^0 := {u^0 \over \eps}, \qquad v^1 := {u^1 \over \eps}, \qquad \text{ with } \eps := {1\over c}. 
\ee
Hence, the components of the energy-momentum tensor read 
$$
 \aligned
& T^{00}  = (1-2M/r)^2 {1  \over  \eps^2} \rho v^0 v^0+  \eps^2 p v^1 v^1, \quad &&T^{01} = (\rho+ \eps^2 p )v^0 v^1, 
\\
& T^{11} = \eps^2 pv^0 v^0+ (1-2M/r)^{-2} \eps^2 \rho v^1 v^1, \quad  &&T^{22} =T^{33} =p.
\endaligned
$$
and the system \eqref{EulerC} takes the form:
\bel{EulerNEW}
\aligned
& \del_0 \Big(  r(r-2M) \big( (1-2M/r)^2 \rho v^0v^0+\eps^4pv^1v^1\big) \Big) + \del_1 \Big(  r(r-2M) \eps^2(\rho+ \eps^2p)v^0v^1 \Big) = 0,
\\
& \del_0 \Big(  r(r-2M) \Big( ( \rho+\eps^2p)v^0v^1 \Big) \Big) + \del_1 \Big(  r(r-2M) \Big(  \eps^2(pv^0 v^0+ (1-2M/r)^{-2} \rho v^1v^1) \Big) \Big) = \Omegat,
\\
& \Omegat := {3M\over  \eps^2} \Big( \eps^4pv^1 v^1+(1-2M/r)^2 \rho v^0v^0 \Big)
\\
& \qquad - {M \over r^2} (r-2M)^2 \Big( pv^0v^0+(1-2M/r)^{-2}\rho  v^1v^1\Big) +2r(r-2M)^2p, 
\endaligned
\ee
supplemented by the relation for the velocity vector
\bel{eq:395}
(1-2M/r) v^0 v^0- \eps^2 (1-2M/r)^{-1}v^1 v^1=1, \qquad v^0 > 0.
\ee
It is convenient also to introduce the {\bf scalar velocity} 
\bel{eq-def-scalV}
v:= {1 \over (1 - 2M/r)} {v^1 \over v^0}, 
\ee
leading us to 
$$
(v^0)^2 = {1 \over (1-\eps^2 v^2)(1-2M/r)},
\qquad (v^1)^2= (1-2M/r){v^2 \over 1- \eps^2 v^2}.
$$
In summary, we have shown that the {\bf Euler system on a Schwarzschild background} takes the form: 
\bel{Euler}
\aligned 
& \del_0 \Big( r^2 {\rho+ \eps^4 p v^2 \over 1 - \eps^2 v^2} \Big) +\del_1 \Big(r(r-2M) {\rho+\eps^2 p \over 1 - \eps^2 v^2} v\Big) = 0,
\\
& \del_0 \Big(r(r-2M) {\rho+\eps^2 p \over 1 - \eps^2 v^2}v \Big) + \del_1 \Big((r-2M)^2 {\rho v^2+ p \over 1 - \eps^2 v^2} \Big)
\\& =3M   \Big( 1 - {2M \over r} \Big) {\rho v^2+ p \over 1 - \eps^2 v^2}  -M {r-2M \over \eps^2 r} {\rho+ \eps^4 p v^2 \over 1 - \eps^2 v^2}+2{(r-2M)^2 \over r}p. 
\endaligned 
\ee

\begin{remark}
\label{Minkowski}
1. 
In the limit $M\to0$, the Schwarzschild metric converges to the Minkowski metric in radial coordinates 
\bel{eq-Mino} 
g = - c^2 \, dt^2 + dr^2 + r^2 \, g_{S^2}
\ee
and from \eqref{Euler} we deduce the radially-symmetric Euler equations in Minkowski space:  
\bel{inMinkowski}
\aligned
& \del_0 \Big(  r^2 (\rho v^0v^0+\eps^4pv^1v^1) \Big) + \del_1 \Big(  r^2 \eps^2(\rho+p \eps^2)v^0v^1 \Big) = 0,
\\
& \del_0 \Big( r^2(\rho+\eps^2p)v^0v^1 \Big) + \del_1  \Big(  r^2 \eps^2(pv^0 v^0+ \rho v^1v^1) \Big) = 2 r \, p, 
\endaligned
\ee
with $v^0 v^0- \eps^2 v^1 v^1=1$ and $v^0 > 0$ and $p=p(\rho)$. 

2. In the singular limit $v \to \pm 1/\eps$, the (unit) velocity vector $v=(v^0, v^1)$ converges (after normalization!) to a null vector, namely:  
$$
(1-\eps^2 v^2)^{1/2}(1-2M/r)^{1/2} (v^0, v^1 )  =
(1, (1-2M/r) v ) \to (1, \pm (1-2M/r)/\eps).
$$

\end{remark}


\subsection*{Hyperbolicity and genuine nonlinearity properties}

Throughout the rest of this section, we regard \eqref{Euler} as a {\bf system of nonlinear balance laws}, that is, 
\be
\del_0 U + \del_1 F(U,r) = S(U,r)
\ee
(with obvious notation) and we study the homogeneous part  $\del_0 U+ \del_1 F(U,r_0) = 0$, where the expressions $F$ and $S$ are evaluated at some fixed $r_0 >2M$. We determine necessary and sufficient conditions ensuring the hyperbolicity and genuine nonlinearity properties for \eqref{Euler}. We are going to rewrite the homogeneous part of \eqref{EulerNEW} in the diagonal form (with the source-terms suppressed)
\be
\del_0 w + \lambda(w,z, r_0) \del_1 w = 0, \qquad
\del_0 z + \mu(w,z, r_0) \del_1 z = 0 
\ee
for a suitable choice of functions $w =w(\rho,v)$ and $z =z(\rho,v)$, refered to as the {\bf Riemann invariants,}
and $\lambda=\lambda(\rho,v, r_0)$ and $\mu= \mu(\rho,v,r_0)$, refered to as the {\bf wave speeds}.

\begin{lemma}
\label{Riemann-inva} 
For the Euler system on a Schwarzschild background \eqref{Euler}, a choice of Riemann invariants is
\bel{Riemannin}
\aligned
w = {1 \over 2\eps}\ln \Bigg({1 + \eps v\over 1 -\eps v} \Bigg) + \int_1^\rho {\sqrt{p'(s)}\over s+\eps^2 p(s)}ds,
\qquad
z = {1 \over 2\eps}\ln \Bigg({1 +\eps v\over 1 -\eps v} \Bigg) - \int_1^\rho {\sqrt{p'(s)}\over s+\eps^2 p(s)}ds, 
\endaligned
\ee
while the corresponding eigenvalues read 
\bel{eignvalue}
\aligned
& \lambda := \Bigg(1- {2M\over r_0} \Bigg){v-\sqrt {p'(\rho)}\over{1-\eps^2 \sqrt{p'(\rho)}v}},
\qquad
 \mu := \Bigg(1- {2M\over r_0} \Bigg){v+\sqrt {p'(\rho)} \over{1 + \eps^2  \sqrt{p'(\rho)}v}}. 
\endaligned
\ee 
\end{lemma} 

\begin{proof} 1. In order to determine the Riemann invariants, we may fix a time $t_0 \geq 0$ and search for solutions depending on the self-similar variable $y := {r - r_0 \over t-t_0}$ (further studied in Section~\ref{sec-raf} below), therefore satisfying $- y {dw \over dy} + \lambda(w,z, r_0) {dw \over dy} = 0$ and $- y {dz \over dy} + \mu(w,z, r_0) {dz \over dy} = 0$. Either $w$ or $z$ must thus be constant for such solutions. Moreover, by parametrizing such solutions by one of the unknown variables, say with the density $\rho$, we can regard the unknowns $v^0$ and $v^1$ as functions of $\rho$ and, using a prime to denote the derivative with respect to  $\rho$, we find 
\bel{eq:257}
\aligned 
\Big( (1-2M/r_0)^2 \rho v^0v^0+\eps^4pv^1v^1\Big) '\del_0  \rho+
\Big(  \eps^2(\rho+ \eps^2p)v^0v^1 \Big)'\del_1\rho = 0,
\\ 
 \Big( ( \rho+\eps^2p)v^0v^1 \Big) '\del_0  \rho
 +\Big(  \eps^2(pv^0 v^0+ (1-2M/r_0)^{-2} \rho v^1v^1\Big) '\del_1\rho= 0,
\endaligned
\ee
where we have neglected low-order, algebraic terms. By differentiating \eqref{eq:395}, we also have 
\bel{eq:260}
(1-2M/r_0)(v^0)' v^0- (1-2M/r_0)^{-1}\eps^2 (v^1)' v^1= 0.
\ee
By combining the two equations in \eqref{eq:257} together, we obtain 
$$
\Big((\eps^2p+\rho)v^0v^1\Big)'^2= \Big(pv^0v^0+(1-2M/r_0)^{-2} \rho v^1v^1\Big)'\Big( \eps^4pv^1v^1+ (1-2M/r_0)^2 \rho v^0v^0 \Big)', 
$$
from which we deduce  
$$
\aligned
& p'\Big((1-2M/r_0)(v^0)^2-(1-2M/r_0)^{-1}(\eps v^1)^2\Big)^2
\\
& = (\eps^2 p+\rho)^2\Big( \big( (v^0)'v^1+v^0(v^1)' \big)^2-4v^0(v^0)'v^1(v^1)'\Big).
\endaligned
$$
Using again \eqref{eq:395}, we find
\bel{eq:262}
{(v^0)'\eps v^1 -c v^0(\eps v^1)' \over (1-2M/r_0) (v^0)^2-(1-2M/r_0)^{-1}\eps^2 (v^1)^2}\pm {\eps \sqrt {p'} \over \eps^2 p+ \rho} = 0.
\ee
After integration, we see that 
\bel{eq:263}
{1 \over 2}\ln \Bigg({(1-2M/r_0) v^0+\eps v^1\over(1-2M/r_0)  v^0-\eps v^1} \Bigg) \pm \int_1^\rho \eps {\sqrt{p'(s)}\over s+\eps^2 p(s)}ds
\ee
is a constant for the solutions under consideration. This calculation provides us with the Riemann invariants
$$
\aligned
w = {1 \over 2\eps}\ln \Bigg({(1-2M/r_0)v^0+\eps v^1\over (1-2M/r_0) v^0-\eps v^1} \Bigg)
+ \int_1^\rho {\sqrt{p'(s)}\over s+\eps^2 p(s)}ds,
\\
z = {1 \over 2\eps}\ln \Bigg({(1-2M/r_0) v^0+\eps v^1\over (1-2M/r_0) v^0-\eps v^1} \Bigg) -\int_1^\rho {\sqrt{p'(s)}\over s+\eps^2 p(s)}ds, 
\endaligned
$$
which take the form \eqref{Riemannin} by replacing $v^0$ and $v^1$ by their expression in terms of $v= {1  \over 1-2M/r_0} {v^1  \over v^0}$. 

\

2. We determine the eigenvalue $\lambda$ from the first equation in the system $\del_ t U+ \del_r F(U,r_0) = 0$, 
that is, 
\bel{eq:270}
 \lambda= {{(1-2M/ r_0)^2 \Big( (p \eps^2+\rho)v^0v^1\Big)'}\over{\Big( \eps^4 pv^1v^1+(1-2M/r_0)^2 \rho v^0v^0 \Big)'}}.
\ee
In view of \eqref{eq:260} and \eqref{eq:262} (where we take the minus sign), we have 
\bel{eq:272}
(v^0)'=-(1-2M/r_0)^{-1}{\eps^2 \sqrt{p'} \over \eps^2 p+\rho}v^1,
\qquad 
(v^1)'=-(1-2M/r_0){\sqrt {p'} \over \eps^2p+ \rho} v^0.
\ee
Therefore, the `first' eigenvalue reads 
$$
\aligned 
 \lambda= & (1-2M/r_0)^2{{(p' \eps^2+1)v^0 v^1 - \sqrt{p'} \Big((1-2M/r_0)(v^0)^2+(1-2M/r_0)^{-1}(\eps v^1)^2\Big)}\over{\eps^4  p'v^1v^1+ (1-2M/r_0)^2v^0v^0-2(1-2M/r_0) \eps^2 \sqrt{p'}v^0v^1}}
 \\
  = & (1-2M/r_0)^2{{((1-2M/r_0) v^0-\eps^2 \sqrt{p'} v^1)((1-2M/r_0)^{-1}v^1- \sqrt{p'}v^0)}\over{((1-2M/r_0) v^0-\eps^2  \sqrt{p'}v^1)^2}}
 \\
 =& (1-2M/r_0){{(1-2M/r_0)^{-1} v^1-\sqrt{p'}v^0}\over{v^0- \eps^2 (1-2M/r_0)^{-1}  \sqrt{p'}v^1}}.
\endaligned
$$
Recalling that $v= {1  \over 1-2M/r_0} {v^1  \over v^0}$, we obtain the desired expression for $\lambda$. The arguments for $\mu$ are entirely similar.  
\end{proof}

We arrive at the following result.

\begin{proposition}[Necessary and sufficient conditions for hyperbolicity and genuine nonlinearity]
\label{nonlinearpo}
\noindent
\begin{enumerate}
\item The Euler system on a Schwarzschild background \eqref{Euler} (within the range ${r>2M}$)
is {\bf strictly hyperbolic}, that is, admits two real and distinct wave speeds, if and only if the pressure satisfies the condition
\bel{eq=-hyp}
p'(\rho) > 0  \qquad \text {for all } \rho> 0.
\ee

\item This system is {\bf genuinely nonlinear}, that is, the derivatives ${\del \lambda \over dw}$ and ${\del \mu \over dz}$ never vanish, if and only if the pressure satisfies the condition
\bel{nonlinear}
\rho \, p'' +2 p' +\eps^2 \big( p'' p-2 (p')^2 \big) > 0 \qquad \text {for all } \rho> 0
\ee
and, therefore, is {\bf linearly degenerate},  that is, the derivatives ${\del \lambda \over dw}$ and ${\del \mu \over dz}$ identically vanish, 
 if and only if the pressure satisfies the condition
\bel{linearnly-degenerate}
\rho \, p'' +2 p' +\eps^2 \big( p''  p - 2 (p')^2 \big) = 0  \qquad \text {for all } \rho> 0.
\ee

\end{enumerate}
\end{proposition}

When the sound speed is a constant $k$ (which is the case of main interest in the present paper), that is, when $p= k^2 \rho$ (with $0 < k< 1/\eps$), the eigenvalues read 
\be
\aligned
\lambda=\Big(1- {2M\over r} \Big){v-k \over 1 - \eps^2 k v},
\qquad
\mu=\Big(1- {2M\over r} \Big){v+k  \over 1 + \eps^2  k v}, 
\endaligned
\ee
and the Riemann invariants take the form 
\be
\aligned
w = {1 \over 2\eps}\ln \Big({1 + \eps v\over 1 -\eps v} \Big) + {k \over 1 + \eps^2 k^2} \ln \rho,
\qquad
z = {1 \over 2\eps}\ln \Big({1 +\eps v\over 1 -\eps v} \Big) - {k \over 1 + \eps^2 k^2} \ln \rho.
\endaligned
\ee
The Euler system, therefore, is strictly hyperbolic and genuinely nonlinear in this case.

\begin{proof} In view of Lemma~\ref{Riemann-inva}, the condition $p'> 0$ is the necessary and sufficient condition for the eigenvalues to be real. Moreover, by definition, the first family $\lambda$ (the second family $\mu$, respectively) is genuinely nonlinear if and only if $\del_w\lambda \neq 0$ (and  $\del_z\mu\neq 0$, resp.). We compute  
$$
\del_w\lambda=\Big({{\del \lambda}\over {\del v^0}}(v^0)'+ {{\del \lambda}\over {\del v^1}}(v^1)'{{\del \lambda}\over {\del v^0}}+ {{\del \lambda}\over {\del \rho}} \Big){\del \rho\over\del w},
$$
following with the calculations: 
$$
\aligned
& {\del \lambda  \over \del v^0} =-(1-2M/r){1 - (\eps^2p'(\rho))v^1\over\Big((1-2M/r)v^0- \eps^2 \sqrt {p'(\rho)}v^1\Big)^2}, 
\\
& {\del \lambda \over \del v^1} = (1-2M/r){(1-\eps^2p'(\rho))v^0 \over \Big((1-2M/r)v^0-\eps^2 \sqrt {p'(\rho)}v^1\Big)^2}, 
\\
& {\del \lambda \over \del\rho} =- (1-2M/r){p''(\rho) \over 2 \sqrt {p'(\rho)} \Big((1-2M/r)v^0- \eps^2  \sqrt {p'(\rho)} v^1\Big)^2}.
\endaligned
$$
Combining these formulas, we obtain
$$
\del_w\lambda= -(1-2M/r){{p''\over 2 \sqrt {p'(\rho)}}+ {(1- \eps^2p) \sqrt {p'(\rho)}  \over{\eps^2p+\rho}}\over\Big((1-2M/r)v^0- \eps^2  \sqrt {p'(\rho)} v^1\Big)^2} {\del \rho\over\del w}.
$$
A similar calculation gives the result associated with the second eigenvalue:
$$
\del_z\mu=-(1-2M/r){{p''\over 2 \sqrt {p'(\rho)}}+ {(1- \eps^2p) \sqrt {p'(\rho)}  \over{\eps^2p+\rho}} \over
\Big((1-2M/r)v^0+ \eps^2  \sqrt {p'(\rho)} v^1\Big)^2}{\del \rho\over\del z}.
$$
Therefore, the sufficient and necessary condition for genuine nonlinearity is \eqref{nonlinear}. On the contrary, the system is linearly degenerate if and only if  \eqref{linearnly-degenerate} holds.
\end{proof}


\subsection*{Linearly degenerate equations of state}
 
The following special case is of particular interest. 

\begin{proposition}[Linearly degenerate equations of state]
\label{linearlly-degenerate} 
The Euler system \eqref{Euler} is linearly degenerate if and only if the pressure (which is defined up to a constant) takes one of the forms (for all $\rho> 0$) 
\bel{eq:847}
p(\rho) = 0
\qquad \text{ or } 
\qquad 
p(\rho) = {\rho \over \eps^2}, 
\qquad \text{ or } 
\qquad p(\rho) = - {A^2 \over \rho+\eps^2  B}, 
\ee
where $A, B > 0$ are arbitrary constants and only the latter two pressure-laws lead to a strictly hyperbolic model. 
\end{proposition}
 
We thus have only two strictly hyperbolic and linearly degenerate models: 

\bei 

\item {\bf Case $p= {1\over \eps^2} \rho$.} The system is well-defined within the full range $\rho>0$ and $|v| < 1/\eps$. The eigenvalues 
$$
- \lambda = \mu = (1-2M/r) /\eps
$$
are independent of the dependent variables, while the Riemann invariants read 
$$
w = {1\over2\eps}\ln\Bigg({1 + \eps v \over 1 - \eps v} \Bigg) + {1\over2\eps}\ln\rho,
\qquad 
z = {1\over2\eps}\ln\Bigg({1 + \eps v \over 1 - \eps v} \Bigg) - {1\over2\eps}\ln\rho.
$$

\item When $p= - {A^2\over  \rho+\eps^2  B}$, the system is well-defined in {\sl limited range of $\rho$, only.} For instance, when 
 $p= - {A^2\over  \rho+\eps^2  B}$,
the eigenvalues read 
$$
\lambda= (1-2M/r){\rho v-A \over \rho -A\eps^2v},
\qquad 
\mu= (1-2M/r){\rho v +A \over \rho +A\eps^2v}
$$
while the Riemann invariants are 
$$
w = {1\over2\eps} \ln\Bigg({1 + \eps v\over 1 -\eps v} \Bigg) + {1 \over 2\eps}\ln \Bigg({\rho-\eps A \over\rho + \eps A} \Bigg),
\qquad 
z = {1\over2\eps} \ln\Bigg({1 + \eps v \over  1 - \eps v} \Bigg) - {1 \over 2\eps}\ln \Bigg({\rho-\eps A \over\rho + \eps A} \Bigg).
$$
This model can be considered within the range $|\rho| < \eps A$ and $|v| < 1/\eps$ (even with negative density values). 

\eei 
 
\begin{proof} From Proposition~\ref{nonlinearpo}, we recall the condition $\rho \, p'' +2 p' +\eps^2(p'' p-2 p'^2) = 0$. If we set $q := \eps^2p+\rho$, we thus need to solve the ordinary differential equation
\bel{eq:280}
q''q-2(q')^2+6q'-4 = 0.
\ee
We treat $q$ as an {\sl independent variable} and set ${dq (\rho) \over d \rho} =: \nu(q)$, hence  
$$
{d^2q \over d \rho^2} = {d\nu \over d q}{dq \over d\rho} = \nu(q) {d\nu \over dq}. 
$$
We see that \eqref{eq:280} transforms into a {\sl separable} equation for the function $\nu=\nu(q)$, that is, provided ${(\nu-1)(\nu-2)}$ does not vanish 
\bel{eq:493}
{\nu \over (\nu-1)(\nu-2) } {d\nu \over dq} = {2 \over q}
\ee
or else $\nu \equiv 1$ or $\nu \equiv 2$. Solutions satisfying ${dq \over d\rho} \equiv 1$ correspond a constant pressure function, since
${dq \over d\rho} := \eps^2p'+1 =1$ implies that $p$ is a constant. The condition ${dq (\rho) \over d \rho} =2$ generates the solutions of the form $p(\rho) = {\rho \over \eps^2} + C$. Finally, by integrating \eqref{eq:493}, we find the third class of solutions. 
\end{proof}


\section{Formal derivation of simplified models}
\label{sec:3}

\subsection*{Fluid flows with constant sound speed}

In this section, we formally analyze the structure of the Euler equation in a Schwarzschild background. 
We focus our attention on the Euler system \eqref{Euler} when the sound speed is assumed to be a constant $k \in [0, 1/\eps]$, that is, with the pressure law $p(\rho) = k^2 \rho$, \eqref{Euler} becomes 
\bel{Euler-conM}
\aligned 
& \del_t\Bigg( r^2 {1 + \eps^4 k^2 v^2 \over 1 - \eps^2 v^2} \rho\Bigg) +\del_r \Bigg(r(r-2M) {1 + \eps^2 k^2 \over 1 - \eps^2 v^2} \rho v\Bigg) = 0,
\\
& \del_t \Bigg(r(r-2M) {1 + \eps^2 k^2\over 1 - \eps^2 v^2} \rho v \Bigg) + \del_r \Bigg((r-2M)^2 {v^2+ k^2 \over 1 - \eps^2 v^2} \rho \Bigg) 
\\
&=3M   \Big( 1 - {2M \over r} \Big) {v^2+ k^2 \over 1 - \eps^2 v^2} \rho -M {r-2M \over \eps^2 r} {1 + \eps^4 k^2 v^2 \over 1 - \eps^2 v^2} \rho+2{(r-2M)^2 \over r}k^2 \rho. 
\endaligned 
\ee
It will be necessary to rescale the mass $M$ and we thus set 
\be
m := {M \over \eps^2}
\ee
and refer to the following system as {\bf the family of Euler models $\Mscr (\eps, k, m)$} 
\bel{Euler-con}
\aligned 
& \del_t\Bigg( r^2 {1 + \eps^4 k^2 v^2 \over 1 - \eps^2 v^2} \rho\Bigg) +\del_r \Bigg(r(r-2\eps^2 m) {1 + \eps^2 k^2 \over 1 - \eps^2 v^2} \rho v\Bigg) = 0,
\\
& \del_t \Bigg(r(r-2\eps^2 m) {1 + \eps^2 k^2\over 1 - \eps^2 v^2} \rho v \Bigg) + \del_r \Bigg((r-2\eps^2 m)^2 {v^2+ k^2 \over 1 - \eps^2 v^2} \rho \Bigg) 
\\
&= {3\eps^2 m \over r} \big( r-2\eps^2 m \big) {v^2+ k^2 \over 1 - \eps^2 v^2} \rho
    - {m  \over r} \big( r-2\eps^2 m \big) {1 + \eps^4 k^2 v^2 \over 1 - \eps^2 v^2} \rho
    + {2 \over r} \big( r-2\eps^2 m\big)^2 k^2 \rho. 
\endaligned 
\ee
Here the main unknowns are the mass-energy density $\rho> 0$ and the scalar velocity $|v |< 1/\eps$, and are defined for $r>2\eps^2 m$. %
We are interested in investigating limiting regimes determined by extremal values of the physical parameters, i.e. the mass of the black hole $m \in (0, +\infty)$, the light speed ${1\over \eps} \in (0, +\infty)$, and the sound speed $k \in (0, 1/\eps)$. 
Figure~\ref{LIMIT} provides an illustration of this family of models. Let us also summarize, for this family of models, our conclusions in the previous section. 

\begin{proposition}
Consider the Euler equation \eqref{Euler-con}, take the pressure $p$ as a linear function of the density $\rho> 0$, that is, $p(\rho) = k^2 \rho$ where the sound speed $k$ is a positive constant. When $0 < k< 1/\eps$, \eqref{Euler-con} is strictly hyperbolic and genuinely nonlinear. When $k= 0$, it is non-strictly hyperbolic and linearly degenerate; when $k= {1\over \eps}$, it is strictly hyperbolic but linearly degenerate. 
\end{proposition}

\begin{figure}[htbp]
\label{LIMIT}
\centering
\epsfig{figure=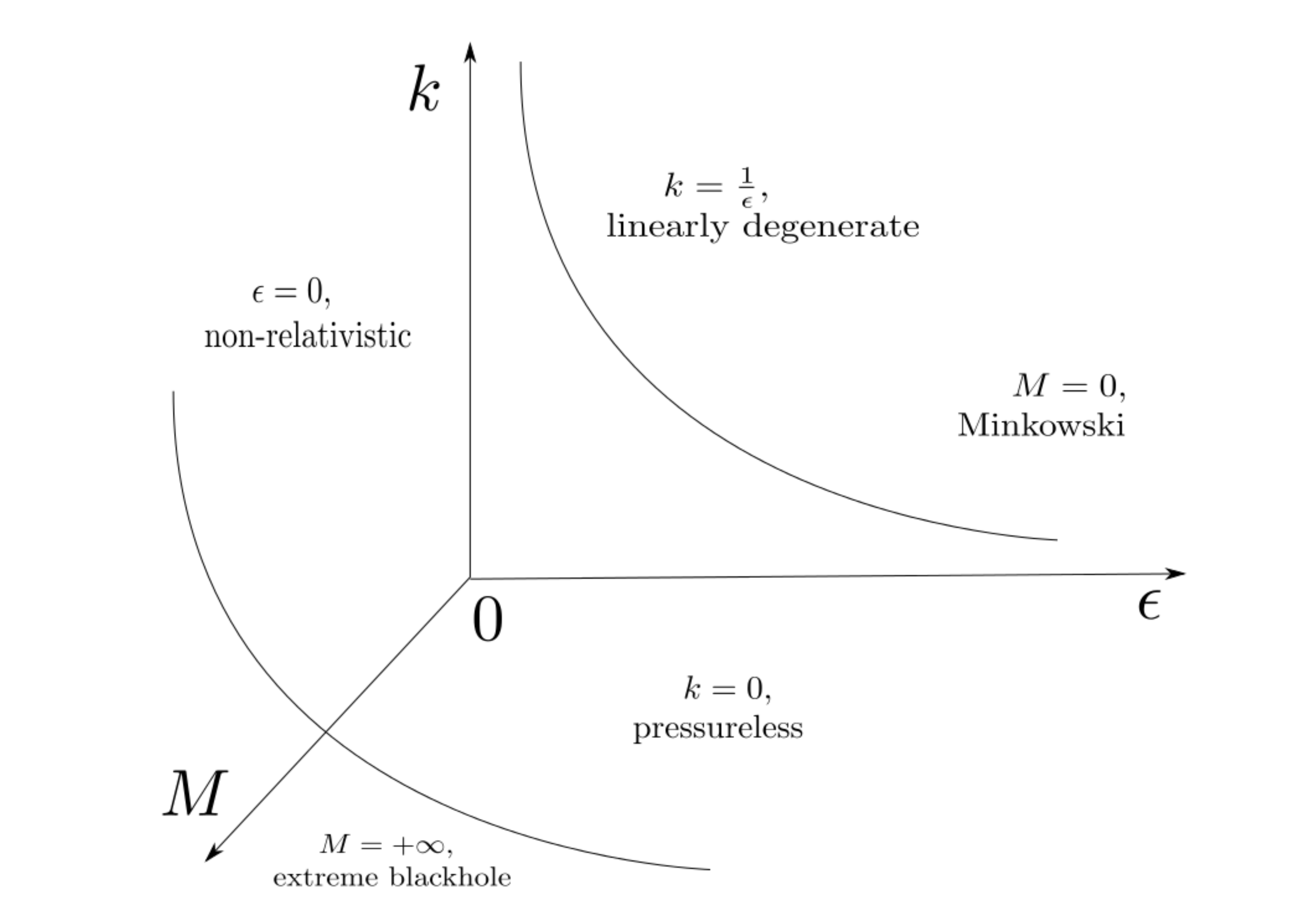,height=2.5in} 
\caption{Limit regimes of model $\Mscr (\eps, k, m)$.}
\end{figure}


\subsection*{Formal limits on the light speed and sound speed}

\subsubsection*{Non-relativistic fluid flows}

First of all, when $\eps \to 0$, the light speed goes to infinity and in order to avoid a blow-up of the source term,
$M {r-2M \over \eps^2 r} {1 + \eps^4 k^2 v^2 \over 1 - \eps^2 v^2} \rho$
in the `second' Euler equations in \eqref{Euler-conM}, we keep the ratio $m= {M\over \eps^2}$ constant. 
Letting $\eps \to 0$, we arrive at the {\bf Euler model for non-relativistic fluid flows on a Schwarzschild background}, denoted by $\Mscr (0, k, m)$: 
\bel{nonrelativ-con}
\aligned
&\del_t(r^2\rho) +\del_r(r^2\rho v) = 0,
\\
&\del_t(r^2\rho v) +\del_r\Big(r^2 \rho( v^2+k^2) \Big) -2k^2 r \rho + m \rho = 0, \qquad t \geq 0, \, r > 0.
\endaligned
\ee
Interestingly, this model applies to non-relativistic flows but yet contains a ``relaxation term'', that is $m \rho$, which is induced by the black hole geometry. 
Provided $k> 0$, this model is strictly hyperbolic (for $\rho> 0$ and $v \in \RR$) and admits two genuinely nonlinear characteristic fields. 
In Section~\ref{sec:4}, we will first study the family of steady state solutions and, for the Cauchy problem in Section~\ref{sec:92}, we will establish a global-in-time theory of weak solutions. 


\subsubsection*{Stiff fluid flows}

Returning to the regime of finite light speed, we now consider limiting regimes for the sound speed $k \in (0, 1/\eps)$. By definition, a stiff fluid is governed by the equation $p=\eps^{-2} \rho$ for which the sound speed coincides with the light speed. Letting therefore $k  \to 1/ \eps$, we define the {\bf Euler model for stiff fluid flows on a Schwarzschild background} $\Mscr (\eps,{1\over \eps}, m)$ as 
\bel{sonetlu}
\aligned 
& \del_t\Big( r^2 {1 + \eps^2 v^2 \over 1 - \eps^2 v^2} \rho\Big) +\del_r \Big(r(r-2\eps^2 m) {2 \rho  v \over 1 - \eps^2 v^2} \Big) = 0,
\\
& \del_t \Big(r(r-2\eps^2 m) {2 \rho v  \over 1 - \eps^2 v^2} \Big) + \del_r \Big((r-2\eps^2 m)^2 {1 + \eps^2 v^2 \over \eps^2(1-\eps^2 v^2)} \rho \Big)
 \\
 & =2\eps^2 m {r-2\eps^2 m \over \eps^2 r} {1 + \eps^2 v^2 \over 1 - \eps^2 v^2} \rho  +2{(r-2\eps^2 m)^2 \over \eps^2 r} \rho. 
\endaligned 
\ee
According to Proposition~\ref{nonlinearpo}, this model has two linearly degenerate characteristic fields. The Cauchy problem for this system will be studied in Section~\ref{sec:91}, below.


\subsubsection*{Pressureless fluid flows}

Letting now the sound speed $k  \to 0$, we obtain a regime where the pressure vanishes identically and we can introduce the 
{\bf Euler model of pressureless fluid flow} $\Mscr (\eps,0, m)$: 
\bel{sonnul}
\aligned 
& \del_t\Bigg( r^2 {\rho \over 1 - \eps^2 v^2} \Bigg) +\del_r \Bigg(r(r-2\eps^2 m) {\rho v  \over 1 - \eps^2 v^2} \Bigg) = 0,
\\
& \del_t \Bigg(r(r-2\eps^2 m) {\rho v \over 1 - \eps^2 v^2} \Bigg) + \del_r \Bigg((r-2\eps^2 m)^2 {\rho v^2 \over 1 - \eps^2 v^2} \Bigg)
  = {m \over r} (3\eps^2 -  1)(r-2\eps^2 m) {\rho v^2 \over 1 - \eps^2 v^2}. 
\endaligned 
\ee
Observe that this system is not hyperbolic, since it admits only one eigenvalue: $\lambda=\mu=\Big(1- {2M\over r} \Big)v$. 
Note also $p\equiv 0$ obviously satisfies \eqref{linearnly-degenerate}, so that \eqref{sonnul} admits one linearly degenerate characteristic field, while it can be checked that the other field is genuinely nonlinear. This model can not be handled by the techniques in the present paper, and we postpone its analysis to a follow-up work. 


\subsubsection*{Non-relativistic pressureless regime} 

In addition to having $k \to 0$, we can also take the limit  $\eps \to 0$ in \eqref{sonnul} and thus define the {\bf Euler model for pressureless non-relativistic flows} $\Mscr (0,0, m)$: 
\bel{nonrelative-san}
\aligned
&\del_t(r^2\rho) +\del_r(r^2\rho v) = 0,
\qquad\del_t(r^2\rho v) +\del_r\Big(r^2  \rho v^2\Big) +m \rho= 0.
\endaligned
\ee


\subsection*{Vanishing black hole mass}
 
\subsubsection*{Relativistic regime} 
 
When the black hole mass is taken to vanish, that is, $m \to 0$, the Schwarzschild metric approaches the Minkowski metric \eqref{Minkowski}, and we arrive at the {\bf Euler model for radially symmetric fluid flows in Minkowski space} denoted by $\Mscr (\eps,k, 0)$: 
\bel{EulerMI}
\aligned 
& \del_t\Bigg( {1 + \eps^4 k^2 v^2 \over 1 - \eps^2 v^2}\rho \Bigg) +\del_r \Bigg( {1 + \eps^2 k^2 \over 1 - \eps^2 v^2}\rho  v\Bigg) = 0,
\\
& \del_t \Bigg({1 + \eps^2 k^2  \over 1 - \eps^2 v^2}\rho v \Bigg) + \del_r \Bigg( {v^2+ k^2 \over 1 - \eps^2 v^2} \rho \Bigg) = 0.
\endaligned 
\ee 

 
\subsubsection*{Relativistic pressureless regime} 

If in addition we let the sound speed $k \to 0$ in \eqref{EulerMI}, we have the {\bf Euler model for radially symmetric, pressureless flows in Minkowski space} $\Mscr (\eps,0, 0)$: 
 \bel{EulerMI-k1}
\aligned 
& \del_t\Big( {\rho \over 1 - \eps^2 v^2} \Big) +\del_r \Big( {\rho  v \over 1 - \eps^2 v^2} \Big) = 0,
\qquad
 \del_t \Big({\rho  v \over 1 - \eps^2 v^2} \Big) + \del_r \Big( {\rho v^2 \over 1 - \eps^2 v^2}  \Big) = 0.
\endaligned 
\ee
Observe that, for sufficiently regular solutions, these equations are equivalent to  
$$
\aligned 
& \del_t v+ \del_r \Big( {v^2 \over 2}  \Big) = 0,
\qquad
 \del_t\Big( {\rho \over 1 - \eps^2 v^2} \Big) +\del_r \Big( {\rho  v \over 1 - \eps^2 v^2} \Big) = 0,
\endaligned
$$
from which we see that the velocity component satisfies Burgers' equation. 


\subsubsection*{Non-relativistic regime} 

Finally, letting both $m \to 0$ and $\eps \to 0$, we obtain the {\bf Euler model for radially symmetric, non-relativistic fluid flows} $\Mscr (0,k, 0)$: 
\bel{nonrelativ-con'}
\aligned
&\del_t(r^2\rho) +\del_r(r^2\rho v) = 0,
\qquad
\del_t(r^2\rho v) +\del_r\Big(r^2 \rho ( v^2+k^2) \Big) = 2k^2 \rho r, 
\endaligned
\ee 
and its pressureless version
\bel{nonrelativ-conpzero}
\aligned
&\del_t(r^2\rho) +\del_r(r^2\rho v) = 0,
\qquad \del_t(r^2\rho v) +\del_r\Big(r^2 \rho v^2 \Big) = 0. 
\endaligned
\ee 

\subsection*{Fluid flows in a black hole background with extreme mass}

Another limit of interest is obtained when $M\to +\infty$. In order to analyze this regime, we fix $\eps > 0$ and $k \in (0, 1/\eps)$ and we  define a rescaled variable $\tilder := {r\over 2M} \in (1, +\infty)$. We can rewrite \eqref{Euler-con} in the form
\bel{Euler-change}
\aligned 
&  \del_t\Bigg( \tilder^2 {1 + \eps^4 k^2 v^2 \over 1 - \eps^2 v^2} \rho \Bigg) + {1\over 2M} \del_ {\tilder} \Bigg(\tilder(\tilder-1) {1 + \eps^2 k^2 \over 1 - \eps^2 v^2} \rho v\Bigg) = 0,
\\
& \del_t \Bigg(\tilder(\tilder-1) {1 + \eps^2 k^2 \over 1 - \eps^2 v^2}\rho v \Bigg) + {1\over 2M} \del_ {\tilder} \Bigg((\tilder-1)^2 {v^2+ k^2 \over 1 - \eps^2 v^2} \rho\Bigg) = \Omegat,
\\
& \Omegat := {3 \over 4M} {\tilder-1 \over \tilder} {v^2+ k^2 \over 1 - \eps^2 v^2} \rho 
            - {1 \over 4M} {\tilder-1\over \eps^2 \tilder} {1 + \eps^4 k^2 v^2 \over 1 - \eps^2 v^2}\rho + {(\tilder-1)^2 \over M \tilder}k^2 \rho, 
\endaligned 
\ee 
and we now formally investigate the singular limit $M \to +\infty$. 

\begin{lemma}
\label{form-express}
For solutions to \eqref{Euler-change} expanded in the form (for $t \geq 0$ and $\tilder>1$)
$$
\aligned
& \rho(t, \tilder) =  \sum\limits_{j= 0}^{+\infty} {1\over M^j}\rho^{(j)}(t,\tilder),
\qquad 
v(t, \tilder) = \sum\limits_{j= 0}^{+\infty} {1\over M^j}v^{(j)}(t,\tilder),
\endaligned
$$
it follows that the functions $\rho^{(0)}, v^{(0)}$ must be independent of the time variable $t$, while 
$\rho^{(j)}, v^{(j)}$ satisfy a coupled system of ordinary differential equations in the time variable: 
\bel{ODE-j}
\aligned
\del_t \rho^{(j)} (t, \cdot) & = \sum_{i = 0}^{j-1}  \Big( A_3^{j,i} \rho^{(i)} (t, \cdot) +A_4^{j,i} v^{(i)}(t, \cdot) \Big),
\\
  \del_t v^{(j)}(t, \cdot) & = \sum_{i = 0}^{j-1} \Big( B_3^{j, i} \rho^{(i)}(t, \cdot) +B_4^{j,i} v^{(i)}(t, \cdot) \Big),
\endaligned
\ee
in which the coefficients are constants depending upon $\eps$ and $k$ only. 
\end{lemma}

\begin{proof} Keeping only the terms of zero-order in ${1\over M}$, we easily find the ordinary differential system
\bel{1/M0}
\aligned 
&  \del_t\Bigg( {1 + \eps^4 k^2 (v^{(0)})^2 \over 1 - \eps^2 (v^{(0)})^2} \rho^{(0)} \Bigg) = 0,
\qquad
 \del_t \Bigg( {1 + \eps^2 k^2 \over 1 - \eps^2 (v^{(0)})^2}\rho^{(0)}v^{(0)} \Bigg) = 0.
\endaligned 
\ee
which is equivalent to saying that $\del_t\rho^{(0)} =  \del_t v^{(0)} = 0$, so that $\rho^{(0)} = \rho^{(0)}(\tilder)$
and $v^{(0)} = v^{(0)}(\tilder)$ depend on the spatial variable only. Next, keeping the terms of the first-order in ${1\over M}$, we find the following system of equations: 
\bel{1/M1}
\aligned 
&   \tilder^2  \del_t\Bigg( \rho^{(1)} {1 + \eps^4 k^2 (v^{(0)})^2 \over 1 - \eps^2 (v^{(0)})^2} +\rho^{(0)} v^{(0)} v^{(1)} \Big({2 \eps^4 k^2  \over 1 - \eps^2 (v^{(0)})^2}+2\eps(1 + \eps^4 k^2  v^{(0)^2}) \Big) \Bigg)
\\
& \hskip5.cm + {1\over 2} \del_ {\tilder} \Big( \tilder(\tilder-1) {1 + \eps^2 k^2 \over 1 - \eps^2 v^{(0)^2}} \rho^{(0)} v^{(0)} \Big) = 0,
\\
& \tilder(\tilder-1) \del_t \Big( {1 + \eps^2 k^2 \over 1 - \eps^2 v^{(0)^2}}(\rho^{(1)} v^{(0)}+ \rho^{(0)} v^{(1)}) +2\eps^2(1 + \eps^2 k^2) \rho^{(0)}v^{(0)^2} v^{(1)} \Big)
\\
&  \hskip5.cm + {1\over 2} \del_ {\tilder} \Big((\tilder-1)^2 {v^{(0)^2}+ k^2 \over 1 - \eps^2 v^{(0)^2}} \rho^{(0)} \Big)
\\
& = {3 \over 4} {\tilder-1 \over \tilder} {v^{(0)^2}+ k^2 \over 1 - \eps^2 v^{(0)^2}} \rho^{(0)} 
 - {1 \over 4} {\tilder-1\over \eps^2 \tilder} {1 + \eps^4 k^2 v^{(0)^2} \over 1 - \eps^2 v^{(0)^2}}\rho^{(0)} + {(\tilder-1)^2 \over  \tilder}k^2 \rho^{(0)}. 
\endaligned 
\ee
The functions $\rho^{(0)}, v^{(0)}$ being already fixed, we see that \eqref{1/M1} is a differential system in the time variable $t$, which has the general form (higher-order terms $\rho^{(j)}, v^{(j)}$ (with $j>2$) being determined similarly):   
$$
\aligned
A_1^j \del_t \rho^{(j)} (t, \cdot) +A_2^j  \del_t v^{(j)}(t, \cdot) & = \sum_{i = 0}^{j-1}  \Big( A_3^{j,i} \rho^{(i)} (t, \cdot) +A_4^{j,i} v^{(i)}(t, \cdot) \Big),
\\
B_1^j \del_t \rho^{(j)} (t, \cdot) +B_2^j  \del_t v^{(j)}(t, \cdot) & = \sum_{i = 0}^{j-1} \Big( B_3^{j, i} \rho^{(i)}(t, \cdot) +B_4^{j,i} v^{(i)}(t, \cdot) \Big),
\endaligned
$$
in which the coefficients are constants depending upon $\eps$ and $k$ only. This system is non-degenerate in the sense that it can be expressed as an ordinary differential system in $t$ for the functions $\del_t \rho^{(j)} (t, \cdot)$ and $\del_t v^{(j)} (t, \cdot)$. Changing the notation, we thus arrive at \eqref{ODE-j}. 
\end{proof}

In view of Lemma~\ref{form-express}, in the extreme mass regime $M \to +\infty$, the leading-order behavior of solutions only depends on the space variable $\tilder$, that is, 
$$
\rho(t, \tilder) = \rho^{(0)}( \tilder), \qquad v(t, \tilder) = v^{(0)}( \tilder). 
$$
Proceeding at a formal level, the following result is now immediate. It would be interesting to rigorously justify the expansion below, but this is outside the scope of the present paper. 

\begin{proposition}[Asymptotic solutions for black holes with extreme mass] 
Consider the Euler model \eqref{Euler-change} with initial data prescribed at $t=0$:
$$ 
\rho(0, \tilder) = \rho^{(0)}( \tilder), \qquad v(0, \tilder) = v^{(0)}( \tilder), \qquad \tilde r > 0.
$$
\begin{enumerate}
\item If the data $\rho^{(0)}, v^{(0)}$ belong to $C^l$ for some $l \geq 1$, then there exists an approximate solution, i.e. 
$$
\aligned
&\rhot(t, \tilder) := \rho^{(0)}( \tilder) +  \sum\limits_{j=1}^ l {1\over M^j}\rho^{(j)}(t,\tilder),
\qquad
 \velt(t, \tilder) := v^{(0)}( \tilder) + \sum\limits_{j=1}^ l {1\over M^j}v^{(j)}(t,\tilder), 
\endaligned
$$ 
which satisfies \eqref{Euler-change} up to an error 
${\mathcal O}(1/M^{l+1})$.  

\item If $\rho^{(0)}, v^{(0)}$ has $C^\infty$ regularity, then exists a formal series defined at all order. 

\end{enumerate}
\end{proposition}


\section{Non-relativistic equilibria on a Schwarzschild background} 
\label{sec:4}

We now turn our attention to the main model of interest in this section, that is, the Euler model for non-relativistic flows on a Schwarzschild background \eqref{nonrelativ-con}, which we have denoted by $\Mscr (0, k, m)$. We begin by considering general pressure-laws, that is, 
\bel{nonrelativ}
\aligned
&\del_t(r^2\rho) +\del_r(r^2\rho v) = 0,
\\
&\del_t(r^2\rho v) +\del_r\Big(r^2(\rho v^2+p) \Big) -2pr+m\rho= 0, 
\endaligned
\ee
for solutions defined on $r \in (0,+\infty)$. We search for steady state solutions $\rho=\rho(r)$ and $v=v(r)$, which satisfy the differential system:
\bel{steady1}
\aligned 
& {d\over dr}(r^2\rho v) = 0, 
\\
&{d\over dr} \Big(r^2(\rho v^2+p) \Big) -2pr+m\rho = 0
\endaligned
\ee
with initial condition $\rho_0, v_0 > 0$ prescribed at some given radius $r=r_0 > 0$,
\bel{eq:JS98}
\rho(r_0) =\rho_0 > 0, \qquad v(r_0) =v_0.
\ee
It is straigthforward to check the following statement. 

\begin{lemma}
\label{Twoformula}
All solutions \eqref{steady1} --\eqref{eq:JS98} satisfy 
\bel{eq:308}
r^2\rho(r) v(r) =r_0^2\rho_0v_0,
\ee
\bel{eq:309}
{1\over2}v(r)^2+h(\rho(r)) -m{1\over r} = {1\over 2}v_0^2+h(\rho_0) -m {1\over{r_0}}, 
\ee
where  $h(\rho) :=\int^\rho {p'(s) \over s} \, ds$.
\end{lemma}

In view of \eqref{eq:308}, we see that the solution $v$ has the sign of the initial condition $v_0$, and without loss of generality, we now assume that $v_0 \geq 0$. We are especially interested in a constant sound speed, that is, $p= k^2 \rho$ with $k> 0$, hence 
\bel{steady1-con}
\aligned
& {d\over dr}(r^2\rho v) = 0,
 \\
 &{d\over dr} \Big(r^2 \rho (v^2+k^2) \Big) -2k^2 r \rho +m\rho= 0. 
\endaligned
\ee
According to Lemma~\ref{Twoformula}, we must solve the system 
\bel{eq:ksKEY}
\aligned 
& r^2\rho v=r_0^2\rho_0v_0,
\\
& {1\over2}v^2+k^2 \ln \rho -m{1\over r} = {1\over 2}v_0^2+k^2 \ln \rho_0-m {1\over{r_0}}.
\endaligned 
\ee
After eliminating $\rho$, we find an algebraic equation for the velocity, i.e. 
\bel{eq:306}
{1\over2}v^2+k^2 \ln {r_0^2 v_0 \over r^2 v} -m{1\over r} = {1\over 2}v_0^2-m {1\over{r_0}} 
\ee
and we now focus on this equation. 

Let us introduce the function  
\bel{Fon-G}
G(r,v;r_0,v_0):= {1 \over 2} \big(v^2 - v_0^2 \big) + k^2 \ln {r_0^2 v_0 \over r^2 v} - {m\over r}+ {m \over{r_0}}.
\ee
By definition, if $v=v(r)$ is a steady state solution, then $G(r,v(r);r_0,v_0) \equiv 0$ and, in addition, $v(r_0) =v_0$. Clearly, we have $G(r_0,v_0;r_0,v_0) = 0$. Differentiating $G$ with respect to $v$ and $r$, we obtain  
$$
\del_v G(r,v;r_0,v_0) =v- {k^2\over v}, \qquad
\del_r G(r,v;r_0,v_0) = {1\over r^2} (m-2k^2 r). 
$$
Hence, the function $G$ is decreasing with respect to $v$ when $v< k$, and is increasing when $v> k$ (that is, a non-sonic velocity). Also, the derivative of a solution $v=v(r)$ is found to be  
\bel{derive-one}
{dv\over  dr} = {v\over  r^2} {2k^2 r -m \over v^2-k^2}. 
\ee

\begin{figure}[htbp]
\centering
\epsfig{figure=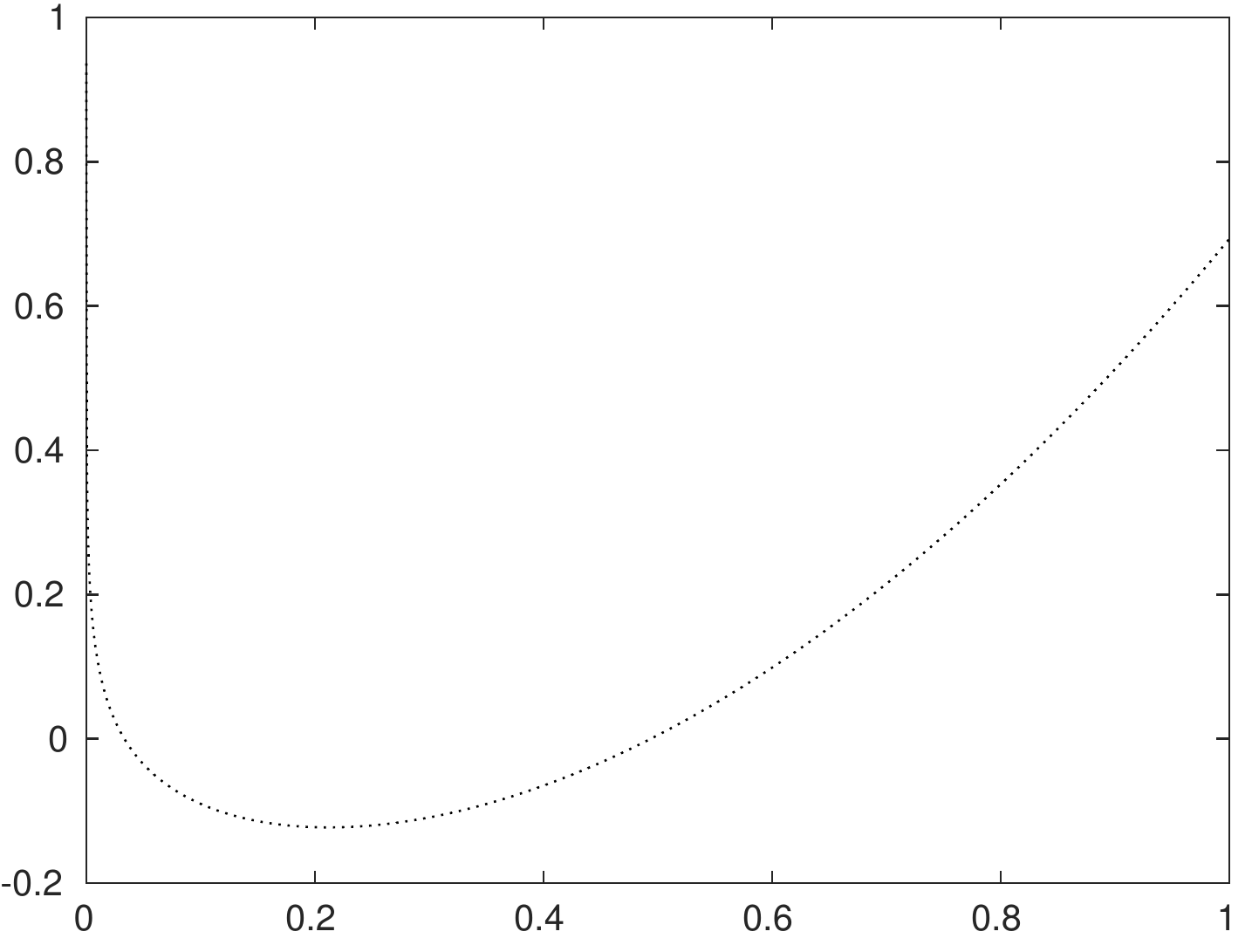,height=2.in} 
\caption{Plot of the map $v \mapsto G(v) =G(r,v; r_0, v_0)$.}
\end{figure}

Since $\del_v G(r_0,v_0;r_0,v_0) \neq 0 $ when $v_0 \neq k$, it is immediate to apply the implicit function theorem for the function $v=v(r)$ and then recover the density $\rho=\rho(r)$ by \eqref{eq:ksKEY}. We thus have the following local existence statement. 

\begin{lemma}[Locally-defined steady state solutions]
\label{localsol}
Given any values $r_0 > 0$, $\rho_0 > 0$, $v_0 \geq 0$ with $v_0 \neq k$, the system \eqref {steady1-con} with initial condition \eqref{eq:JS98} at $r=r_0$ admits a unique smooth solution $\rho=\rho(r)$ and $v=v(r)$ defined in a neighborhood $\Uscr_0$ of $r_0$ and denoted by 
$$
\rho=\rho(r;r_0,\rho_0,v_0), \qquad v=v(r;r_0,\rho_0,v_0). 
$$
\end{lemma}

According to \eqref{derive-one}, the derivative of a solution $v=v(r)$ may blow up if at some radius $r_*$ the velocity $v(r_*) = k$ reaches the sonic value. We will use the following notation. 

\begin{definition} A radius $r_*> 0$ is called a {\bf sonic point} if it is a root of the equation 
\bel{eq:319}
{1 \over 2} \big(k^2 - v_0^2 \big) + k^2 \ln {v_0 \over k} - k^2 \ln {r^2 \over r_0^2} - {m\over r} + {m \over r_0} = 0.
\ee
\end{definition}

If such a radius $r_*$ exists, then the derivative ${dv\over dr} $ tends to infinity when $r \to r_*$ and the velocity loses its regularity. 

\begin{lemma}
\label{sonic -point} 
One can distinguish between two alternatives: 
\begin{enumerate}

\item Either $ {3\over 2}+ \ln {m^2 \over 4 k^3 r_0^2 v_0}+ {1 \over 2 k^2}(v_0^2- {2 m \over r_0}) > 0$ and there exists no sonic point.

\item Or $ {3\over 2}+ \ln {m^2 \over 4 k^3 r_0^2 v_0}+ {1 \over 2 k^2}(v_0^2- {2 m \over r_0}) \leq 0$, there exist two (possibly coinciding) sonic points, denoted by $\underr_*\leq \barr_* $. Moreover, in this case, one has: 

\bei 

\item When $r_0 \geq {m\over 2k^2}$, the roots satisfy $\underr_*\leq \barr_*\leq r_0$.

 \item When $r_0 <{m\over 2k^2}$, the rots satisfy $r_0 <\underr_*\leq \barr_* $. 

\eei 

\end{enumerate}
\end{lemma}

\begin{proof} We introduce the functions 
$f(x): = {1\over 2} - {1\over 2}x^2 + \ln x$ and $g(r): = \ln {r^2 \over r_0^2}+ {m\over k^2}({1\over r} - {1\over r_0})$, 
so that  a sonic point $r_*$ is characterized by the condition $f({v_0 \over k}) =g(r_*)$. Since $f'(x) = -x + {1\over x}$, we see that $f$ reaches its maximum at $x=1$, with $f(1) = 0$. Since we assume $v_0 \neq k$, we have $-\infty < f({v_0 \over k}) < 0$. 
Turning our atention to the function $g=g(r)$, we have $g'(r) = {1\over r^2}(2r- {m\over k^2})$. Therefore, the minimum of $g=g(r)$ is obtained at $r= {m\over 2 k^2}$, with $g({m\over 2 k^2}) =  \ln {m^2 \over 4 k^4 r_0^2}+2- {m \over k^2r_0}$. We now set 
$\tilde g(x): =  \ln {x^2\over4}+2- x$. 
According to our definition, $\tilde g({m\over r_0 k^2}) = g({m\over 2 k^2})$. We have $\tilde g'(x) = {2\over x} -1$, so that $-\infty < \tilde g(x) \leq \tilde g(2) = 0$. Therefore, \eqref{eq:319} admits no solution if and only if $f({v_0 \over k}) < g({m\over 2 k^2})$. This yields us the condition 
$
{1\over 2} - {1\over 2} \Big({v_0 \over k} \Big)^2 + \ln {v_0 \over k} 
<  
\ln {m^2 \over 4 k^4 r_0^2}+2- {m \over k^2r_0} 
$, 
as announced. 
If the opposite inequality holds, then, \eqref{eq:319} admits two solutions (which may coincide). Furthermore, since $g(r_0) = 0$, we have either $\underr_*\leq \barr_*\leq r_0$ or $ r_0 \leq \underr_*\leq \barr_*$. If $\underr_*\leq \barr_*\leq r_0$, we must have $g'(r_0) \geq 0$, which gives $r_0 \geq {m\over 2 k^2}$. 
\end{proof}

We now define a function $P$ which only depends upon the initial radius $r_0$ and the initial velocity $v_0$:
\bel{function-P}
P(r_0, v_0) : = {3\over 2}+ \ln {m^2 \over 4 k^3 r_0^2 v_0}+ {1 \over 2 k^2} \Big(v_0^2- {2 m \over r_0} \Big). 
\ee
According to Lemma~\ref{sonic -point}, the existence/non-existence of sonic points is determined by the sign of $P$. 
We will now distinguish between several cases and introduce a general notation: 
\be
\aligned 
& \mathfrak A:  P(r_0, v_0) > 0, \quad &&\mathfrak B:   P(r_0, v_0) \leq 0,
\\
& \mathfrak 1: v_0 < k, \quad &&\mathfrak 2: v_0 > k,
\\
& \mathfrak i: r_0 \geq {m \over 2k^2}, \quad &&\mathfrak {ii} : r_0 < {m \over 2k^2}. 
\endaligned 
\ee 
Hence, the symbol $\mathfrak {A1}$ refers to the case where both conditions 
${3\over 2}+ \ln {m^2 \over 4 k^3 r_0^2 v_0}+ {1 \over 2 k^2}(v_0^2- {2 m \over r_0}) > 0$ and $v_0 < k$ hold.

\begin{lemma}[Extension of solutions without sonic point]
\label{extension1}
Consider the local solution $\rho=\rho(r;r_0,\rho_0,v_0)$ and $v=v(r;r_0,\rho_0,v_0)$ given by Lemma~\ref{localsol}. 
\begin{enumerate}

\item {\bf Case} $\mathfrak {A1}$:. The solution can be extended tothe whole domain $(0,+\infty)$ and globally satisfies $v< k$, with 
$$ 
\lim_{r \to 0}v(r;r_0,\rho_0,v_0) =\lim_{r \to +\infty}v(r;r_0,\rho_0,v_0) = 0.
$$
One has the following monotonicity property: the velocity $v$ is increasing with respect to $r$ on the interval $(0, {m \over 2k^2})$, while
it is decreasing on the interval $({m \over 2k^2}, +\infty)$.

\item {\bf Case} $\mathfrak {A2}$. The solution can be extended to the whole domain $(0,+\infty)$ and globally satisfies $v> k$, with 
$$ 
\lim_{r \to 0}v(r;r_0,\rho_0,v_0) =\lim_{r \to +\infty}v(r;r_0,\rho_0,v_0) =+\infty. 
$$
One has the following monotonicity property: the velocity $v$ is decreasing with respect to $r$ on the interval $(0, {m \over 2k^2})$, while
it is increasing on the interval $( {m \over 2k^2}, +\infty)$. 
\end{enumerate}
\end{lemma}

\begin{proof} The two cases are completely similar and we treat the case $\mathfrak {A1}$. Since we have sonic point, the velocity $v=v(r)$ never reaches the sound speed $k$ and, by the implicit function theorem, the solution can be continued and extended to the whole interval $(0,\infty)$. Its derivative, given by \eqref{derive-one}, remains finite. From the definition of the function $G$ in \eqref{Fon-G}, we obtain 
$$
{1 \over 2} \big(v^2-v_0^2 \big) + k^2 \ln {v_0 \over v} = k^2 \ln {r^2 \over r_0^2}+ {m\over r} - {m \over{r_0}}.
$$
When $r  \to 0 $ or $r \to  +\infty $, the left-hand side of this identity goes to infinity. Such a limit is reached if and only if $v$ goes to $0$ or infinity. Since $v< k$ always holds in this case, we obtain the asymptotic behavior limits, as stated in the lemma. 
Furthermore, the expression \eqref{derive-one} of ${dv \over d r}$ determines the monotonicity properties: ${dv \over d r}$ has the sign of
${m \over 2k^2} - r$. 
 \end{proof}

\begin{lemma}[Extension of solutions with sonic points]
\label{extension2}
Consider the local solutions $\rho=\rho(r;r_0,\rho_0,v_0)$ and $v=v(r;r_0,\rho_0,v_0)$ given by Lemma~\ref{localsol}. 
\begin{enumerate}

\item {\bf Case} $\mathfrak {B1i}$. The solution can be extended to the interval $(\barr_*,+\infty)$ and satisfies $v\leq k$, with 
$$ 
\lim_{r \to +\infty}v(r;r_0,\rho_0,v_0) = 0, \qquad \lim_{r \to \barr_*}v(r;r_0,\rho_0,v_0) = k. 
$$
Moreover, $v$ is decreasing with respect to $r$ on $( \barr_*, +\infty)$. 

\item {\bf Case} $\mathfrak {B2i}$. The solution can be extended to the interval $(\barr_*,+\infty)$ and satisfies $v\geq k$, with 
$$ 
\lim_{r \to +\infty}v(r;r_0,\rho_0,v_0) =+\infty, \qquad \lim_{r \to \barr_*}v(r;r_0,\rho_0,v_0) = k. 
$$
Moreover, $v$ is increasing with respect to $r$ on $( \barr_*, +\infty)$.

\item {\bf Case} $\mathfrak {B1ii}$. The solution can be extended to the interval $(0,\underr_*)$ and satisfies $v\leq k$, with 
$$ 
\lim_{r \to 0}v(r;r_0,\rho_0,v_0) = 0, \qquad \lim_{r \to \underr_*} v(r;r_0,\rho_0,v_0) = k. 
$$
Moreover, $v$ is increasing with respect to $r$ on $( 0, \underr_*)$. 

\item {\bf Case} $\mathfrak {B2ii}$. The solution can be extended to the interval $(0,\underr_*)$ and satisfies $ v \geq k$, with 
$$ 
\lim_{r \to 0}v(r;r_0,\rho_0,v_0) =+\infty, \qquad \lim_{r \to \underr_*}v(r;r_0,\rho_0,v_0) = k. 
$$
Moreover, $v$ is decreasing with respect to $r$ on $( 0, \underr_*)$. 
\end{enumerate}
\end{lemma}

\begin{proof} Consider the case $\mathfrak {B1i}$ (while the case $\mathfrak {B2i}$ is completely similar). According to Lemma~\ref{sonic -point}, there exist two sonic points  $\underr_*\leq \barr_*\leq r_0$, so that by continuation the solution can be extended to the whole interval  $(\barr_*,+\infty)$ and the limits ${r \to +\infty}$ and ${r \to \underr_*}$ are easily computed. Moreover, since in this case $r\geq \barr_* \geq {m \over 2k^2}$ and $v\leq k$, the function $v=v(r)$ is decreasing in $r$. 
 
In the case $\mathfrak {B1ii}$ (while the case $\mathfrak {B2ii}$ can be treated similarly), Lemma~\ref{sonic -point} shows that there exist two sonic points $r_0 <\underr_* < \barr_* $. In this case, the solution can be extended to $(0, \underr_*)$ and the limits 
${r \to 0}$ and ${r \to \underr_*}$ are easily computed. The condition $r <\barr_* < {m \over 2k^2}$ gives the monotonicity property. 
\end{proof}

Observe also that no solution can be defined on the interval $r \in ( \underr_*, \barr_*)$. Indeed, since $G$ reaches its minimum at $v= k$, we deduce that, for any radius $r \in ( \underr_*, \barr_*)$, the inequality  
$$
G(r,v, r_0,v_0) >G(r,k, r_0,v_0) > 0
$$
holds, that is, $G$ cannot admit roots between the two sonic points. Therefore, a solution cannot be further extended when it reaches a sonic point. We summarize our conclusions in this section in the following theorem.

\begin{theorem}[Non-relativistic steady flows on a Schwarzschild background]
\label{steady-non}
For any sound speed $k> 0$ and black hole mass $m> 0$, consider the Euler model $\Mscr (0, k, m)$ given in \eqref{nonrelativ-con}, describing non-relativistic flows on a Schwarzschild background. 
Then, given any radius $r_0 > 0$, density $\rho_0 > 0$, and velocity $v_0 \geq 0$, there exists a unique steady state solution 
denoted by 
$$
\rho=\rho(r;r_0,\rho_0,v_0), \qquad v=v(r;r_0,\rho_0,v_0), 
$$ 
satisfying the system \eqref {steady1-con}  together with the initial condition 
$\rho(r_0) =\rho_0$ and $v(r_0) =v_0$. Moreover, the velocity component satisfies $\sgn(v(r) -k) = \sgn(v_0-k)$ for all relevant values $r$, and in order to specify the range of the independent variable $r$ where this solution is defined, we distinguish between two alternatives: 
\begin{enumerate}

\item {\bf Regime without sonic point: $P(r_0, v_0) > 0$} (with $P$ defined in \eqref{function-P}). Then, the solution is defined on the whole interval $(0,+\infty)$.

\item{\bf Regime with sonic points: $P(r_0,v_0) \leq 0$.} The solution is defined on the interval ${\Xi  \varsubsetneqq (0,+\infty)}$, defined by
\bel{Xi}
\Xi :=
\begin{cases}
(0, \underr_*), \qquad & r_0 \leq {m \over 2k^2},
\\
(\barr_*,+\infty),               & r_0 > {m \over 2k^2}.
 \end{cases}
\ee 
Moreover, the velocity $v(r) \to k$ when $r$ approaches the sonic point. 
\end{enumerate}
\end {theorem}

These solutions will be used to design a method of approximation o general weak solutions to the Cauchy problem. In fact (cf.~Section~\ref{sec:7}), we will need to introduce {\sl discontinuous} solutions in order to construct globally-defined steady state solution (defined for all $r$). This will be achieved with solutions containing a jump discontinuity connecting two smooth steady state solutions.

\begin{figure}[htbp]
\begin{minipage}[t]{0.3\linewidth}
\centering
\epsfig{figure=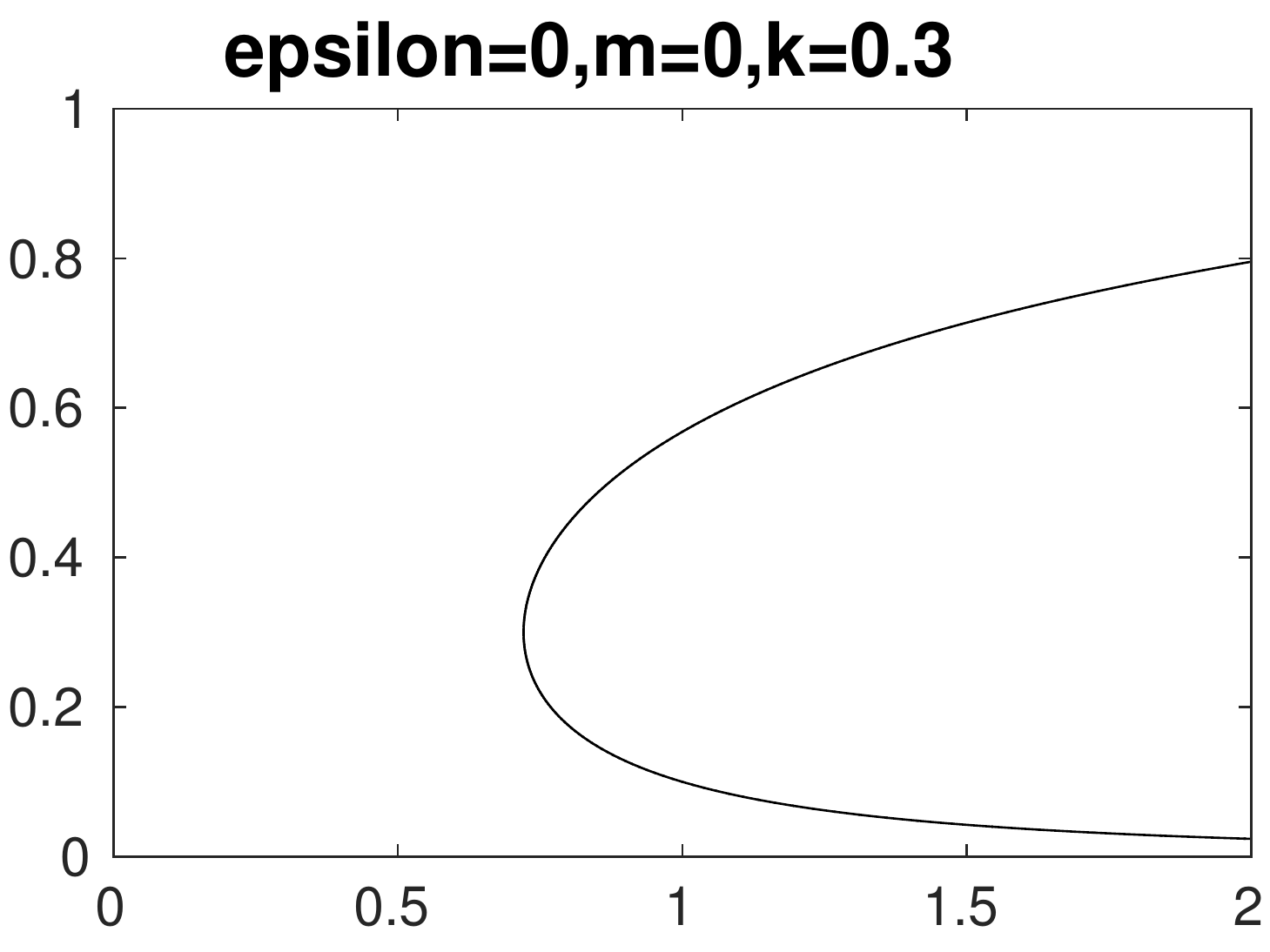,height=1.3in} 
\end{minipage}
\begin{minipage}[t]{0.3\linewidth}
\centering

\epsfig{figure=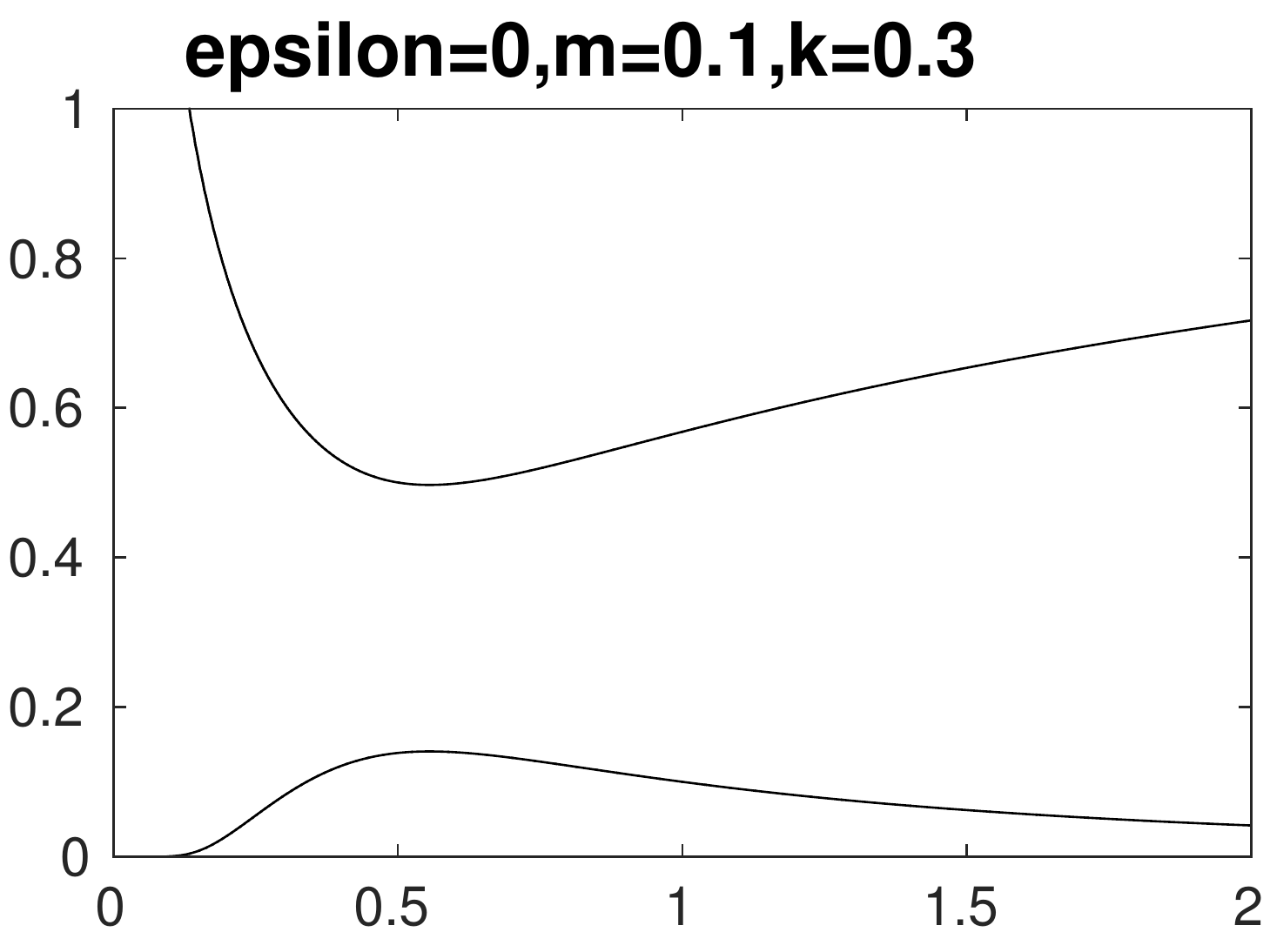,height=1.3in} 
\end{minipage}
\begin{minipage}[t]{0.3\linewidth}
\centering
\epsfig{figure=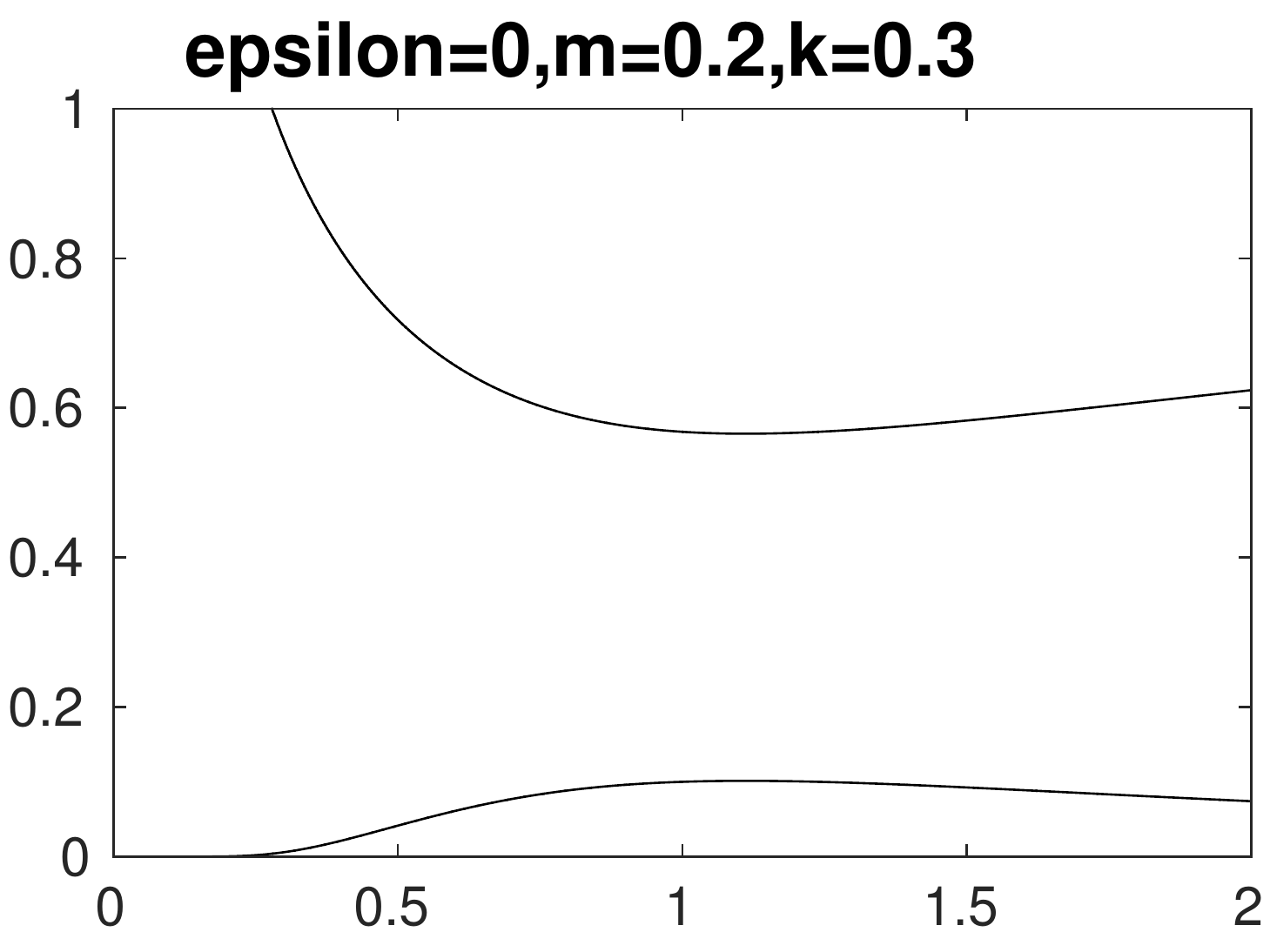,height=1.3in} 
\end{minipage}
\caption{Plots of $v=v(r)$ with sound speed $k= 0.3$ and different masses.}
\end{figure} 

\begin{figure}[htbp]
\begin{minipage}[t]{0.3\linewidth}
\centering
\epsfig{figure=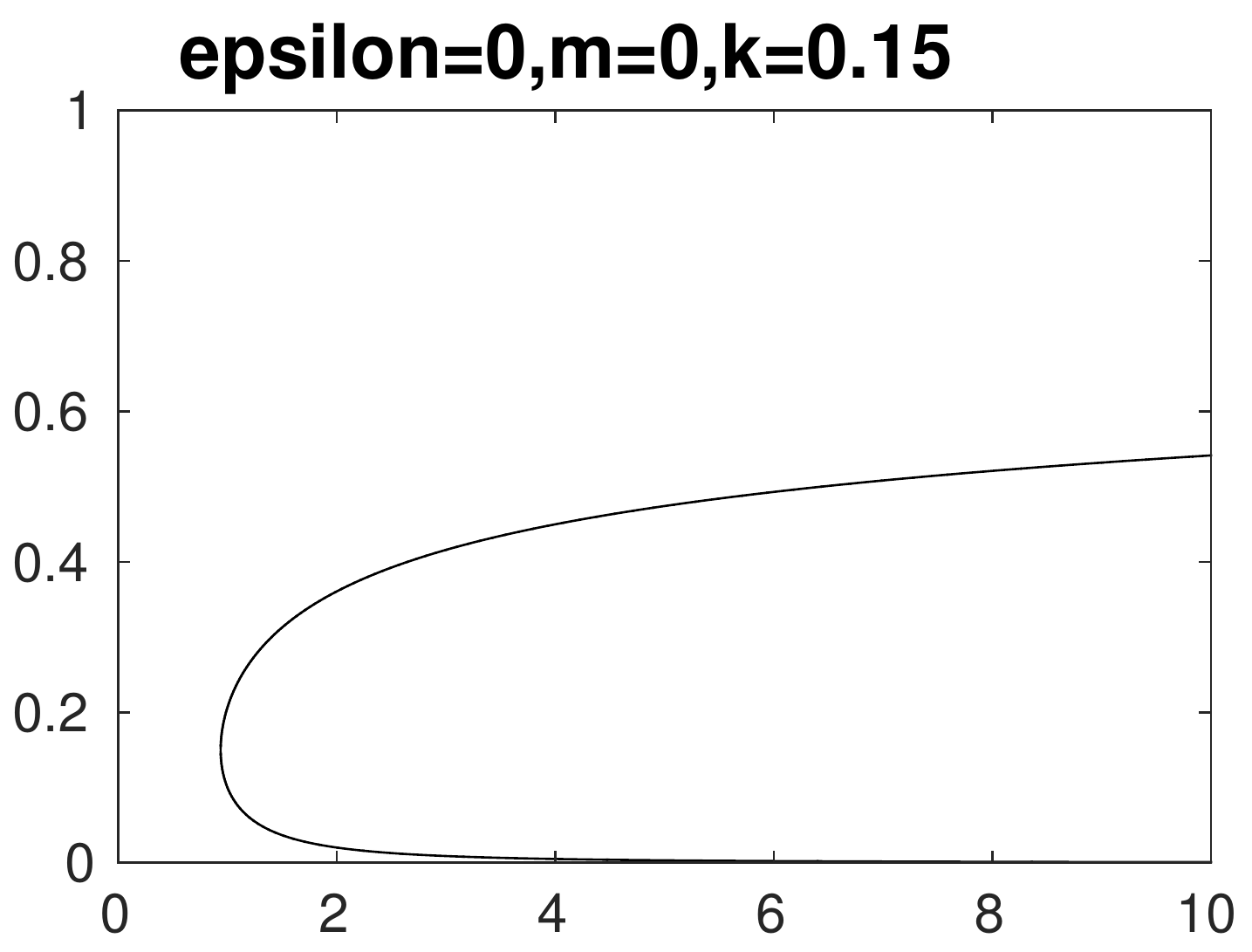,height=1.3in} 
\end{minipage}
\begin{minipage}[t]{0.3\linewidth}
\centering

\epsfig{figure=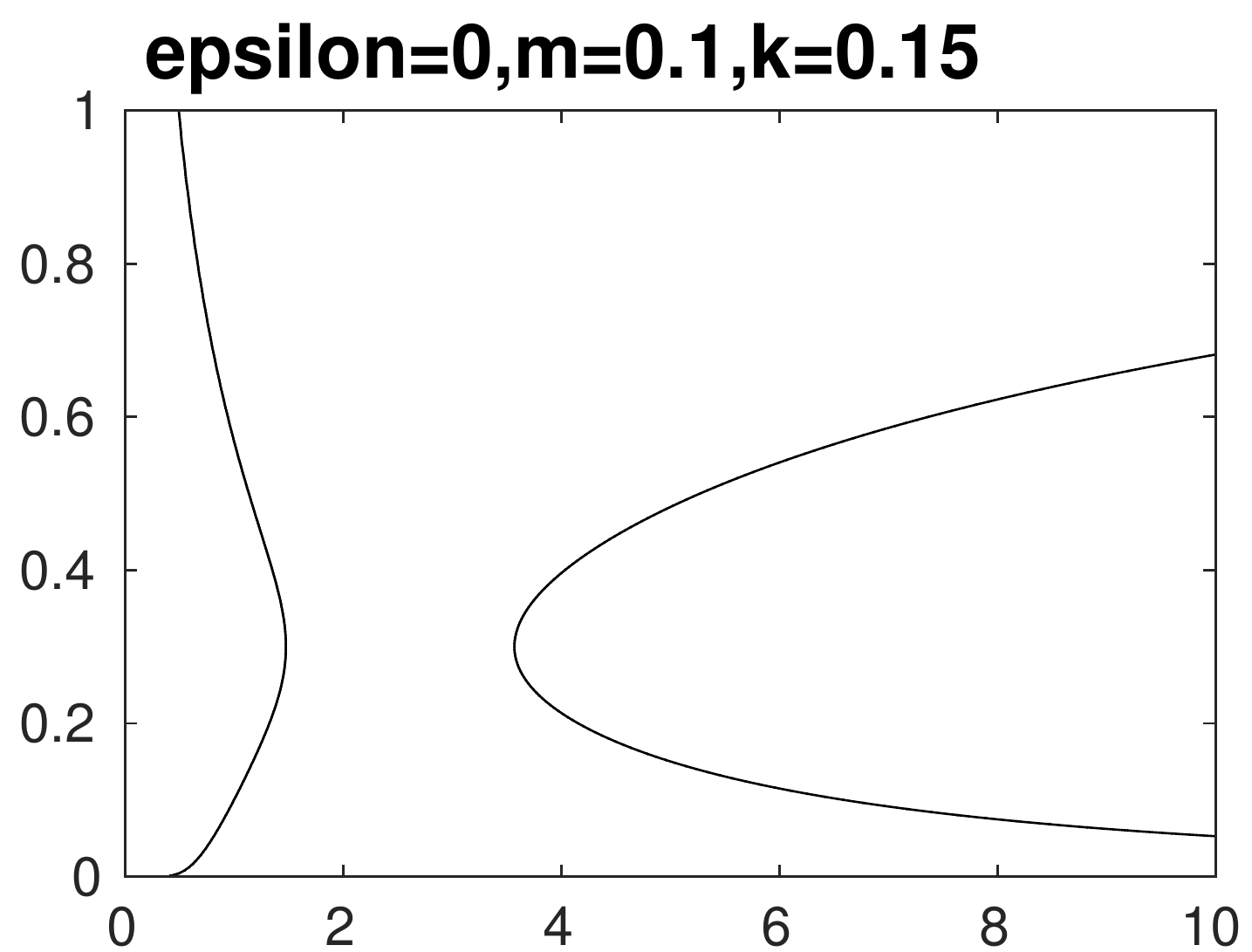,height=1.3in} 
\end{minipage}
\begin{minipage}[t]{0.3\linewidth}
\centering
\epsfig{figure=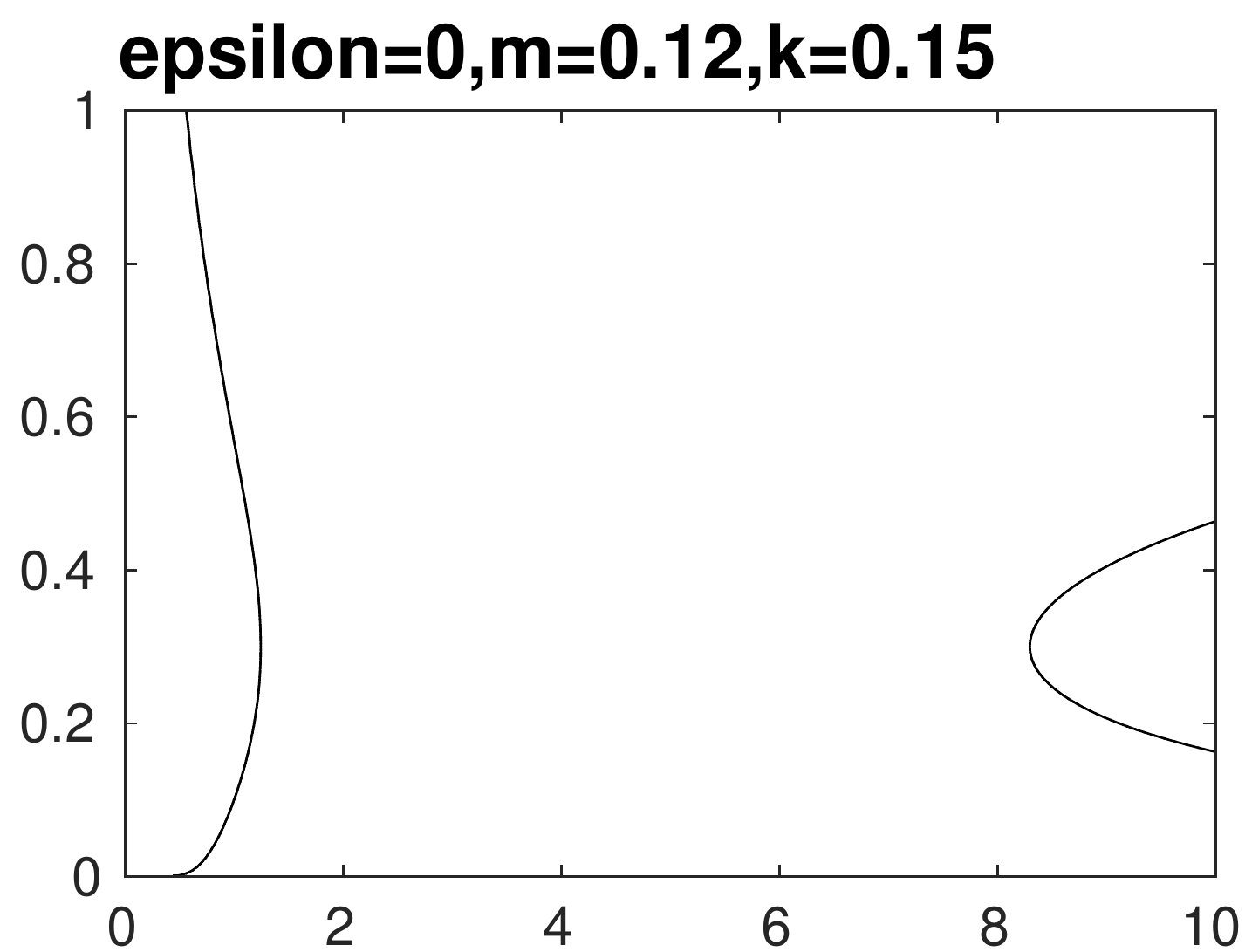,height=1.3in} 
\end{minipage}

\caption{Plots of $v=v(r)$ with sound speed $k= 0.15$ and different masses.}
\end{figure}

\section{Fluid equilibria on a Schwarzschild background}
\label{sec:5} 

\subsection*{Local existence result}
\label{sec:5-1}

This section is devoted to the analysis of (smooth) steady state solutions to the Euler system on a Schwarz\-schild background, i.e. the general model \eqref{Euler}. Such solutions must satisfy the following two coupled ordinary differential equations with unknowns $\rho=\rho(r)$ and $v=v(r)$ (defined over $r >2M$) 
\begin{subequations}
\label{steady2}
\begin{eqnarray}
 & {d\over dr} \Big(r(r-2M) {(\rho+\eps^2 p)v \over 1 - \eps^2 v^2} \Big) = 0 \label{steady2a},
 \\
& {d\over dr} \Big((r-2M)^2 {\rho v^2+ p \over 1 - \eps^2 v^2} \Big)
 = {M \over r} 
  {(r-2M)  \over 1 - \eps^2 v^2}  \Big(3 \rho v^2+ 3 p - \eps^{-2} \rho- \eps^2 p v^2\Big)  
  + {2 \over r} (r-2M)^2  \, p, 
 \label{steady2b}
\end{eqnarray}
\end{subequations}
formulated here for a general pressure-law $p=p(\rho)$. We are interested in solving the associated initial value problem for a given radius $r_0 >2M$ with data $\rho_0, v_0$ prescribed at $r=r_0$:
\bel{eq:837}
\rho(r_0) =\rho_0 > 0, \qquad v(r_0) =v_0.
\ee

\begin{lemma}
\label{sol-ODE}
If $\rho=\rho(r)$ and $v=v(r)$ is a solution to \eqref{steady2} --\eqref{eq:837}, then one has 
\bel{eq:322}
\aligned
r(r-2M){(\rho+\eps^2 p(\rho)) v \over 1 - \eps^2 v^2} &= D_0,
\\
 - {1\over 2\eps^2}\ln({1-\eps^2 v^2}) +l(\rho) + {1\over 2\eps^2} \ln\Bigg( 1 - {2M \over r} \Bigg) & = C_0,
\endaligned
\ee
where the function $l=l(\rho)$ is defined by $l'(\rho) := {p'(\rho) \over \rho+\eps^2 p(\rho)}$, and the constants above are determined by the initial conditions, that is, 
$$
\aligned
& D_0 :=r_0(r_0-2M){(\rho_0+\eps^2 p(\rho_0)) v_0 \over 1 - \eps^2 v_0^2},
\\
& C_0 := - {1\over 2\eps^2}\ln({1-\eps^2 v_0^2}) +l(\rho_0) + {1\over 2\eps^2} \ln \Bigg( 1 - {2M \over r_0} \Bigg). 
\endaligned
$$
\end{lemma}

Observe that by letting $\eps \to 0$ in \eqref{eq:322}, we recover our earlier formulas \eqref{eq:308}, \eqref{eq:309} for non-relativistic flows. 

\begin{proof}
The equation \eqref{steady2a} leads us immediately to the first equation in \eqref{eq:322}. 
Next, by multiplying \eqref{steady2b} by ${r\over r-2M}$, we find
$$
{d \over dr} \Big(r(r-2M){\rho v^2+ p \over 1 - \eps^2 v^2} \Bigg) = M{\rho v^2+ p \over 1 - \eps^2 v^2}  - {M \over \eps^2}{\rho+ \eps^4 p v^2 \over 1 - \eps^2 v^2}+2(r-2M)p, 
$$
which is equivalent to
$ {\rho+\eps^2 p  \over 1 - \eps^2 v^2}v{dv\over dr}+ {dp\over dr}+ {M\over r(r-2M)}(\eps^{-2} \rho+p) = 0$. 
Multiplying this equation by ${1 \over \rho+\eps^2p}$, we thus find ${v \over 1 - \eps^2 v^2}{dv\over dr}+ {1\over \rho+\eps^2 p}{dp\over dr}+ {M\over \eps^2  r(r-2M)} = 0$, which, by integration, yields the second equation in \eqref{eq:322}. 
\end{proof}

By now assuming the linear pressure law $p(\rho) = k^2 \rho$ with (constant) sound speed $0 < k<1/\eps$, we thus consider the differential system 
\bel{steady2-con}
\aligned
{d\over dr} \Big(r(r-2M) {(1 + \eps^2 k^2)  \over 1 - \eps^2 v^2} \rho v \Big)
& = 0, 
\\
{d\over dr} \Big((r-2M)^2 {v^2+ k^2 \over 1 - \eps^2 v^2} \rho \Big)
& 
= {M \over r} { (r-2M)  \over 1 - \eps^2 v^2}  \Big(3 \rho v^2+ 3 k^2 \rho - \eps^{-2} \rho- \eps^2 k^2 \rho v^2 \Big)  
  + {2k^2 \over r} (r-2M)^2  \rho. 
\endaligned
\ee 
By elementary algebra, in view of \eqref{eq:322} and $l(\rho) = {k^2 \over 1 + \eps^2 k^2} \log \rho$, we obtain 
$$
\aligned
&  r^2\rho^{1-\eps^2 k^2 \over 1 + \eps^2 k^2}v= r_0^2\rho_0^{1-\eps^2 k^2 \over 1 + \eps^2 k^2}v_0,\\ 
& \Bigg( 1 - {2M \over r} \Bigg) {1\over1-\eps^2 v^2}\rho^{2\eps^2 k^2 \over 1 + \eps^2 k^2} 
= \Bigg( 1 - {2M \over r_0} \Bigg) {1\over1-\eps^2 v_0^2}\rho_0^{2\eps^2 k^2 \over 1 + \eps^2 k^2}.
\endaligned 
$$
Consequently, by introducing the notation  
\bel{not349} 
\kappa := {1-\eps^2 k^2 \over 1 + \eps^2 k^2} \in (0,1), 
\qquad 
1 - \kappa = {2\eps^2 k^2 \over 1 + \eps^2 k^2},
\ee
we find 
\bel{330}
\aligned
r^2 \, \rho^\kappa v & = r_0^2 \, \rho_0^\kappa v_0,
\\ 
\Big( 1 - {2M \over r} \Big) {\rho^{1- \kappa} \over1-\eps^2 v^2}
& = \Big( 1 - {2M \over r_0} \Big) {\rho_0^{1- \kappa} \over1-\eps^2 v_0^2}.
\endaligned 
\ee
Clearly, the component $v$ has a constant sign and, for definiteness, we can now assume that $v_0 \geq 0$. 
By eliminating the density variable $\rho$, we arrive at an algebraic equation of the velocity $v$, i.e. 
\be
\ln {1-\eps^2 v_0^2 \over 1 - \eps^2 v^2} + {1- \kappa \over \kappa} \ln {v_0 \over v}
=
{1- \kappa \over \kappa} \ln {r^2 \over r_0^2} + \ln {r(r_0-2M) \over r_0( r-2M)}. 
\ee

Let us define a function $G_\eps$ of the variables $r,v$ (depending also upon the data $r_0,v_0$) by 
\bel{fon-G-eps}
G_\eps(r,v;r_0,v_0)
: = 
 \ln {1-\eps^2 v_0^2 \over 1 - \eps^2 v^2} + {1- \kappa \over \kappa} \ln {r_0^2 v_0 \over r^2 v} - \ln {r(r_0-2M) \over r_0(r-2M)}. 
\ee
(See Figure~\ref{Fig-51} for an illustration.) Note that, in the limit $\eps \to 0$ we recover the non-relativistic expression \eqref{Fon-G}.)  By definition, a function $v=v(r)$ is a solution to the problem \eqref{steady2-con} with initial data \eqref{eq:837} 
if and only if $G_\eps(r,v(r);r_0,v_0) \equiv 0$ and $v(r_0) =v_0$. 
We differentiate $G_\eps$ with respect to $v$ and $r$ and obtain 
$$
\aligned
\del_v G_\eps & = {v - k^2/v\over 1 - \eps^2 v^2},
\qquad 
\del_r G_\eps  = - {1 \over \eps^2 r} \Bigg({1- \kappa \over \kappa} + {M \over r-2M}
\Bigg) < 0. 
\endaligned
$$
Observe that $\del_v G_\eps= 0$ if and only if $v= k$. Moreover, $G_\eps$ is decreasing with respect to $v$ when $v< k$ and increasing when $v> k$. The derivative of a steady state solution is given by 
\bel{first-der}
{d v\over dr} = {v  \over r(r-2M)}
{{1- \kappa \over \kappa} (r-2M) - M \over{\eps^2 v^2 \over 1 - \eps^2 v^2} - {1- \kappa \over 2 \kappa}}
\ee
and changes sign once, at $r = {2- \kappa \over 1- \kappa} M \in (2M,+\infty)$. 

\begin{figure}[htbp]
\centering
\epsfig{figure=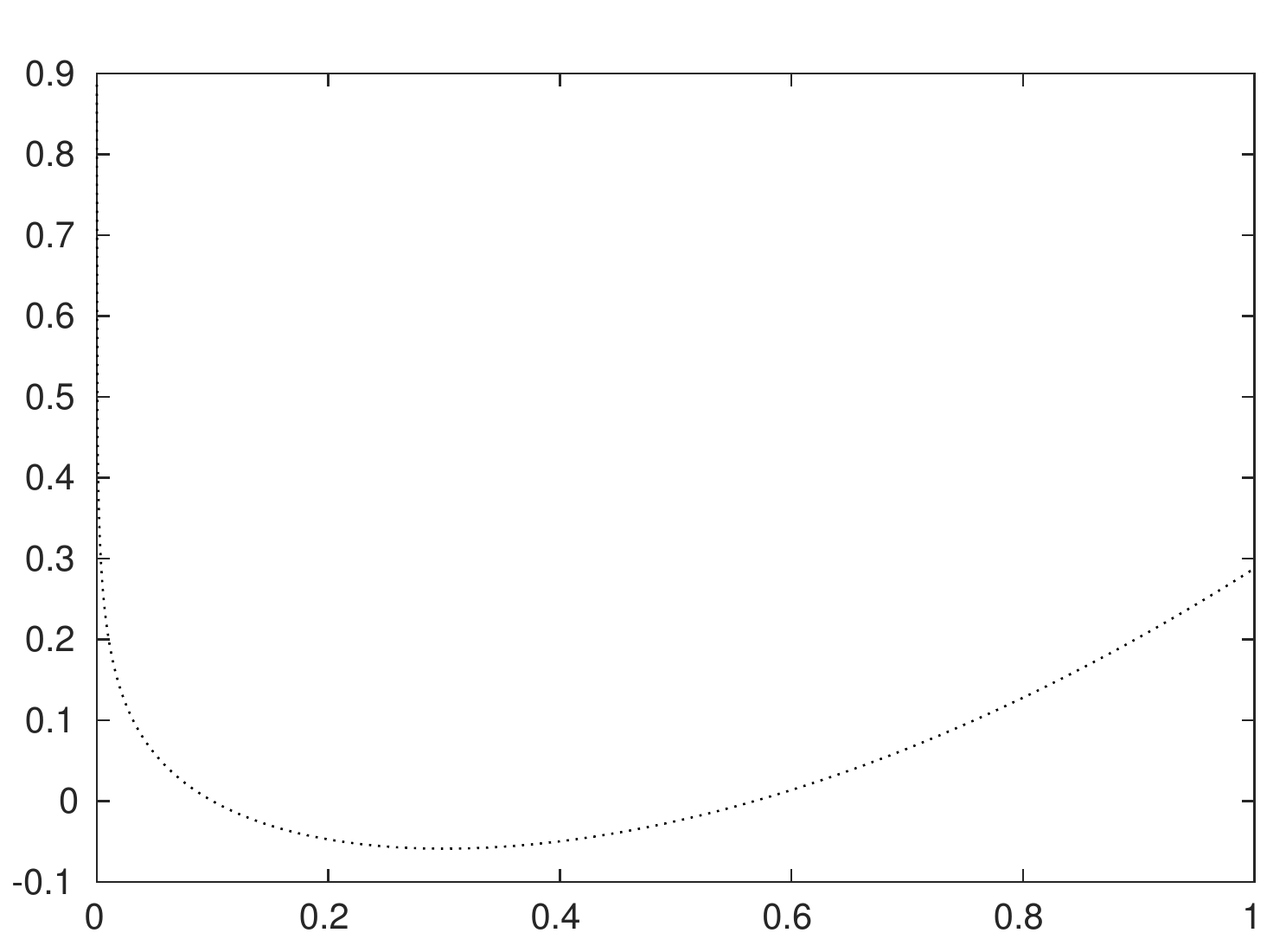,height=2.in} 
\caption{Plot of the function $v \mapsto G_\eps(v) =G_\eps(r,v; r_0, v_0)$ with $\eps=0.01$.}
\label{Fig-51}
\end{figure}


Since $G_\eps(r_0, v_0; r_0, v_0) = 0$ and $\del_v G_\eps(r_0, v_0; r_0, v_0) \neq 0$ provided $v_0 \neq k$, we can apply the implicit function theorem to a non-sonic velocity $v_0$. 

\begin{lemma}[The family of locally-defined steady states]
\label{local-sol} 
Given any radius $r_0 >2 M$ and any initial data $\rho_0 > 0$ and $v_0 \geq 0$ satisfying the non-sonic condition $v_0 \neq k$,
the initial value problem defined in \eqref{eq:837} and \eqref{steady2-con} 
admits a solution $\rho=\rho(r)$ and $v=v(r)$ denoted by 
$$
\rho= \rho(r;r_0,\rho_0,v_0), \qquad v= v(r;r_0,\rho_0,v_0),
$$ 
and defined in some neighborhood $\Uscr_0^\eps$ of $r_0 \in \RR$ (at least). 
\end{lemma}


\subsection*{Global existence theory} 

We now analyze the possible extension of the (smooth) solutions above to their maximum domain of existence. 
Since $\del_v G_\eps(r_0, v_0; r_0, v_0) = 0$ if and only if $v= k$, a solution can always be continued, unless the velocity component $v$ reaches the sonic speed. 

\begin{definition} A radius $r=r_* >2 M$ is called a {\bf sonic point} for the problem \eqref{eq:837} and \eqref{steady2-con} 
if it is a root of the following algebraic equation:  
\bel{sonic -2} 
\ln \Bigg({1-\eps^2 v_0^2 \over 1 - \eps^2 k^2} \Bigg)
+ {1- \kappa \over \kappa} \ln \Bigg( {v_0 \over k}\Bigg) 
= {1- \kappa \over \kappa}\ln \Bigg( 
{r^2 \over r_0^2} + \ln {r(r_0-2M) \over r_0(r-2M)} \Bigg).
\ee
\end{definition}

From \eqref{first-der}, it follows that the derivative ${dv \over dr}$ of a steady state solution blows-up when one approaches a sonic value. 
In the following, it will be useful to observe that 
\bel{nota-85}
\eps^2 k^2 = {1 - \kappa \over 1 + \kappa}, 
\qquad 
1 + 3 \eps^2 k^2 = {2(2 - \kappa) \over 1 +\kappa}, 
\qquad 
{1 + 3 \eps^2 k^2 \over 2 \eps^2 k^2} = {2 - \kappa \over 1- \kappa}. 
\ee
In order to simplify the notation, we introduce the following function $P_\eps$ of the radius $r_0$ and velocity $v_0$:
\bel{function-P-eps}
P_\eps(r_0, v_0): = \ln \Bigg( {(2 - \kappa)^2 \over (1- \kappa)^2} {M^2 k \over r_0^2 v_0} \Bigg)
+ {\kappa \over 1- \kappa} \ln \Bigg( {2(2 - \kappa) \over 1 +\kappa} 
{(r_0-2M) \over  r_0(1-\eps^2 v_0^2)} \Bigg). 
\ee
The importance of the sign of $P_\eps(r_0, v_0)$ is identified in the following lemma.

\begin{lemma}[Existence/non-existence of sonic points]  
\label{sonic-point}
Consider a solution $v=v(r)$ associated with a positive and non-sonic velocity $v_0 > 0$ with $v_0 \neq k $:
\begin{enumerate}
\item If $P_\eps(r_0, v_0) > 0$, 
there exists no sonic  point.

\item If if $P_\eps(r_0, v_0) \leq 0$, 
there exist two sonic  points $\underr_*\leq \barr_*$. Moreover, one has: 

\bei 
 
\item If $r_0 \geq {2 - \kappa \over 1- \kappa} M$, the roots satisfy $2M<\underr_*\leq \barr_* \leq r_0$.

\item If $r_0 < {2 - \kappa \over 1- \kappa} M$, the roots satisfy $2M<r_0 <\underr_*\leq \barr_* $. 
\eei
\end{enumerate}
\end{lemma}

\begin{proof} Introduce the following function of the velocity variable $v_0 > 0$: 
$$
L_\eps (v_0) := {\kappa \over 1 - \kappa} 
\ln \Bigg( {1-\eps^2 v_0^2 \over 1 - \eps^2 k^2} \Bigg) + \ln {v_0 \over k}, 
$$
which satisfies $L_\eps' (v_0) = {1\over v_0} \Big(1- {v_0^2 \over k^2} {1-\eps^2 k^2\over 1 - \eps^2 v_0^2} \Big)$. Thus, $L_\eps' (v_0) = 0 $ if and only if $v_0= k$. Hence, $L_\eps$ achieves its maximum at $k$, that is, 
$L_\eps (v_0 \leq L_\eps (k) = 0$. Therefore for all non-sonic $v_0$, we have  
$-\infty<L_\eps (v_0) < 0$. 

Now, consider the following function of the spatial variable 
$$
R_\eps(r): =\ln {r^2 \over r_0^2} + {\kappa \over 1- \kappa} \Bigg( \ln {r\over r-2M} - \ln {r_0 \over r_0-2M} \Bigg), 
$$
which satisfies
$R_\eps'(r) = {2 \over r(r-2M)} \Big((r-2M) - {\kappa \over 1- \kappa} 
M \Big)$. 
Therefore, the function $R_\eps$ reaches its minimum at $r_{min} := {2 - \kappa \over 1- \kappa} M$ 
and
$$
R_\eps (r_{min}) = \ln \Bigg( {(2 - \kappa) \over (1- \kappa)} {M^2 \over r_0^2} \Bigg) 
+ {\kappa \over 1- \kappa}  \ln\Bigg( {2 - \kappa \over \kappa} \Big( 1 - {2M \over r_0} \Big) \Bigg).
$$ 
Observe also that the mininum value $R_\eps (r_{min})$ reaches its maximum value $0$ when $r_0=r_{min}$. Therefore, if and only if $R_\eps (r_{min}) -L_\eps (v_0) > 0$, no sonic  point can be found; otherwise, we have two sonic points. The positions of $r_0$ and $r_{min}$ determine the location of the sonic  points $\underr_*\leq \barr_*$. Furthermore, since $R_\eps(2M) = +\infty$, we have the lower bound $2M<\underr_*$. 
\end{proof}

We need now to distinguish between several cases and the following notation will be useful: 
\be
\aligned 
&  \tilde {\mathfrak A}: P_\eps(r_0, v_0) > 0, \quad 
&&&&\tilde {\mathfrak B}: P_\eps(r_0, v_0) \leq 0,
\\
& \tilde {\mathfrak 1}:  v_0 < k, \quad 
&&&&\tilde {\mathfrak 2}:  v_0 > k,
\\
& \tilde {\mathfrak i}: r_0 \geq {2 - \kappa \over 1- \kappa} M, \quad 
&&&& \tilde {\mathfrak {ii}}: r_0 < {2 - \kappa \over 1- \kappa} M. 
\endaligned 
\ee
We are now ready to continue the local solutions in Lemma~\ref{local-sol} beyond the neighborhood $\Uscr_0^\eps$. There are two main regimes, which we now discuss. 

\begin{lemma}[Extension of steady state solutions without sonic point]
\label{extension3}
Given a radius $r_0 >2M$, a density $\rho_0 > 0$, and a non-sonic velocity $0 \leq v_0 < 1/\eps$ (satisfying $v_0 \neq k$), the local solution  $\rho=\rho(r;r_0,\rho_0,v_0)$ and $v=v(r;r_0,\rho_0,v_0)$ given in Lemma~\ref{local-sol} satisfies the following properties:  
\begin{enumerate}

\item {\bf Case} $\widetilde {\mathfrak {A1}}$. The solution can be extended to $(2M,+\infty)$ satisfying $v< k$ with 
$$
\lim_{r \to 2M}v(r;r_0,\rho_0,v_0) =\lim_{r \to +\infty}v(r;r_0,\rho_0,v_0) = 0.
$$
The solution satisfies the monotonicity that $v$ is increasing with respect to $r$ on the interval $(2M, {2 - \kappa \over 1- \kappa} M)$ while it is decreasing on $({2 - \kappa \over 1- \kappa} M, +\infty)$. 

\item {\bf Case} $\widetilde {\mathfrak {A2}}$. The solution can be extended to $(2M,+\infty)$ satisfying $v> k$ with 
$$ 
\lim_{r \to 2M}v(r;r_0,\rho_0,v_0) =\lim_{r \to +\infty}v(r;r_0,\rho_0,v_0) = {1 \over  \eps}. 
$$
The following monotonicity property holds:  $v$ is decreasing with respect to $r$ on the interval $(2M, {2 - \kappa \over 1- \kappa} M)$ while it is increasing on $({2 - \kappa \over 1- \kappa} M, +\infty)$. 
\end{enumerate}
\end{lemma}

\begin{lemma}[Extension of steady state solutions with sonic points]
\label{extension4}

Given a radius $r_0 >2M$, a density $\rho_0 > 0$, and a non-sonic velocity $0 \leq v_0 < 1/\eps$ (satisfying $v_0 \neq k$), the local solution  $\rho=\rho(r;r_0,\rho_0,v_0)$ and $v=v(r;r_0,\rho_0,v_0)$ given in Lemma~\ref{local-sol} satisfies the following properties, in which $\barr_* \leq \underr_*$ denotes the sonic points given by Lemma~\ref{sonic-point}: 
\begin{enumerate}

\item {\bf Case} $\widetilde {\mathfrak {B1i}}$. The solution $v=v(r)$ can be extended to $(\barr_*,+\infty)$ and satisfies $v\leq k$, with 
$$ 
\lim_{r \to +\infty}v(r;r_0,\rho_0,v_0) = 0, \qquad \lim_{r \to \barr_*}v(r;r_0,\rho_0,v_0) = k. 
$$
Moreover, $v$ is decreasing with respect to $r$ on $( \barr_*, +\infty)$. 

\item {\bf Case} $\widetilde {\mathfrak {B2i}}$. The solution $v=v(r)$ can be extended to $(\barr_*,+\infty)$ and satisfies  $v\geq k$, with 
$$ 
\lim_{r \to +\infty}v(r;r_0,\rho_0,v_0) = {1 \over  \eps}, \qquad \lim_{r \to \barr_*}v(r;r_0,\rho_0,v_0) = k. 
$$
Moreover, $v$ is increasing with respect to $r$ on $( \barr_*, +\infty)$.

\item {\bf Case} $\widetilde {\mathfrak {B1ii}}$. The solution $v=v(r)$ can be extended to $(0,\underr_*)$ and satisfies  $v\leq k$, with 
$$ 
\lim_{r \to 2M}v(r;r_0,\rho_0,v_0) = 0, \qquad \lim_{r \to \underr_*}v(r;r_0,\rho_0,v_0) = k. 
$$
Moreover, $v$ is increasing with respect to $r$ on $( 0, \underr_*)$. 

\item {\bf Case} $\widetilde {\mathfrak {B2ii}}$. The solution $v=v(r)$ can be extended to $(0,\underr_*)$ and satisfies  $v\geq k$, with 
$$ 
\lim_{r \to 2M}v(r;r_0,\rho_0,v_0) = {1 \over  \eps}, \qquad \lim_{r \to \underr_*}v(r;r_0,\rho_0,v_0) = k. 
$$
Moreover, $v$ is decreasing with respect to $r$ on $( 0, \underr_*)$. 
\end{enumerate}
\end{lemma}

The proof of Lemmas~\ref{extension3} and \ref{extension4} follows the same lines as the ones of Lemmas~\ref{extension1} and  \ref{extension2}, respectively. Note that since $G_\eps$ has its minimum at $v= k$, we have $G_\eps(r, v; r_0, v_0) >G_\eps(r,k; r_0, v_0) > 0$ for all $r \in (\underr_*, \barr_*)$, and we see that no solution can be defined on the interval $(\underr_*, \barr_*)$ limited by the two roots. 


\subsection*{Main conclusion for this section}

We can now summarize the properties of steady state solutions. We refer to Figures~\ref{Fig-JD8} to \ref{Fig-JD10} for an illustration for several values of the physical parameters $\eps, k,m$. 

\begin{theorem}[Steady flows on a Schwarzschild background]
\label{steady-relativistic} 
Given some values of the light speed $\eps> 0$, sound speed $k \in (0, 1/\eps)$, and black hole mass $M> 0$, consider the Euler model $\Mscr(\eps, k, m= M/\eps^2)$ in \eqref{Euler-con} describing fluid flows on a Schwarzschild background. 
Then, for any given any radius $r_0 >2M$, density $\rho_0 > 0$, and velocity $v_0 \geq 0$ with $v_0 \neq k$, there exists a unique steady state solution denoted by 
$$
\rho=\rho(r;r_0,\rho_0,v_0), \qquad v=v(r;r_0,\rho_0,v_0), 
$$ 
satisfying the steady state equations \eqref{steady2-con} together with the initial condition 
$\rho(r_0) =\rho_0$ and $v(r_0) =v_0$. Moreover, the velocity component satisfies $\sgn(v(r) -k) = \sgn(v_0-k)$ for all relevant values $r$, and in order to specify the range of the independent variable $r$ where this solution is defined, one distinguishes between two alternatives:  
\begin{enumerate}

\item {\bf Regime without sonic point.} If $ P_\eps(r_0, v_0) > 0$ (this function being introduced in \eqref{function-P-eps}), the solution is defined on the whole interval $(2M,+\infty)$. 

\item {\bf Regime with sonic points.}  If $P_\eps(r_0, v_0) \leq 0$, the solution is defined on the interval ${\Pi \varsubsetneqq (2M,+\infty)}$ defined by    
\bel{pi}
\Pi :=
\begin{cases}
(0, \underr_*), \qquad & r_0 < {2 - \kappa \over 1- \kappa} M,
\\
(\barr_*,+\infty),          & r_0 \geq {2 - \kappa \over 1- \kappa} M. 
\end{cases} 
\ee
Moreover, the velocity $v(r)$ tends to the sonic velocity $k$ when $r$ approaches the sonic radius ($\underr_*$ or $\barr_*$, introduced in Lemma~\ref{sonic-point}). 
\end{enumerate}
\end{theorem}


\begin{figure}[htbp]
\begin{minipage}[t]{0.3\linewidth}
\centering
\epsfig{figure=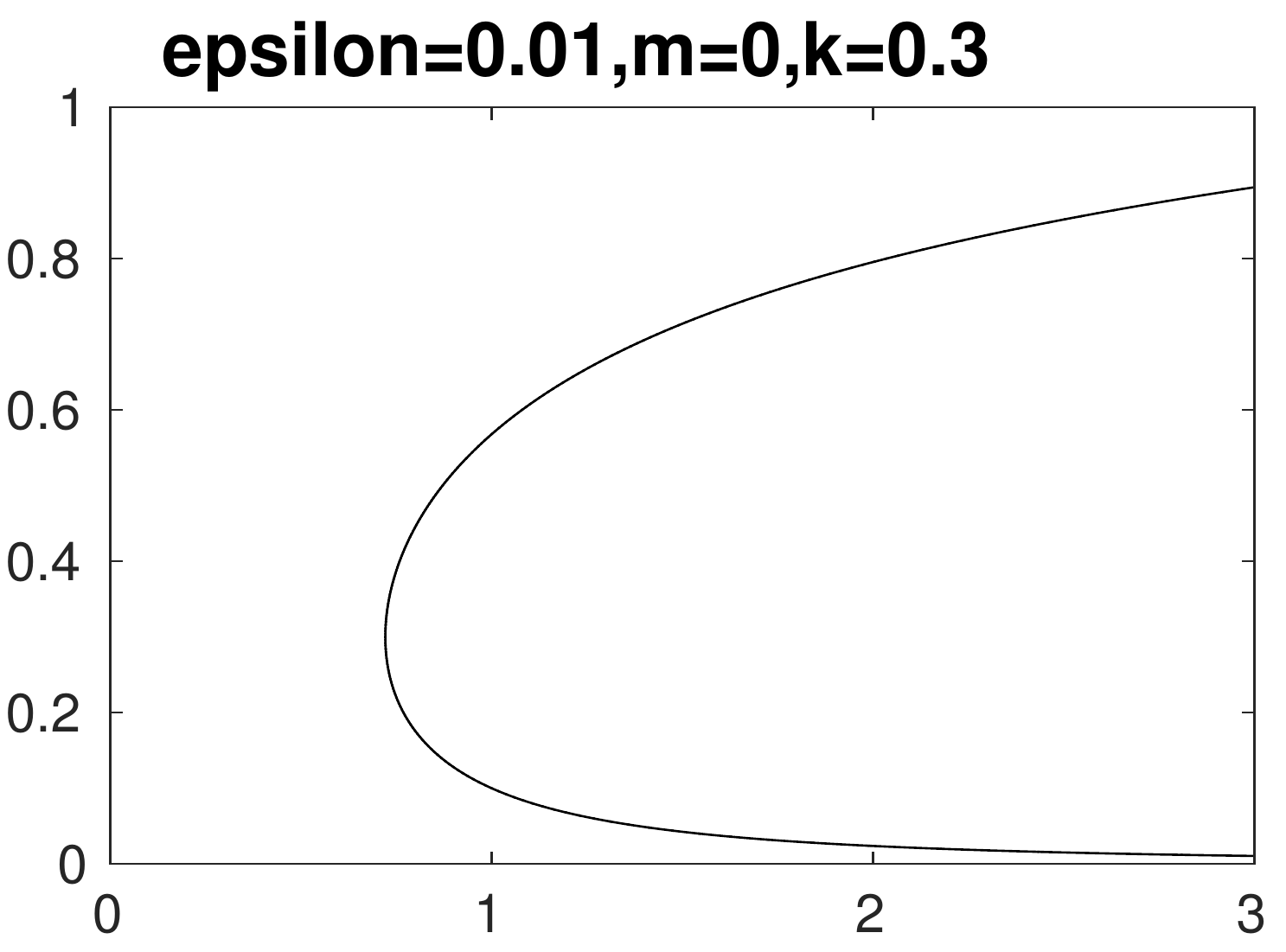,height=1.3in} 
\end{minipage}
\begin{minipage}[t]{0.3\linewidth}
\centering

\epsfig{figure=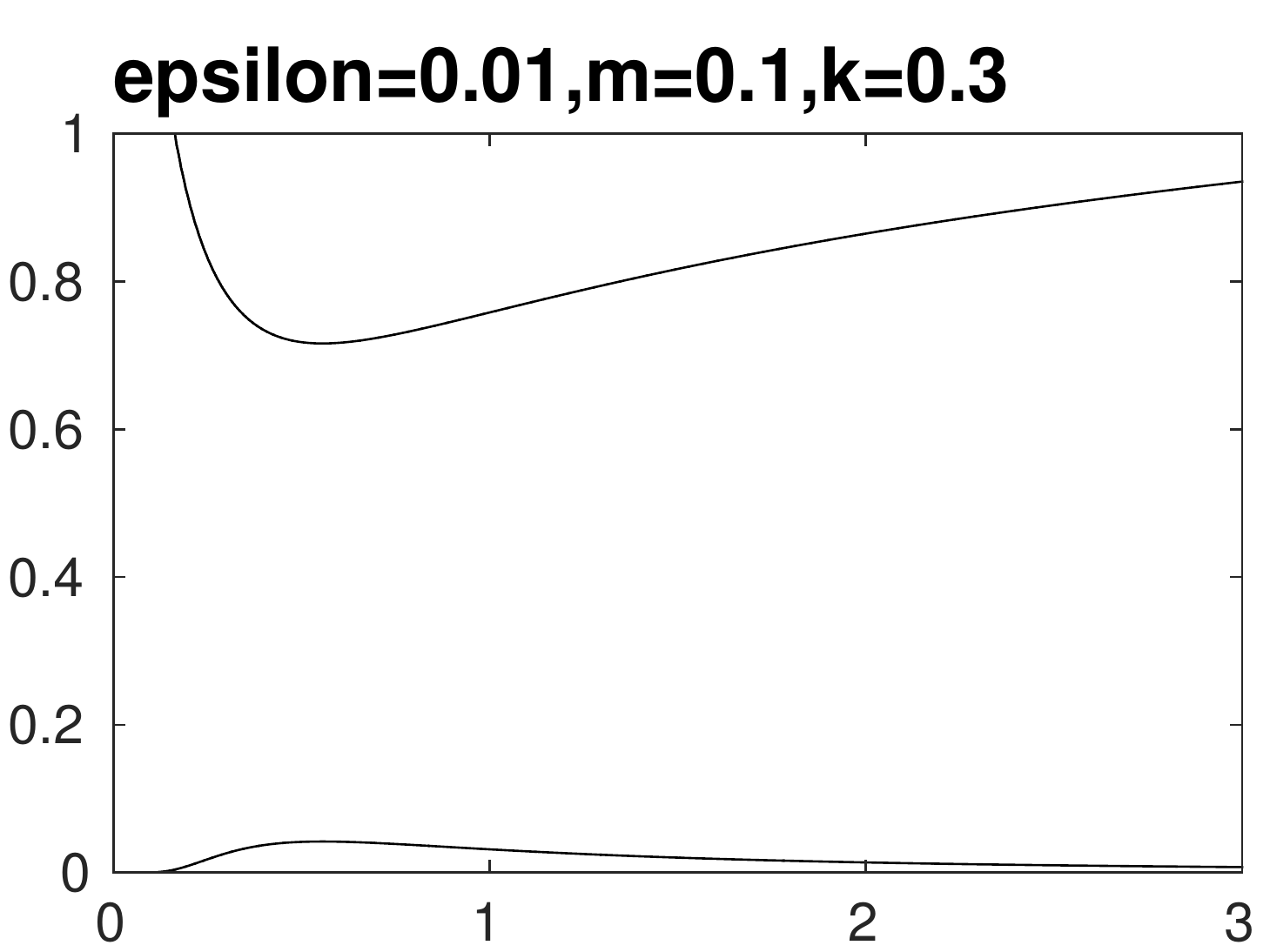,height=1.3in} 
\end{minipage}
\begin{minipage}[t]{0.3\linewidth}
\centering 
\epsfig{figure=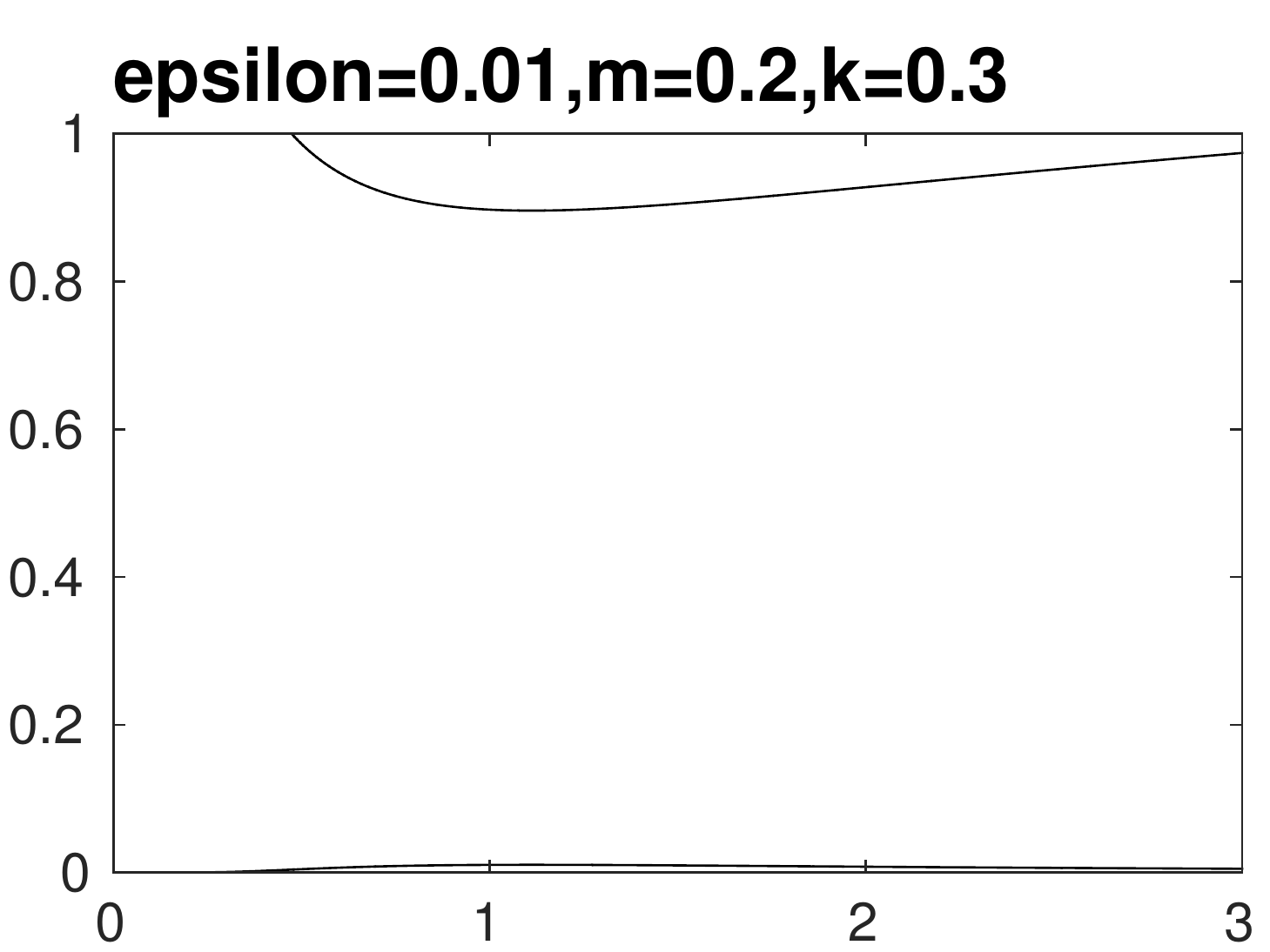,height=1.3in} 
\end{minipage}
\caption{Solution $v=v(r)$ for $\eps= 0.01, k= 0.3$ and several values $m= M/\eps^2$.}
\label{Fig-JD8}
\end{figure}

\begin{figure}[htbp]
\begin{minipage}[t]{0.3\linewidth}
\centering
\epsfig{figure=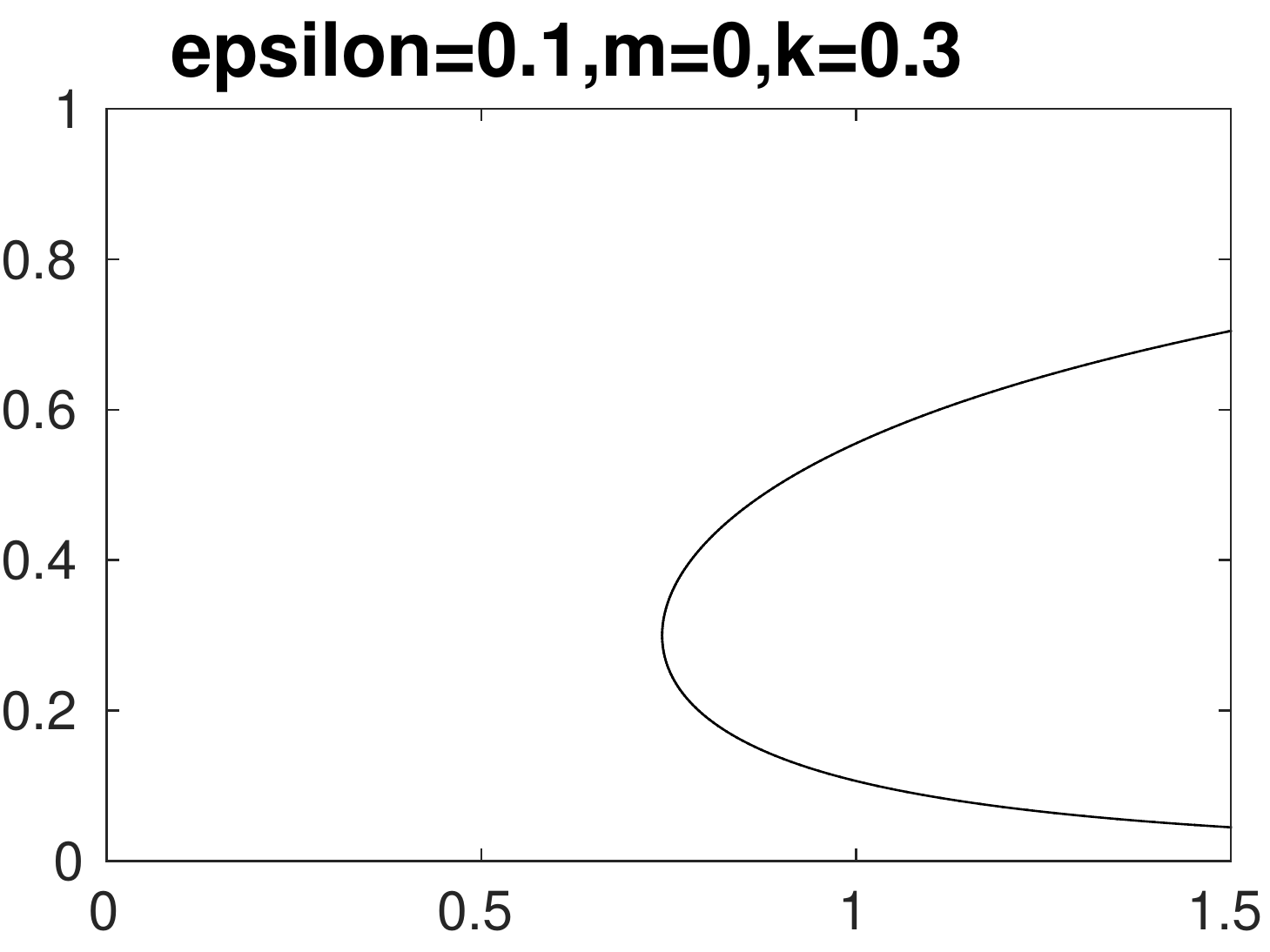,height=1.3in} 
\end{minipage}
\begin{minipage}[t]{0.3\linewidth}
\centering

\epsfig{figure=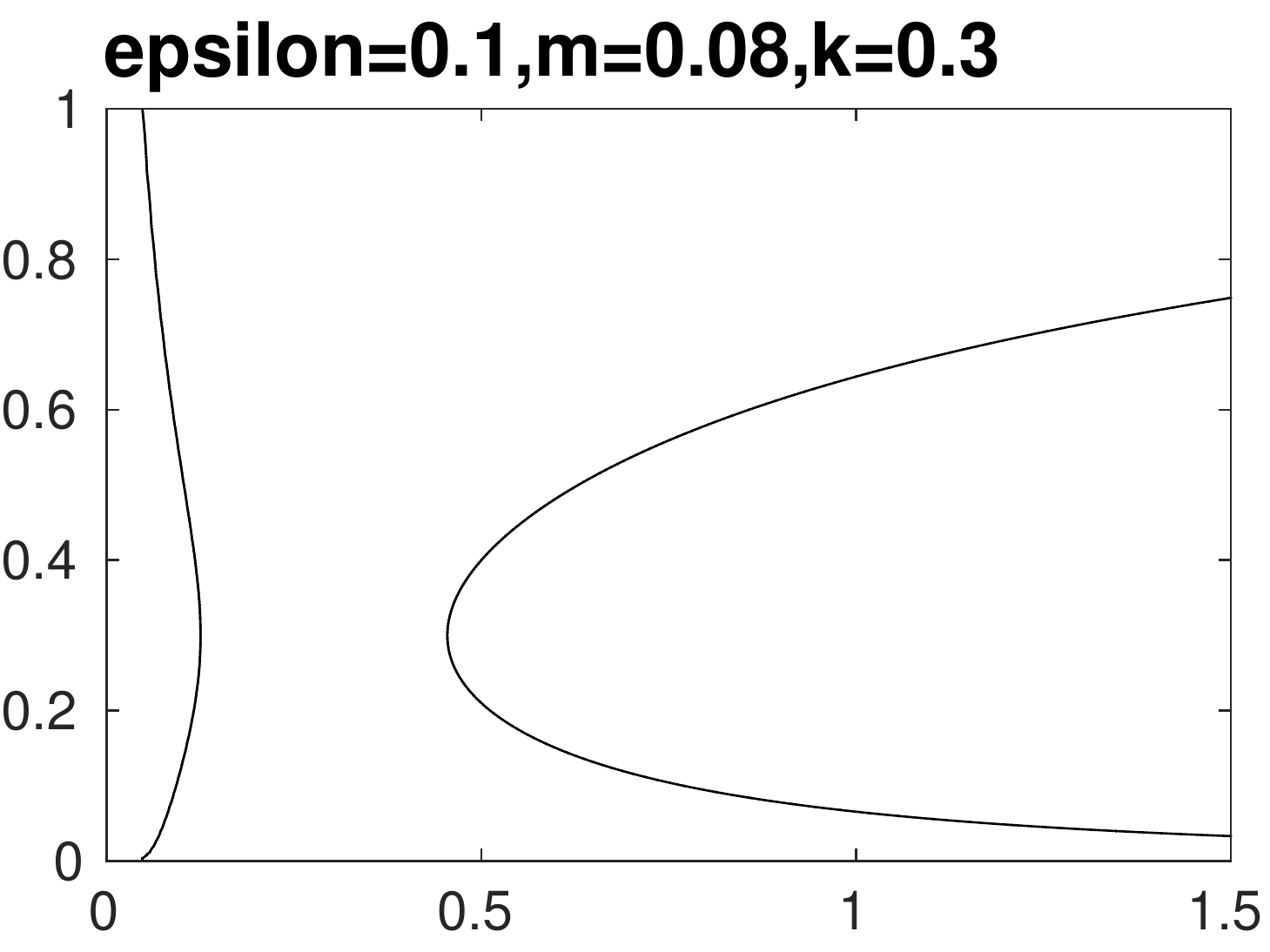,height=1.3in} 
\end{minipage}
\begin{minipage}[t]{0.3\linewidth}
\centering

\epsfig{figure=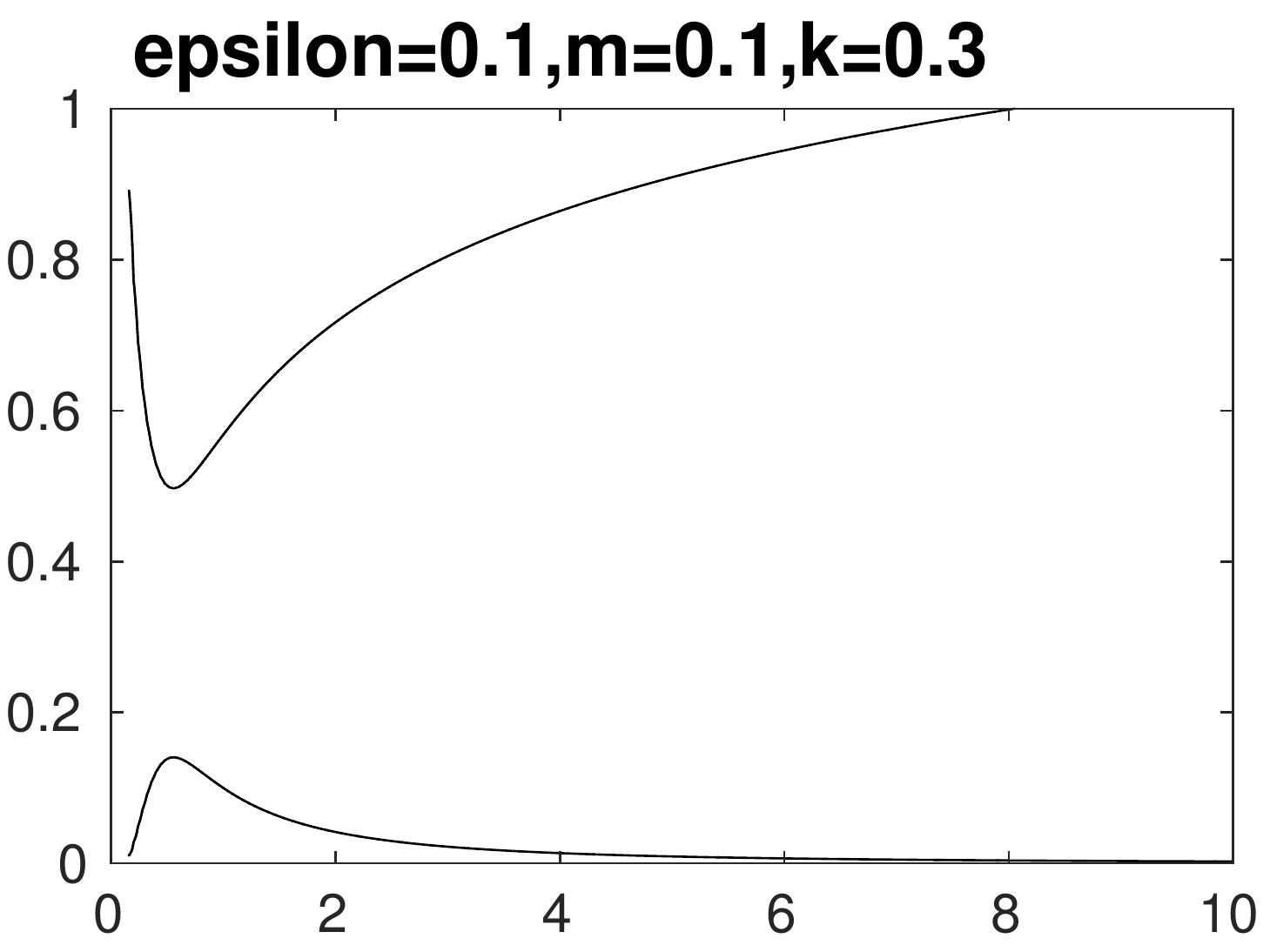,height=1.3in} 
\end{minipage}
\caption{Solution $v=v(r)$ for $\eps= 0.1, k= 0.3$ and several values $m= M/\eps^2$.}
\label{Fig-JD9}
\end{figure} 

\begin{figure}[htbp]
\begin{minipage}[t]{0.3\linewidth}
\centering
\epsfig{figure=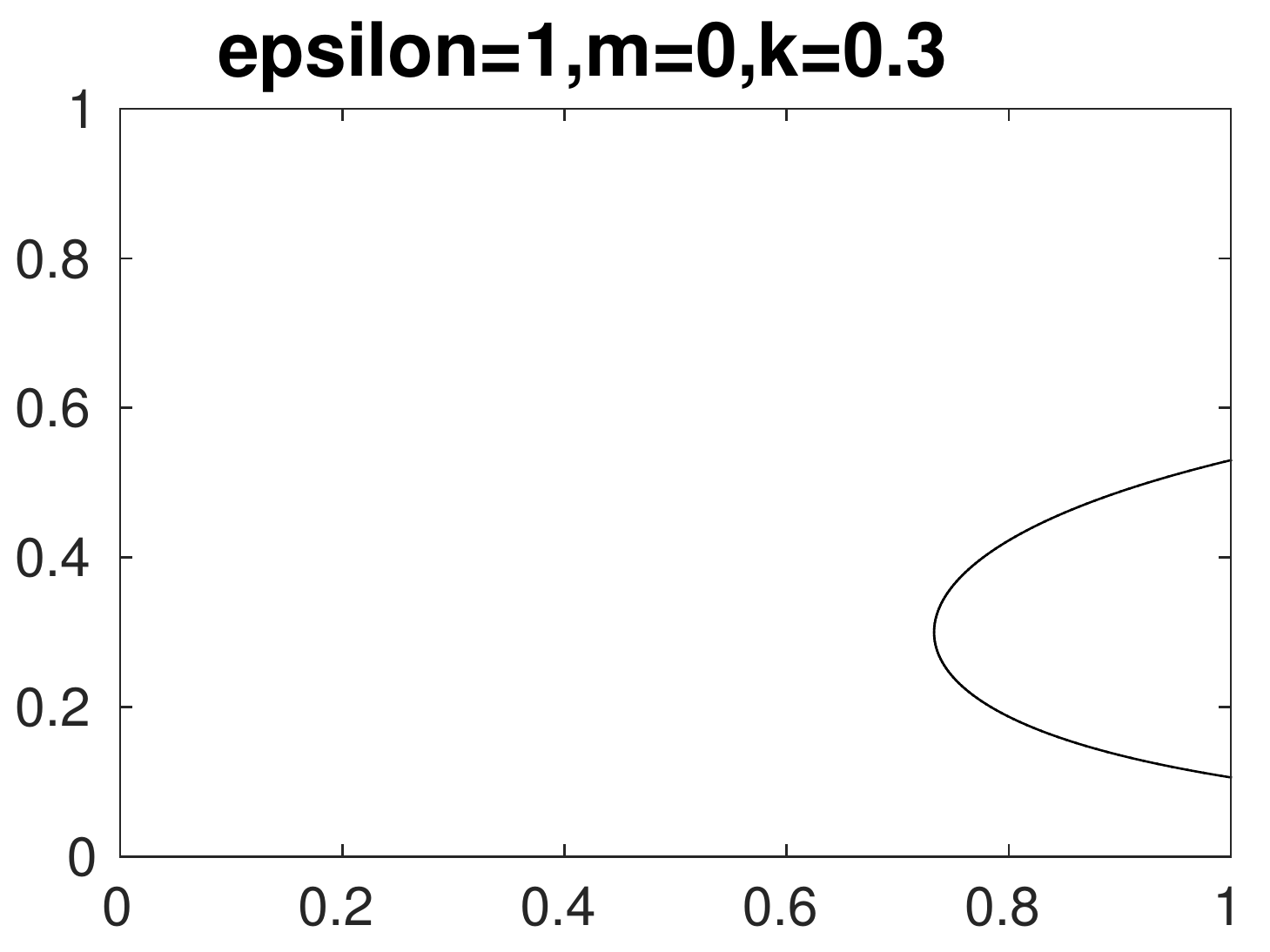,height=1.3in} 
\end{minipage}
\begin{minipage}[t]{0.3\linewidth}
\centering

\epsfig{figure=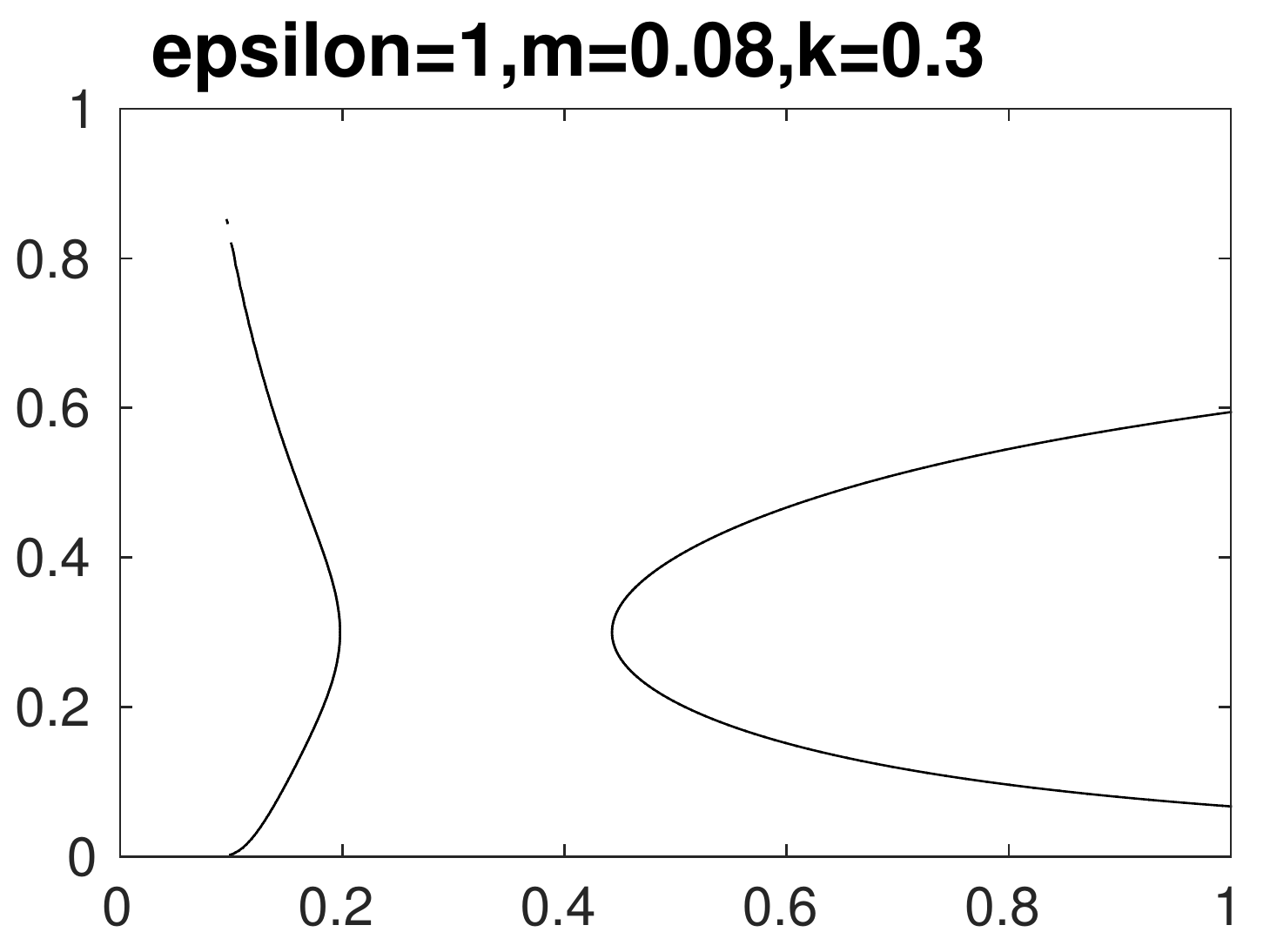,height=1.3in} 
\end{minipage}
\begin{minipage}[t]{0.3\linewidth}
\centering

\epsfig{figure=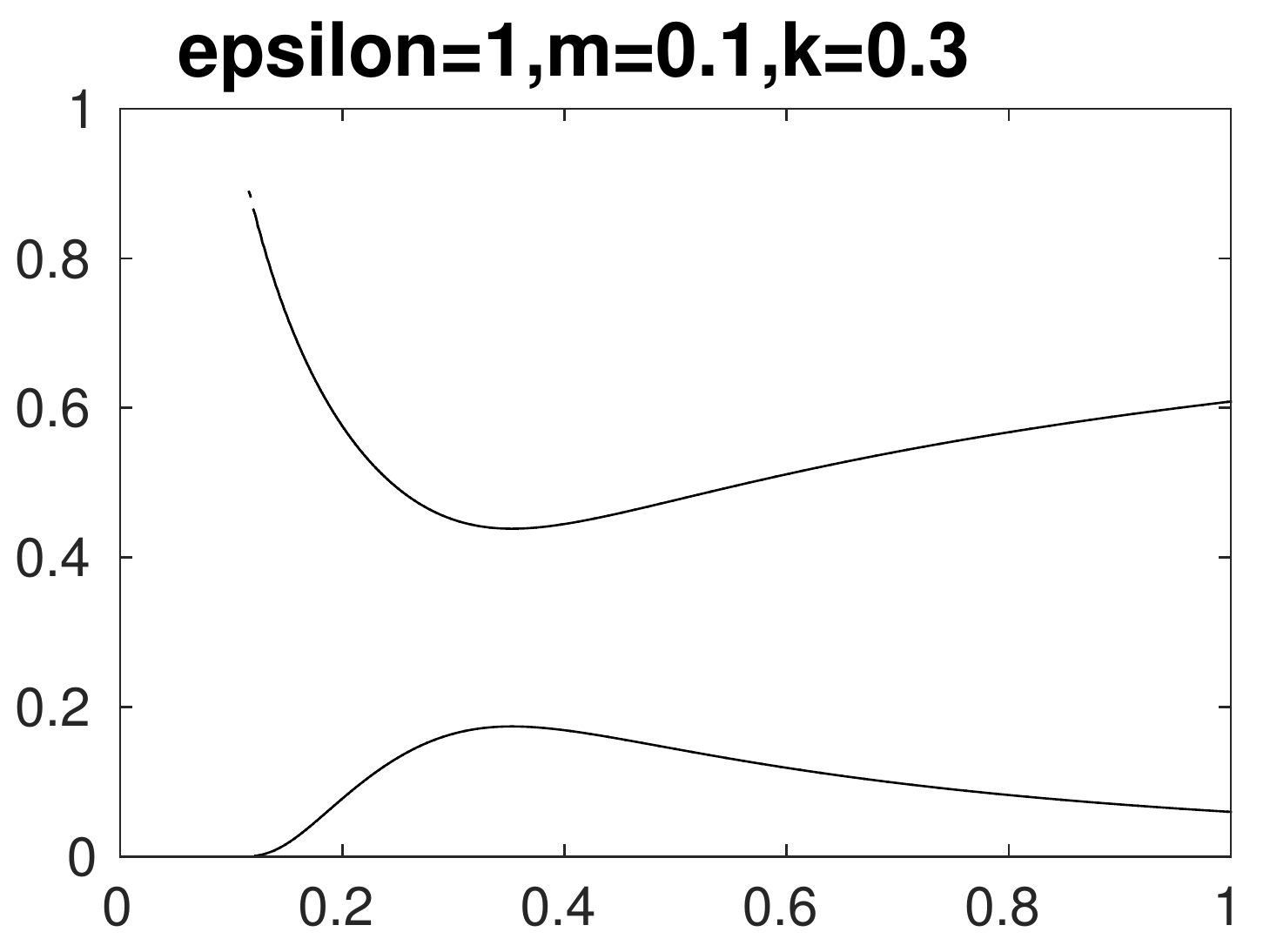,height=1.3in} 
\end{minipage}
\caption{Solution $v=v(r)$ for $\eps=1, k= 0.3$ and several values $m= M/\eps^2$.}
\label{Fig-JD10}
\end{figure}


\section{The Riemann problem for the Euler equations}
\label{sec:6}

\subsection*{Preliminaries}

In this section, we consider the solution of the Riemann problem for our general Euler model in a Schwarzschild background \eqref{Euler-con}, which has the form of a nonlinear hyperbolic system of balance laws: 
\bel{Eulerform}
\del_ t U+ \del_r F(U,r) = S(U,r),
\ee 
where the ``conservative variables'' and ``flux variables'' are 
\bel{Eulerform22}
U=\Bigg(
\begin{array}{cccc}
 U_1\\
U_2\\
\end{array}
\Bigg)
=\Bigg(
\begin{array}{cccc}
    r^2 {1 + \eps^4 k^2 v^2 \over 1 - \eps^2 v^2} \rho\\
 r(r-2M) {1 + \eps^2 k^2 \over 1 - \eps^2 v^2} \rho v\\

\end{array}
\Bigg),
\ee
\bel{Eulerform23}
F(U,r) =\Bigg(
\begin{array}{cccc}
  F_1(U,r) \\
F_2(U,r) \\
\end{array}
\Bigg)
=\Bigg(
\begin{array}{cccc}
   r(r-2M) {1 + \eps^2 k^2 \over 1 - \eps^2 v^2} \rho v\\
(r-2M)^2 {v^2+ k^2 \over 1 - \eps^2 v^2} \rho\\
\end{array}
\Bigg),
\ee
respectively, while the ``source term'' reads  
\bel{Eulerform24}
S(U,r) =\Bigg(
\begin{array}{cccc}
  S_1(U,r) \\
S_2(U,r) \\
\end{array}
\Bigg)
=
\Bigg(
\begin{array}{cccc}
  0 \\
3M   \Big( 1 - {2M \over r} \Big) {v^2+ k^2 \over 1 - \eps^2 v^2} \rho
 -M {r-2M \over \eps^2 r} {1 + \eps^4 k^2 v^2 \over 1 - \eps^2 v^2}\rho +2{(r-2M)^2 \over r}k^2 \rho
\end{array}
\Bigg). 
\ee
By definition, the {\bf Riemann problem} for \eqref{Eulerform} is the initial value problem associated with an initial data $U_0$ consisting of a  left-hand constant state $U_L= (\rho_L,v_L)$ and a right-hand constant state $U_R= (\rho_R,v_R)$, separated by a jump discontinuity at some point $r=r_0$ (with $r_0 >2M$). In other words, we set 
\bel{eq:500}
U_0(r) =
\begin{cases}
U_L, \qquad & r < r_0,
\\
U_R              & r > r_0.
\end{cases} 
\ee
In Proposition~\ref{nonlinearpo}, we have seen that both eigenvalues of \eqref{Eulerform} are genuinely nonlinear, when the sound speed $k$ is a constant satisfying $0 < k< 1/\eps$, which we now assume throughout. We are going to solve the Riemann problem first 
for the {\bf homogeneous system} 
\bel{Eulerform-homo}
\del_ t U+ \del_r F(U,r_0) = 0
\ee 
for a given $r_0 >2M$. 
in the class of self-similar functions (depending only on the variable $y:= {r-r_0 \over t}$) 
consisting of constant states, separated by either shock waves or rarefaction waves. 
Furthermore, it is convenient to introduce the {\bf fluid constant} $\chi$  
and the {\bf scaled velocity} defined by  
\bel{def-nu}
\chi
:= {2\eps k \over 1 + \eps^2 k^2} \in (0,1), \qquad  
\nu : = {1\over 2\eps}{1 +\eps v\over 1 - \eps v} \in (-\infty, +\infty). 
\ee


\subsection*{Rarefaction curves}
\label{sec-raf}

We begin by searching for smooth solutions to the Euler system  depending only upon the sef-similar variable. The partial differential system \eqref{Eulerform} then reduces to an ordinary differential system for functions $\rho=\rho(y)$ and $\nu=\nu(y)$ and, according to the discussion in the proof of Lemma~\ref{Riemann-inva}, we know that one of the Riemann invariants $w,z$ must remain constant throughout. Hence, we are led to the notion of {\bf rarefaction curves}: given any state $U_L$, the $1$-rarefaction curve $R_1^\rightarrow(U_L)$ is the curve passing throught $U_L$ along which the Riemann invariant $w$ remains constant and, in addition, the first eigenvalue $\lambda$ is {\sl increasing}. The definition of the (backward) curve $R_2^{\leftarrow}(U_R)$ for a given right-hand state $U_R$ is similar: the Riemann invariant $z$ remains constant and, in addition, the second eigenvalue $\mu$ is {\sl decreasing}.
We thus have 
$$ 
\aligned
& R_1^{\rightarrow}(U_L) =  \Big\{w(\rho,v) = w(\rho_L,v_L), \quad z(\rho,v) < z(\rho_L,v_L) \Big\},
\\
& R_2^{\leftarrow}(U_R) = \Big\{ z(\rho,v) = z(\rho_R,v_R), \quad  w(\rho,v) > w(\rho_R,v_R) \Big\}. 
\endaligned
$$
By observing that 
\be
w = {1 \over 2\eps} \ln \big(2\eps \nu \rho^\chi \big), 
\qquad 
 z = {1 \over 2\eps}\ln \big(2\eps \nu \rho^{-\chi} \big), 
\ee
the following statement is immediate. 

\begin{lemma} 
\label{Rarefac}
The two rarefaction curves associated with constant states $U_L$ and $U_R$, respectively, are given by 
\bel{Rarefac-curves} 
\aligned
& R_1^{\rightarrow}(U_L) =  \Bigg\{ {\nu \over \nu_L} = \Big({\rho \over \rho_L} \Big)^{-\chi}, \quad \rho>\rho_L \Bigg\},
\qquad 
 R_2^{\leftarrow}(U_R) =\Bigg\{ {\nu \over \nu_R} = \Big({\rho \over \rho_R} \Big)^{\chi}, \quad \rho<\rho_R \Bigg\}. 
\endaligned
\ee
\end{lemma}

\subsection*{Shock curves}
\label{Shock}

We next search for solutions consisting of two constant states separating a single jump discontinuity satisfying the Euler system \eqref{Eulerform}. Along a shock curve we impose the Rankine-Hugoniot relations (see below) as well as Lax entropy inequalities (see \eqref{ineq-Lax}, below), which can be stated as follows: the characteristic speed $\lambda$ must be {\sl decreasing} when moving away from the left-hand state $U_L$ on the $1$-shock curve $S_1^{\rightarrow}(U_L)$, while 
$\mu$ is {\sl increasing} as one moves away from the right-hand state $U_R$ on the {\sl backward} $2$-shock curve $S_2^{\leftarrow}(U_R)$. 
 
\begin{lemma} 
 \label{avshock}
The $1$-shock curve and the $2$-shock curve issuing from given constant states, denoted by $U_L= (\rho_L,v_L)$ and $U_R= (\rho_R,v_R)$, respectively, are given by  
\bel{Shock-curves}
\aligned 
& S_1^{\rightarrow}(U_L) = \Bigg\{ \sqrt{\nu \over \nu_L} - \sqrt{\nu_L \over \nu} =-\chi  \Big( \sqrt{\rho \over \rho_L} -\sqrt{\rho_L\over \rho} \Big), \quad \rho>\rho_L \Bigg\},
\\
& S_2^{\leftarrow}(U_R) =  \Bigg\{ \sqrt{\nu \over \nu_R} - \sqrt{\nu_R \over \nu} = \chi \Big( \sqrt{\rho \over \rho_R} -\sqrt{\rho_R\over \rho} \Big), \quad \rho<\rho_R \Bigg\}. 
\endaligned 
\ee 
The speed $s_1(U_L, U)$ along the $1$-shock curve and the speed $s_2(U, U_R)$ along the $2$-shock curve are given by 
\bel{shockspeed}
\aligned
& \eps s_1(U_L, U) =-\Big(1- {2M\over r} \Big)
    \Bigg({\rho \over \rho-\rho_L} {\eps^2 v^2 \over 1 - \eps^2 v^2} + {\eps^2 k^2 \over 1 + \eps^2 k^2} \Bigg)^{1/2}
      \Bigg({\rho \over \rho-\rho_L} {\eps^2 v^2 \over 1 - \eps^2 v^2} + {1 \over 1 + \eps^2 k^2} \Bigg)^{-1/2}, 
 \\
& \eps s_2(U, U_R) = \Big(1- {2M\over r} \Big)
           \Bigg({\rho \over \rho-\rho_R} {\eps^2 v^2 \over 1 - \eps^2 v^2} + {\eps^2 k^2 \over 1 + \eps^2 k^2} \Bigg)^{1/2}
            \Bigg({\rho \over \rho-\rho_R} {\eps^2 v^2 \over 1 - \eps^2 v^2} + {1\over 1 + \eps^2 k^2} \Bigg)^{-1/2}.
\endaligned 
\ee
\end{lemma} 

\begin{proof} 1. We use here the notation $U_i = (\rho_i,v_i)$ and $U= (\rho,v)$ for the two states on each side of a jump discontinuity, which must satisfy the {\bf Rankine-Hugoniot relations} associated with \eqref{Eulerform}. To simplify the calculation, we use the tensor components $v^0,v^1$ rather than the scalar velocity $v$. Denoting the shock speed by $s$, we see that the Rankine-Hugoniot jump conditions read
$$
\aligned
s \Big[ r(r-2M) \Big((1-2M/r)^2 \rho v^0v^0+\eps^4 k^2 \rho v^1v^1\Big) \Big]
& = \Big[ r(r-2M) \eps^2\Big( (1 + \eps^2 k^2) \rho v^0v^1 \Big) \Big],
\\
s \Big[ r(r-2M) \Big( (1 + \eps^2 k^2) \rho v^0 v^1 \Big) \Big] 
& = 
\Big[ r(r-2M) \eps^2\Big(k^2 \rho  v^0 v^0+(1-2M/r)^{-2}\rho v^1v^1\Big) \Big].
\endaligned
$$
where, in our notation, the bracket $\big[\Phi \big] := \Phi - \Phi_i$ denotes the jump  a quantity $\Phi$. 
Eliminating $s$, we find
$$
\left [ (k^2\eps^2+1) \rho v^0v^1 \right]^2=\left[( k^2 \eps^4 \rho v^1v^1+ (1-2M/r)^2\rho v^0v^0 \right ] \left[ k^2 \rho v^0 v^0+ {\rho v^1v^1 \over (1-2M/r)^2}\right].
$$
On the other hand, a straighforward calculation gives 
$$
\aligned
 0= \, & k^2  \rho^2 ( (1-2M/r)v^0 v^0 -(1-2M/r)^{-1}\eps^2 v^1 v^1)^2 \\
 & + k^2  \rho_i^2 ((1-2M/r)v^0 v^0 -(1-2M/r)^{-1}\eps^2 v^1 v^1)^2 \\
& -2k^2  \rho \rho_i ( (1-2M/r)v^0 v^0 -(1-2M/r)^{-1}\eps^2v^1 v^1)^2+2(k^2\eps^2+1)^2 \rho\rho_i v^0 v^0_iv^1 v^1_i \\
& -k^4  \eps^4\rho \rho_i v^0_i v^0_i v^1 v^1-2 \eps^2 k^2  \rho \rho_i v^1 v^1_i v^1 v_i^1- \rho \rho_i v^0 v^0v^1_i v^1_i \\
& -k^4  \eps^4\rho \rho_i v^0 v^0v^1_i v^1_i-2 \eps^2 k^2  \rho \rho_i v^0 v^0_i v^0v^0_i- \rho \rho_i v^0_i v^0_i v^1 v^1.
\endaligned   
$$
Using the fact that the velocity vector is unit, that is, $(1-2M/r)v^0 v^0 -(1-2M/r)^{-1}\eps^2 v^1 v^1=1$, we find 
$$
0= k^2 (\rho^2+\rho_i^2) -2 k^2   \rho\rho_i-(1 + \eps^2 k^2)^2\rho  \rho_i (v^0_iv^1-v^0 v^1_i)^2. 
$$
Thus, we have arrived at an equation for the density ratio $\rho \over \rho_i $:
\bel{shock1}
0= \Big({\rho \over \rho_i} \Big)^2 - \Bigg(2+ {(1 + \eps^2 k^2)^2 (v^0_iv^1-v^0 v^1_i)^2   \over k^2} \Bigg){\rho \over \rho_i}+1.
\ee

Furthermore, we obtain 
${1-2 \eps \nu \over 1+2 \eps \nu} =-\eps(1-2M/r)^{-1}{v^1 \over  v^0}$ in view of the definition of the velocity variable $\nu$ in \eqref{def-nu} and, therefore, 
\bel{eq:410}
\aligned
(v^0_iv^1-v^0 v^1_i)^2   =& {(1-2M/r)^{-2}({v^1\over v^0} - {v^1_i\over v^0_i})^2 \over \Big(1-(1-2M/r)^{-2}({\eps v^1\over v^0})^2  \Big) \Big(1-(1-2M/r)^{-2}({\eps v^1_i\over v^0_i})^2 \Big)} \\ 
=& {({1-2 \eps \nu \over 1+2 \eps \nu} - {1-2 \eps \nu_i \over 1+2 \eps\nu_i})^2 \over  \eps^2 \Big(1-({1-2 \eps \nu \over 1+2 \eps \nu})^2  \Big) \Big(1-({1-2 \eps \nu_i \over 1+2 \eps \nu_i})^2 \Big)} = {1\over 4 \eps^2} \Big( \sqrt{\nu \over \nu_i} - \sqrt{\nu_i \over \nu} \Big)^2.
\endaligned
\ee
Now, plugging \eqref{eq:410} into \eqref{shock1}, we find 
$$
{\rho \over \rho_i} =1+ {1\over 2 \chi^2} \Big( \sqrt{\nu \over \nu_i} - \sqrt{\nu_i \over \nu} \Big)^2\pm \sqrt{{1\over 4(\chi)^4} \Big( \sqrt{\nu \over \nu_i} - \sqrt{\nu_i \over \nu} \Big)^4+ {1\over \chi^2} \Big( \sqrt{\nu \over \nu_i} - \sqrt{\nu_i\over \nu} \Big)^2}, 
$$
or 
$$
\Bigg(\sqrt{\rho \over \rho_i} -\sqrt{\rho_i\over \rho} \Bigg)^2
=
{1 \over \chi^2} \Big( \sqrt{\nu \over \nu_i} - \sqrt{\nu_i \over \nu} \Big)^2.
$$
We have $v < v_L$ for $1$-shock, so we take the minus sign to guarantee that $\nu < \nu_L$. The analysis of the $2$-shock curve is similar.

\vskip.3cm 

2. With the Rankine-Hugoniot relations, we obtain
$$
s^2=\Big(1- {2M\over r} \Big)^2 \Big({\rho \over \rho-\rho_i} {v^2 \over 1 - \eps^2 v^2} + {k^2 \over 1 + \eps^2 k^2} \Big) \big/ \Big({\rho \over \rho-\rho_i} {\eps^2 v^2 \over 1 - \eps^2 v^2} + {\eps^2 k^2 \over 1 + \eps^2 k^2} \Big). 
$$
Lax's shock inequalities require that 
\bel{ineq-Lax}
\lambda(U_L) > s_1 > \lambda(U), \qquad \mu(U) > s_2 > \mu(U_R),
\ee
which provide us with the relevant signs for both characteristic families. 
\end{proof}


\subsection*{Wave curves and wave interaction estimates} 

Combining shock waves and rarefaction waves together, we are able to construct the solution to the Riemann problem, as follows. 
Given a left-hand state $U_L$ and a right-hand state $U_R$, by concatenating the two types of curves above, we define the {\bf $1$-wave curve} and the {\bf $2$-wave curve}, respectively, by 
\bel{eq:525}
\aligned 
& W_1^{\rightarrow}(U_L) : = S_1^{\rightarrow}(U_L) \cup R_1^{\rightarrow}(U_L), 
\qquad
W_2^{\leftarrow}(U_R) : = S_2^{\leftarrow}(U_R) \cup R_2^{\leftarrow}(U_R). 
\endaligned 
\ee
The following observation are in order: 

\bei 

\item Observe that in the special case that the initial states satisfy $U_R  \in R_1^{\rightarrow}(U_L)$ or $U_L  \in R_2^{\leftarrow}(U_R)$, then the Riemann problem can be solved by a single rarefaction wave. In this case, each state $U$ in the solution lie between $U_L$ and $U_R$ along the corresponding rarefaction curve and the associated propagation speed is $\lambda(U)$ and $\mu(U)$, respectively.

\item Similarly, in the special case that $U_R  \in S_1^{\rightarrow}(U_L)$ or $U_L  \in S_2^{\leftarrow}(U_R)$, the Riemann solution consists of a single shock propagating at the speed given by \eqref{shockspeed}.

\item  Moreover, it can be checked that the curves $S_1$ and $S_2$ are tangent up to second-order derivatives with the corresponding integral curves. Consequently, the wave curves $W_1^{\rightarrow}(U_L)$ and $W_2^{\leftarrow}(U_R) $ are of class $C^2$.

\eei 

\noindent Furthermore, according to  \eqref{Rarefac-curves} and  \eqref{Shock-curves}, the density component $\rho$ is increasing (from $-\infty$ to $+\infty$) along the wave curve $W_1^{\rightarrow}(U_L)$, while it is decreasing (from $+\infty$ to $-\infty$) along the wave curve $W_2^{\leftarrow}(U_R)$. This implies that the velocity component $\nu$ is increasing along $W_1^{\rightarrow}(U_L)$, and is decreasing along $W_2^{\leftarrow}(U_R)$. 

To proceed, we need the following technical lemma.

\begin{lemma}
\label{intersect}
The $1$-shock curve $S_1^{\rightarrow}(U_L)$ satisfies 
\bel{eq:530}
0 \leq {d z \over d w} = - {\sqrt{2 \beta+\chi^2 \beta^2} -\sqrt{2 \beta+ \beta^2} \over -\sqrt{2 \beta+\chi^2 \beta^2} -\sqrt{2 \beta+ \beta^2}}
<1, 
\ee
while the $2$-shock curve $S_2^{\leftarrow}(U_R)$ satisfies 
\bel{eq:532}
 {d z \over d w} = - {-\sqrt{2 \beta+\chi^2 \beta^2} -\sqrt{2 \beta+ \beta^2} \over \sqrt{2 \beta+\chi^2 \beta^2} -\sqrt{2 \beta+ \beta^2}}
>1, 
\ee
with the notation 
$$
\beta =\beta (U,U_i) := {1\over 2 \chi^2} \Bigg(\sqrt{\nu \over \nu_i} - \sqrt{\nu_i \over \nu} \Bigg)^2. 
$$
 \end{lemma}

\begin{proof} The  Riemann invariants read 
\bel{eq:535}
\aligned
w = {1\over 2\eps}\ln \Bigg({1 + \eps v\over 1 - \eps v} \Bigg) - {k \over 1 + \eps^2 k^2} \ln \rho,
\\
z = {1\over 2\eps}\ln \Bigg({1 + \eps v\over 1 - \eps v} \Bigg) + {k \over 1 + \eps^2 k^2} \ln \rho.
\endaligned
\ee 
By introducing the functions $g_\pm(\beta) =1+\beta \Big(1\pm\sqrt {1+ {2 \over \beta}} \Big)$, it is straightforward to see that 
\bel{eq:540}
{\rho \over \rho_i} =g_\pm(\beta). 
\ee
Moreover, we can check that $g_+(\beta)g_-(\beta) =1$. By Lemma~\ref{avshock}, we have
$2\beta \chi^2=\Big( \sqrt{\nu \over \nu_i} - \sqrt{\nu_i \over \nu} \Big)^2$ and, therefore,
\bel{eq:543}
{\nu \over \nu_i} =1+ {\chi^2\beta \over 2} \Big(1\pm \sqrt{1+ {4 \over \chi^2\beta}} \Big) =g_\pm(\chi^2\beta/2).
\ee
For definiteness, we consider $1$-shocks. The tangent to the shock curve $S_1^{\rightarrow}(U_L)$ in the  $(w,z)$-plane satisfies
$$
{dz \over dw} = {d(z-z_L) \over d(w -w_L)} = {d(z-z_L) \over d \beta}{d \beta \over d(w -w_L)}.
$$
Plugging \eqref{eq:540} and \eqref{eq:543} into the expression of the Riemann invariants \eqref{eq:535}, we obtain (for $1$-shocks) 
\bel{eq:545}
\aligned
& w-w_L= {1\over 2\eps} \Big( \ln g_+(\beta) + \chi \ln g_+(\chi^2\beta/2) \Big),
\\
& z-z_L= - {1\over 2\eps} \Big( \ln g_+(\beta) - \chi \ln g_+(\chi^2\beta/2) \Big),
\endaligned
\ee
and, thus,
$$
{dz \over dw} = {d(z-z_L) \over d(w -w_L)} = - {\sqrt{2 \beta+\chi^2 \beta^2} -\sqrt{2 \beta+ \beta^2} \over -\sqrt{2 \beta+\chi^2 \beta^2} -\sqrt{2 \beta+ \beta^2}}.
$$
Since $\chi = {2 \eps k   \over 1 + \eps^2 k^2}<1 $, we have $0 \leq{dz \over dw} <1$. 
\end{proof}

We have arrived at the main result of the present section. 

\begin{proposition}[The Riemann problem for fluid flows] 
\label{Riemannpro}
The homogeneous Euler system \eqref{Eulerform-homo} supplemented with Riemann initial data \eqref{eq:500} 
admits an entropy weak solution for arbitrary initial data $r_0 >2M$, $U_L= (\rho_L,v_L)$, and $U_R= (\rho_R,v_R)$. 
This solution depends on the self-similarity variable $r/t$, only, and is picewise smooth: it consists of two (shock or rarefaction) waves separated by constant states. 
\end{proposition}

\begin{proof} Consider the intersection of the two curves state $W_1^\rightarrow (U_L) \bigcap W_2^\leftarrow(U_R)$. Thanks to  Lemma~\ref{intersect} and our analysis above, the family of $1$-curves and $2$-curves covers the whole region in such a way that, for any given data $U_L, U_R$, the curves $W_1^\rightarrow (U_L)$ and  $W_2^{\leftarrow} (U_R)$ admit precisely one intersection point $U_M$. The Riemann solution is then solved by a $1$-wave connecting from $U_L$ to $U_M$, followed by a $2$-wave connecting from $U_M$ to $U_R$.
\end{proof}

Next, we define the {\bf total wave strength} of the Riemann solution connecting three states $U_L, U_M, U_R$ by
\bel{wavestrength}
\str(U_L, U_R) := |\ln \rho_L-\ln \rho_M |+ |\ln \rho_R-\ln \rho_M |,
\ee
where, by definition, $U_M= (\rho_M, v_M) $ is the intermediate state $\big\{ U_M \big\} = W_1^\rightarrow (U_L) \bigcap W_2^{\leftarrow}(U_R) $.  The following observation is the key in order to establish the global existence theory for the Cauchy problem. 

\begin{proposition}[Diminishing total variation property] 
\label{interaction}
Given three constant states $U_L$, $U_*$, and $U_R$, consider the associated Riemann problems $(U_L,U_*)$, $(U_*,U_R)$, and $(U_L,U_R)$. Then, the total wave strengths satisfy the inequality 
\bel{eq:565}
\str(U_L, U_R) \leq\str(U_L, U_*) + \str(U_*, U_R). 
\ee
\end{proposition}

\begin{proof}

We consider the wave curves in the $(w,z)$-plane of the Riemann invariants. Recall that, in this plane, rarefaction curves are straigthlines, while shock curves are described by explicit formulas. Importantly, the shock curves have the same geometric shape independently of the base point $U_L$ or $U_R$ and are described by the functions $g_\pm(\beta)$. 
Moreover, by observing the 
remarkable algebraic property  $g_+(\beta)g_-(\beta) =1$, we see that the $2$-shock curve is the symmetric of the $1$-shock curve with respect to the straightline $z =w$ (in the $(w,z)$-plane). 
Note that the strength $\str$ does not change at interactions involving two rarefaction waves of the same family, only. 
Since the wave strengths, by definition, are measured in $w-z \sim \ln \rho$, these symmetry properties imply that the wave strengths are non-decreasing at interactions. 
\end{proof}
 

\section{The generalized Riemann problem} 
\label{sec:7}

\subsection*{Discontinuous steady states}

Our strategy is now to solve the Riemann problem for the full Euler model, by replacing the two initial constant states by two equilibrium solutions. We refer to this problem as the {\sl generalized Riemann problem.} In order to proceed, we need first to revisit our analysis in  (cf. Section~\ref{sec:5} and to introduce first global equilibrium solutions, defined for all radius $r \in (2M, +\infty)$ and possibly containing
a jump discontinuity. This is necessary since some (smooth) steady state solutions are defined on a sub-interval of $r>2M$, only; this happens when the velocity component may reach the sonic value $\pm k$. 

Recall that, according to Theorem~\ref{steady-relativistic}, two possible behavior may arise, which are determined by the sign of the function $P_\eps(r_0, v_0)$ defined in \eqref{function-P-eps}, that is, 
$$
\aligned
P_\eps(r_0, v_0)
=&
 \ln \Bigg( {(2 - \kappa)^2 \over (1- \kappa)^2} {M^2 k \over r_0^2 v_0} \Bigg)
+ {\kappa \over 1- \kappa} \ln \Bigg( {2(2 - \kappa) \over 1 +\kappa} {(r_0-2M) \over  r_0(1-\eps^2 v_0^2)} \Bigg)
\\
=& \ln \Bigg( {(1+3\eps^2 k^2)^2 M^2 \over 4 \eps^4  k^3 r_0^2 v_0} \Big))
+ {\kappa \over 1- \kappa} \ln\Bigg( {(1+3\eps^2 k^2) (r_0-2M)  \over  r_0(1-\eps^2 v_0^2)} \Bigg). 
\endaligned
$$
The function $P_\eps$ beign regular, the existence of a sonic point depends also continuously upon the data $r_0, v_0$. 
Recall also that in the special case that the data satisfy $P_\eps(r_0, v_0) = 0$, then the associated two sonic points $\barr_*$ and $\underr_*$ are both equal to (cf. our notation \eqref{not349} and \eqref{nota-85}) 
\be
r_* := {2 - \kappa \over 1 - \kappa} M, 
\ee
which we refer to as the {\bf critical sonic point.} 
We now consider this limiting case, which was excluded in our earlier analysis. 

\begin{proposition}[The global construction for sonic initial data]
\label{one-point}
When the initial data $r_ 0>2M$, $\rho_0 > 0$, and $0 \leq v_0 < 1/\eps$ satisfy the {\bf sonic condition}
\be
P_\eps(r_0, v_0) = 0, 
\ee
then the steady Euler system \eqref{steady2-con} admits a {\bf global steady state solution} $\rho=\rho(r)$ and $v=v(r)$ satisfying the initial conditions $\rho(r_0) =\rho_0$ and $v(r_0) =v_0$ and such that $v(r) - k$ changes sign precisely once.  
\end{proposition}

\begin{proof} Without loss of generality, we assume that $r_0 \geq r_*$. According to Theorem~\ref{steady-relativistic}, there exists a smooth steady state solution defined on the interval $(r_*, +\infty)$. At any radius $r \in (2M, r_*)$, we have 
$$
G(r, v; r_0, v_0) < G(r, k; r_0, v_0) <G(r_*, k; r_0,v_0) = 0. 
$$ 
Therefore, for a given $r \in (2M, r_*)$, the equation $G(r,v; r_0,v_0) = 0$ admits two roots: $v^\flat(r) < k$ and $v^\sharp(r) > k$. Moreover, $v^\flat (r_*-) = v^\sharp (r_*-) = k$, and these solutions are continuous at the sonic point $r=r_*$. We caould in principle define two continuous steady state solutions, but we want to make a unique selection. At the sonic point, the derivative of the solution blows-up to infinity, and it is natural to keep the sign of the derivative. Hence, for the initial data under consideration satisfying
$P_\eps(r_0, v_0) = 0$, we define a continuous global steady state solution by setting  
\bel{global-special}
\hatv (r; r_0, \rho_0, v_0) = 
\begin{cases}
v(r; r_0, \rho_0, v_0), \qquad  & r \in \Pi, 
\\
v^\aleph(r; r_0, \rho_0, v_0),     & r \notin \Pi,
\end{cases}
\ee
in which we have selected $v^\aleph= v^\flat$ if $v_0 > k$ while $v^\aleph= v^\sharp$ if $v_0 <  k$. Hence, the function $v(r) - k$   {\sl changes sign} when we reach the sonic point.  
\end{proof} 

We now turn our attention to general data, when two sonic points $\underr_* <\barr_*$ are available. We can no longer ``cross'' the sonic velocity value, while by remaining within the class of continuous solutions. Instead, we must consider solutions with one shock , as we now explain it. 

\begin{lemma}[Jump conditions for steady state solutions] 
\label{steady-shock}
A steady state discontinuity associated with left/right-hand limits $(\rho, v)$ and $(\rho_i, v_i)$ must satisfy  
\bel{jump-steady}
{\rho\over \rho_i} = {1-\eps^2  k^4 /v_i^2 \over  1 - \eps^2 v_i^2} {v_i^2 \over k^2}, 
\qquad v \, v_i = k^2. 
\ee
\end{lemma}

\begin{proof} From the Rankine-Hugoniot relations 
$\Big[ {1 + \eps^2 k^2 \over 1 - \eps^2 v^2} \rho v \Big] = 0$
and $\Big[ {v^2+ k^2 \over 1 - \eps^2 v^2} \rho \Big] = 0$, we deduce that 
$$
\aligned
& {1 \over 1 - \eps^2 v^2} \rho  v= {1 \over 1 - \eps^2 v_i^2} \rho_i v_i,
\qquad
{v^2+ k^2 \over 1 - \eps^2 v^2} \rho= {v_i^2+ k^2 \over 1 - \eps^2 v_i^2} \rho_i, 
\endaligned
$$
which we solve for $\rho$ and $v$. 
\end{proof} 

In view of Lemma~\ref{steady-shock}, there exist infinitely many discontinuous steady state solutions containing a shock discontinuity at some radius $r_1\in \Pi$. At such a point, we have  
\be
\rho(r_1 \pm) 
: = {1-\eps^2  k^4 /v(r_1 \pm)^2 \over  1 - \eps^2 v(r_1 \pm)^2} {v(r_1 \pm)^2\over k^2} \rho(r_1 \pm),
\qquad  v(r_1 \pm) := {k^2\over v(r_1 \pm)}. 
\ee
Of course, by introducing a shock within a steady state solution, we must guarantee that the new branch of solution allows us to make a global continuation in the sense that we are not limited again by a sonic point. In fact, in order to have also a unique construction, we propose to select the jump point so that the new branch of solution has the ``sonic property'' discussed above. 
The following lemma provides us with the key observation.  

\begin{lemma}[Existence and uniqueness of the critical jump radius] 
\label{steady-with-shock}
Consider a smooth steady state solution $\rho=\rho(r;r_0,\rho_0,v_0)$ and $v=v(r;r_0,\rho_0,v_0)$, which is defined on the interval $\Pi$ and satisfying the initial condition $\rho(r_0) =\rho_0$ and $v(r_0) =v_0$.  
The, if this solution admits a sonic point (which is denoted by $\underr_*$ or $\barr_*$), then then there exists a unique radius, referred to as the {\bf critical jump radius} and denoted by $r_1^*\in \Pi$, such that 
\bel{eq:228}
P_\eps \Big(r_1^*, {k^2\over v_1^*} \Big) = 0 \qquad \text{ with } v_1^*= v(r_1^*; r_0,\rho_0, v_0). 
\ee
Moreover, $r_1^*$ lies in the interval limited by $r_0$ and the sonic point. 
\end{lemma}

\begin{proof} First of all, it is straightforward to check that $P_\eps(r_1, v_1)$ is increasing in $r_1$ on $\Big( 2M, {(1+3\eps^2 k^2 ) M \over 2 \eps^2  k^2} \Big)$ and is decreasing on $\Big({(1+3\eps^2 k^2) M \over 2 \eps^2  k^2}, + \infty \Big)$. It is also decreasing in $v_1$ on $(0, k)$ and decreasing on $(k, 1/ \eps)$.

Let us first establish the existence of a radius satisfying the condition \eqref{eq:228}. 
In the regime under consideration, we have two sonic points and $P_\eps(\barr_*, k) < 0$ and $P_\eps (\underr_*, k) < 0$. 
For definiteness in the discussion, we treat the following case
$$
{2 - \kappa \over 1 - \kappa} M < 
\barr_* < r_0, \qquad k < v_0. 
$$
Thanks to the above monotonicity property of $P_\eps$ with respect to $v$, we can find a neighborhood of $\barr_*$, denoted by $U_*$, such that the inequality $P_\eps\Big(r_1, {k^2\over v_1} \Big) < 0$ holds for all $r_1 \in U_*$.  
For every $(r_1, v_1)$ along the steady solution curve starting from $(r_0, v_0)$, the condition $P_\eps(r_1, v_1) < 0$ holds. By introducing $\tildeM(r) = {(1+3\eps^2 k^2) M \over 2 \eps^2  k^2  r} = {2 - \kappa \over 1 - \kappa} {M \over r}$, we can rewrite the condition $P_\eps(r_1, v_1) < 0$ as 
$$
 2\ln \tildeM (r_1) + \ln\Bigg( {k \over v_1} \Bigg) - {\kappa \over 1- \kappa} \ln \big(1-\eps^2 v_1^2 \big) 
 + {\kappa \over 1- \kappa}  \ln \Bigg(  {2(2 - \kappa) \over 1 +\kappa} - 4\eps^2 k^2 \tildeM (r_1) \Bigg) \leq 0. 
$$
We need to show that there exists some point $r_1$ such that $P_\eps(r_1, {k^2 \over v_1}) > 0$.
 Indeed, by setting  $\barM_*= {2 - \kappa \over 1 - \kappa} {M \over \barr_*}$ and $\underM_*= {2 - \kappa \over 1 - \kappa} {M \over \underr_*}$, we have 
$$
\aligned 
&P_\eps \Big(r_1, {k^2 \over v_1} \Big) 
\geq P_\eps \Big(r_1, {k^2 \over v_1} \Big) + P_\eps ( r_1, v_1) 
\\
& \geq 4 \ln \tildeM (r_1)
           -  {\kappa \over 1 - \kappa} \ln\Big( (1-\eps^2 k^4/ v_1^2)(1-\eps^2 v_1^2) \Big)
  + {2 \kappa \over 1-\kappa}  \ln \Bigg( {2(2 - \kappa) \over 1 + \kappa} - {2(1-\kappa) \over 1+\kappa}  \tildeM (r_1) \Bigg)  
\\ 
& \geq -   {\kappa \over 1 - \kappa} \ln \Bigg((1-\eps^2 k^4/ v_1^2)(1-\eps^2 v_1^2) \Bigg)
+ \max \Bigg(
4 \ln \underM_*, {2 \kappa \over 1 - \kappa}  \ln \Bigg(  {2(2 - \kappa) \over 1 +\kappa} - {2(1-\kappa) \over 1+\kappa} \barM_*\Bigg) 
\Bigg).
\endaligned 
$$
Since $- {\kappa \over 1 - \kappa} \ln \Big( (1-\eps^2 k^4/ v_1^2)(1-\eps^2 v_1^2) \Big) 
\in 
\Big( - {\kappa \over 1- \kappa} \ln (1-\eps^4 k^4), +\infty \Big)$, we can find an interval of $v_1$ where $ P_\eps(r_1, {k^2 \over r_1}) > 0$. By continuity, we conclude that there exists a radius $r_1^*$ such that $P_\eps\Big(r_1^*, {k^2\over v_1^*} \Big) = 0.$ 
 
Now, we turn to the uniqueness of $r_1^*$. We want to show that $P_\eps\Big( r_1, {k^2\over v_1} \Big)$ changes its sign only once along the steady state curve. Recall that we assume (for deifniteness) that $r_0 > {2 - \kappa \over 1 - \kappa} M$, so that the smooth solution is defined on $(\barr_*, + \infty)$. Let $r_1$ be a point such that  $P_\eps(r_1, {k^2 \over r_1}) > 0$. For $r > r_1$, according to the monotonicity properties of steady state solutions, we have $|k- v(r)| > |k - v_1|$, therefore, ${P_\eps \Big(r, {k^2 \over v(r)} \Big) > 0}$ always holds.
Then let $r_2$ be a point such that $P_\eps (r_2, {k^2 \over v_2}) < 0$ holds. For $r \in (\barr_*, r_2)$, according to the monotonicity properties of $P_\eps$, we have 
$$
\aligned
P_\eps \Bigg(r, {k^2 \over v} \Bigg) =  & 2\ln \tildeM(r) + \ln \Bigg( {v(r) \over k} \Bigg) 
     - {\kappa \over 1- \kappa} \ln \big(1-\eps^2 k^2 / v(r)^2 \big) 
     \\
&  + {\kappa \over 1- \kappa}  \ln \Big(1+3 \eps^2 k^2 -4\eps^2 k^2 \tildeM(r) \Big) 
< P_\eps\Bigg(r_2, {k^2 \over v_2} \Bigg) < 0. 
\endaligned 
$$
We have thus established that $P_\eps$ changes sign only once. 

Moreover, let us emphasize that $r_1^*$ lies in the interval limited by $r_0$ and the sonic point. Again, we treat the case $r_0>{2- \kappa \over 1-\kappa} M $.  If the desired property would not hold, then we would have $P_\eps(r_0, v_0) < 0$ and $P_\eps(r_0, {k^2\over v_0})<0$ simultaneously, but this would contradict 
$$
\aligned
P_\eps(r_0, v_0)+ P_\eps(r_0, {k^2 \over v_0})
>
&  - {\kappa  \over 1- \kappa}\ln \Big((1-\eps^2 k^4/ v_1^2 )(1-\eps^2 v_1^2) \Big) 
\\
& + {\kappa  \over 1- \kappa } \ln \Bigg( {4 (2-\kappa)^2 \over (1+\kappa)^2} \Bigg)  >0. 
\endaligned
$$
Therefore, we have $\bar r_*<r_1^*<r_0$ in the case under consideration. 
\end{proof}

From the family of smooth steady states $\rho = \rho(r; r_0, \rho_0, v_0)$ and $v = v(r; r_0, \rho_0, v_0)$ in the regime where they admit a sonic point, we are now in a position to define solutions on the whole interval $r \in (2M,+ \infty)$. We introduce the domains 
$$
\Lambda_s := 
\begin{cases}
[r_1^*, +\infty),  \qquad  & r_1^* \geq {2 - \kappa \over 1 - \kappa} M, 
\\
(2M, r_1^*),                    & r_1^* < {2 - \kappa \over 1 - \kappa} M,
\end{cases}
\qquad
\qquad 
\Lambda_d := 
\begin{cases}
(2M, 	r_1^*],       \qquad  & r_1^*\geq {2 - \kappa \over 1 - \kappa} M, 
\\
(r_1^*, +\infty),               &  r_1^* < {2 - \kappa \over 1 - \kappa} M.
\end{cases} 
$$
We arrive at our main conclusion in this section. 

\begin{theorem}[Globally-defined steady state solutions] 
\label{theo-constrc}
Consider the family of smooth steady state solutions to the Euler system posed on a Schwarzschild background with black hole mass $M$.  
Given any radius $r_0 > 2M$, initial density $\rho_0 > 0$, and initial velocity $|v_0| < 1/\eps$, the initial value problem for the steady Euler system \eqref{steady2-con} with initial condition $\rho(r_0) = \rho_0$  and $v(r_0) =v_0$ admits a {\rm unique weak solution} which is globally 
defined for all $r \in (2M, +\infty)$ and contains at most one shock (satisfying the Rankine-Hugoniot relations and Lax's shock inequalities), 
and such that the velocity component $|v|-k$ changes sign at most once.  Furthermore, the family of steady state solutions with possibly one shock depends Lipschitz continuously upon its arguments $r_0, \rho_0, v_0$ when they vary within the {\rm whole range} of solutions, encompassing smooth solutions with no sonic point, continuous solutions with exactly one sonic point, and discontinuous solutions containing exactly one continuous sonic point and one shock crossing a sonic point. 
\end{theorem}

A precise statement of the continuity property above is as follows: in the case of a solution containing a shock, it is meant that the location of the shock and its left- and right-hand limit vary continuously; moreover, in the transition from a solution of one of three types to a solution of another type, the values taken by the solution vary continuously. 

We have derived all the ingredients in order to establish the theorem above. First of all, for the case without sonic point, smooth solutions defined for all $r \in (2M, +\infty)$ were already constructed in Section~\ref{sec:5}, so that to shock is required when a branch of solution never reaches a sonic point. 

Now consider the case with sonic points. The critical case where the two sonic points coincide is already dealt with in Proposition~\ref{one-point}. So, it remains to discuss the case $\underr_* < \barr_*$. 
Let  $\rho_1^*= \rho(r_1^*;r_0,\rho_0,v_0)$, $v_1^*=v(r_1^*;r_0,\rho_0,v_0)$ be values achieved by the smooth solution at the critical jump point $r_1^*$ provided by Lemma~\ref{steady-with-shock}. In view of Lemmas~\ref{steady-shock} and  \ref{steady-with-shock}, we can now introduce the (discontinuous) {\bf global steady state solution}
as  
\bel{Def-sol}
v (r; r_0, \rho_0, v_0)
:=
\begin{cases} 
v(r; r_0, \rho_0, v_0), \qquad  & r \in  \Lambda_s, 
\\
\hatv\Big(r; r_1^*, {1-\eps^2  k^4 /{v_1^*}^2 \over  1 - \eps^2 {v_1^*}^2}{{v_1^*}^2\over  k^2} \rho_1^*, {k^2\over v_1^*} \Big), 
             & r \in \Lambda_d, 
\end{cases}
\ee
where $\hatv$ is the corresponding steady state solution containing a (unique) sonic point (cf.~\eqref{global-special}). 
According to Lemma~\ref{steady-shock}, the Rankine-Hugoniot relations hold along the discontinuity so that, for any smooth function with compact support $\phi=\phi(r)$, 
$$
\aligned 
& \int_{2M}^{+\infty} \big(F(r, U(r)) {d \over dr} \phi(r) + S(r, U(r)) \phi(r) \big) \, dr
\\
& = \Bigg(\int_{2M}^{r_1^*}+\int_{r_1^*}^{+\infty} \Bigg)
 \Big(- {d \over dr}  F(r, U) + S(r, U)) \phi(r) \Big) \, dr 
-
\Big(F(r_1^*, U(r_1^*+)) - F(r_1^*, U(r_1^*-) \Big) \phi(r_1^*) = 0. 
\endaligned 
$$
Therefore, \eqref{Def-sol} defines a weak solution to the Euler equations in the distributional sense. Moreover, Lax's shock inequalities are satisfied by construction. Indeed, without loss of generality, suppose that $r_0 > {2 - \kappa \over 1- \kappa} M$ and let us use the notation $U_L, U_R$, where $U_R$ is the smooth steady flow. We have either a $1$-shock wave if $v_R>k$ or a $2$-shock wave if $v_R<k$. We treat, for instance, the case $v_R<k$. The two eigenvalues read (after using the jump relations \eqref{jump-steady}) 
$$
\mu(U_R)=\left(1- {2M\over r}\right){v_R+k  \over 1+\eps^2  k v_R},  
\qquad \mu(U_L)= \left(1- {2M\over r}\right){k^2/v_R+k  \over 1+\eps^2  k^3/ v_R},  
$$
while the shock speed is  
$$
\aligned
s(U_L, U_R) =  \left(1- {2M\over r}\right) 
& \Bigg(
{k^2\over v_R^2/k^2 -\eps^2 k^2-1+ \eps^2 v_R^2}+ { k^2 \over 1+\eps^2 k^2}
\Bigg)^{1/2}
\\
& \cdot \Bigg({\eps^2  k^2\over v_R^2/k^2 -\eps^2 k^2-1+ \eps^2 v_R^2}+ { 1 \over 1+\eps^2 k^2} \Bigg)^{-1/2}.  
\endaligned
$$
In fact, there is no new calculation to do here since, by construction, we have chosen $v_L > v_R$ for a $2$ shock and, consequently as observed in our study of general $2$-shock curves, Lax's shock inequalities $\mu(U_L) > s(U_L, U_R) > \mu(U_R)$ hold. For $1$-shock waves, a similar argument gives $\lambda(U_L) > s(U_L, U_R) > \lambda(U_R)$.

For the continuity property, we observe that the regularity is obvious for the family of smooth solutions and we only need to consider the continuous solutions and the discontinuous solutions, as well as the transitions from one case to another. Let us consider first continuous solutions that, by construction, cross the sonic value $k$ at the critical radius. We claim that such solutions $r \mapsto v(r)$ are {\sl Lipschitz continuous} everywhere (except at $r=2M$ where they may blow-up, but later on we will first exclude a neighborhood of the horizon). Namely, we only need to check this property at the critical sonic point: from \eqref{first-der}, we can compute the derivative at the point 
$r_* = {2 - \kappa \over 1 - \kappa} M$ and obtain 
$$ 
\aligned
{d v\over dr} (r_*)
&
\simeq {k \over r_* (r_* - 2M)}  \lim_{r \to r_*} 
{{1- \kappa \over \kappa} (r-2M) - M \over{\eps^2 v^2 \over 1 - \eps^2 v^2} - {1- \kappa \over 2 \kappa}}
\\
& \simeq {k \over r_* (r_* - 2M)} 
{1-\eps^2 k^2 \over 2 \eps^2 k} {1- \kappa \over \kappa} /{d v\over dr} (r_*)
\simeq {k \over r_* (r_* - 2M)}  k/{d v\over dr} (r_*)
\endaligned
$$
and, consequently, the derivative is finite and is given by 
\be
{d v\over dr} (r_*) = \pm {k \over \big( r_* (r_* - 2M)\big)^{1/2}}, 
\ee
the sign depending upon the choice of the branch. This shows that the continuous branch is Lipschitz continuous. The same calculation is valid to deal with discontinuous solutions and shows that, way from the jump discontinuity, the solution depends Lipschitz continuously. In the transition from discontinuous to continuous solutions, the strength of the jump discontinuity shrinks to zero, while the base point $r_0$ approaches the critical point $r_*$. All derivatives remain finite in this limit. Finally the transition from a continuous to a smooth solution is regular away from the sonic point (located at $r_*$), while at the critical point $r_*$ we have a jump of the derivative which is a non-vanishing constant (for continuous solutions) and which vanishes for smooth solutions. Yet, the derivative remains bounded, and we still have the Lipschitz continuity property. This completes the proof of Theorem~\ref{theo-constrc}.  
  
 
\subsection*{A generalized Riemann solver}  
\label{GRP-Constrcution}

The Riemann solver defined in Section~\ref{sec:6} for the homogeneous system is now extended to the full Euler model \eqref{Eulerform} with source term $S(U,r)$. The Riemann solution no longer depends solely on ${r-r_0 \over t-t_0}$ and is no longer given by a closed formula. In particular, wave trajectories are no longer  straigthlines. We are going to construct an {\sl approximate solver,} which will have sufficient accuracy in order to establish our existence theory. Precisely, we consider the {\bf generalized Riemann problem} which, by definition, is based on two steady state solutions separated by a jump discontinuity, that is, 
\bel{Eulerfo}
\del_ t U+ \del_r F(U,r) = S(U,r), \qquad t > t_0, 
\ee   
\bel{initialgeneral}
U(t_0, r) = 
U_0(r) :=
\begin{cases}
U_L(r), \qquad &  r < r_0,
\\
U_R(r),          &  r > r_0,
\end{cases} 
\ee
posed at the point $t_0 \geq 0$ and $r_0 > 2M$, in which the functions $U_L=U_L(r)$ and $U_R=U_R(r)$ are two (global) steady state solutions, that is, weak solutions to the ordinary differential system
\be
{d\over dr} F(U,r) = S(U,r),
\ee 
constructed in Section~6. The exact solution to the generalized Riemann problem, denoted here by $U=U(t,r)$, cannot be determined explicitly, and we thus seek for an {\sl approximate solution}, which we will denote by $\tildeU = \tildeU(t,r)$. 

First of all, we can follow the discussion in Section~\ref{sec:6} and we solve the (classical) Riemann problem posed at the point $(t_0, r_0)$ for the homogeneous Euler system, that is, by denoting this solution by $U^c(t,r; t_0,r_0)$, we have 
\be
\del_ t U^c+ \del_r F(r_0,U^c) = 0, \qquad t \geq t_0, 
\ee
\be
U^c_0(r) =
\begin{cases}
U_L^0 := U_L (r_0), \quad  &  r < r_0,
\\
U_R^0  := U_R(r_0),             & r > r_0.
\end{cases}
\ee 
We know that the solution $U^c$ depends upon $\xi := {r-r_0 \over t-t_0}$, only, and consists of three constant states $U_L^0, U_M^0,  U_R^0$, separated by shock waves or rarefaction waves.
For all sufficiently small times $t>t_0$, the solution to the generalized Riemann problem is expected to remain sufficiently close to the solution of the classical Riemann problem.  

Next, let us introduce the (possibly discontinuous) steady state solution $U_M=U_M(r)$ determined in Theorem~\ref{theo-constrc} from the initial condition $U_M(r_0) =: U_M^0$ at $r_0$. 
For the following discussion, it is convenient to set   
\be
U_0 := U_L^0, \qquad U_1 := U_M^0, \qquad U_2 := U_R^0.
\ee
We also set $s_j^- = \lambda(U_{j-1})$ and $\mu(U_{j-1})$, 
and $s_j^+= \lambda(U_j)$ or $\mu(U_j)$ (for $j=1,2$) be the lower and upper bounds of the speeds in the $j$-rarefactions. If the $j$-wave is a shock, then $s_j^-=s_j^+=s_j$ denotes the $j$-shock speed (given by \eqref{shockspeed}). 

We are now ready to define the {\bf approximate generalized Riemann solver} by setting  
\bel{approximate-sol}
\tildeU(t, r): =
\begin{cases} 
U_L(r), \qquad  & r-r_0 <s_1^-(t-t_0),
\\
V_1(t,\eta_1(t,r)),           &  s_1^-(t-t_0) <r-r_0 <s_1^+(t-t_0),
\\
U_M(r),            &  s_1^+(t-t_0) <r-r_0 <s_2^-(t-t_0),
 \\
V_2(t,\eta_2(t,r)),         &  s_2^-(t-t_0) <r-r_0 <s_2^+(t-t_0),
\\
U_R(r),          &  r-r_0 >s_2^+(t-t_0),
\end{cases}
\ee
in which we have also introduced (in the case that the classical Riemann problem admits rarefactions) the functions
$V_j=V_j(t, \eta_j)$ and the change of variable $(t,r) \mapsto (t,\eta_j)$ given by the following integro-differential problem. Following Liu \cite{Liu2}, we take into account the time-evolutionof the generalized Riemann solution  within a rarefaction fan and define ``approximate rarefaction fans'', as follows. We first seek for $V_j=V_j(t, \eta_j)$ and $r^\sharp=r^\sharp (t, \eta_j)$ as functions of the time variable $t$ together with a new variable denoted by $\eta_j$, satisfying  
\bel{curved-rare}
\aligned
&  \del_ {\eta_j} r^\sharp \del_t V_j + \Big( \del_U F(V_j) - \lambda_j(V_j) \Big) \del_{\eta_j}  V_j = S(V_j)  \, \del_ {\eta_j} r^\sharp, 
\\
& \del_t r^\sharp = \lambda_j(V_j), 
\endaligned
\ee
with the following boundary and initial conditions (with $\eta_j^0= \lambda_j(U_ {j-1}(r_0))$)  
\be
\aligned
& V_j(t, \eta_j^0) =U_ {j-1}(r^\sharp(t, \eta_j^0)), \qquad V_j(t_0, \eta_j) =h_j( \eta_j),
\\
&  \del_t r^\sharp (t, \eta_j^0) = \lambda_j(U_ {j-1}(r^\sharp)), \qquad r^\sharp (t_0,\eta_j) =r_0, 
\endaligned 
\ee
where the function $h_j$ is defined by inverting the eigenvalue  functions along the rarefaction curves, i.e. 
\be
\lambda_j(h_j(\xi)) = \xi = {r-r_0 \over t-t_0}.
\ee  
(As usual, $\lambda_1=\lambda$ and $\lambda_2=\mu$). Next, we recover the ``standard'' radial variable $r$ by setting 
$$
r = r^\sharp (t, \eta_j),
$$  
and, therefore, expressing $\eta_j$ as a function of $(t,r)$. We now check that the conditions above define a unique function.    

\begin{lemma}
\label{discrepancy-t}
For sufficiently small times $\Delta t$, there exists a unique smooth solution of the problem \eqref{curved-rare} defined within the time interval $t_0 < t< t_0+\Delta t$, such that 
$$
\del_t V_j =O(1)G, \qquad \del_{\eta_j} V_j= h_j'(\eta_j) + O(1)G \Delta t, 
$$
where $G$ is a constant independent of $t$ and $r$. 
\end{lemma}

\begin{proof} Let us, for instance, treat the rarefaction waves of the first family $j=1$ and derive first an integral formulation of the problem.  Denoting by $l_1, l_2$ two independent left-eigenvectors of the Jacobian matrix of the Euler system,  we have 
\bel{tildeV}
\aligned 
\mathcal D \tildeV_2 &=  {\del_{\eta_2} r^\sharp \over \mu- \lambda} l_2 \cdot S + \mathcal D l_2 \cdot V_1,
\\
\del_t \tildeV_1 &= l_2\cdot S+ \del_t l_2 \cdot V_1,
\endaligned
\ee
where we write $\tildeV_1= l_1 \cdot V_1$ and $\tildeV_2 =  l_2 \cdot V_1$, and we have also introduced the differential operator  
$\mathcal D: = {\del_{\eta_2} r^\sharp \over \mu- \lambda}\del_t + \del_{\eta_2}$, 
whose integral curves starting from $(s, \lambda(U_0) )$ are denoted by $\mathcal L$. 
By integrating \eqref{tildeV}, we thus obtain 
\bel{T}
\aligned
& \tildeV_2 (t, \eta_1)
=
 \tildeV_2 (s, \lambda(U_0) )+ \int _{\mathcal L} \Bigg(
 {\del_{\eta_2} r^\sharp \over \mu- \lambda} l_2 \cdot S
+ \mathcal D l_2 \cdot V_1 \Bigg) \, d\eta_1,
\\
& \tildeV_1 (t, \eta_1) =\tildeV_1(t_0, \xi) + \int _{t_0}^t  \Big( l_2\cdot S+ \del_t l_2 \cdot V_1 \Big) \, d\eta_1.  
\endaligned 
\ee
Now we define  an operator $T$ to provide the right-hand side of \eqref{T} and we take an arbitrary function $V_1^0 $ such that $V_1^0(t,\eta_1^0)=  V_1(t,\eta_1^0)$ and $ V_1^0(t_0,\eta_1)=V_1(t_0,\eta_1)$.  We then study the iteration scheme  $V_1^{(l)} := T^{(l)} V_1^0$. For all sufficiently small $\Delta t$, the operator $T$ is contractive in the sup-norm of $V_1^0$ and their first-order derivatives, by a standard fixed point argument we deduce that there exists a unique solution $V_1$ to \eqref{T}. Moreover, by integration, we can estimate the first-order derivatives of $V_1$, as stated in the lemma. 
\end{proof}

We define the {\bf wave trajectories} as  
\be
r_j^\pm(t) := s_j^\pm (t-t_0) +r_0
\ee
and, in particular, if the $j$-wave is a shock, we have $r_j(t) := r_j^\pm(t) =s_j (t-t_0) +r_0$. 

\begin{lemma}[Control of the error associated with the generalized Riemann solver] 
\label{approximate-gen}
Let $\tildeU$ be the approximate generalized Riemann solver defined by \eqref{approximate-sol}. Then, for all $t_0 \leq t < t_0+ \Delta t$, one has: 
\begin{enumerate}

\item When $(U_{j-1}, U_j)$ is a $j$-shock wave,  one has 
\bel{approximate-shock}
\aligned
& s_j  \Big( \tildeU(t, r_j(t) +) -\tildeU(t, r_j(t) -) \Big) 
\\
& =F(r_j(t),\tildeU(t, r_j(t) +)) -F(r_j(t), \tildeU(t, r_j(t) -)) + O(1)|U_j-U_{j-1} |\Delta t. 
\endaligned 
\ee
 
\item When $(U_{j-1}, U_j)$ is a $j$-rarefaction wave, one has 
\bel{approximate-rare}
\tildeU(t,r_j^+(t)) -\tildeU(t, r_j^-(t)) = O(1)|U_j-U_{j-1}| \Delta t. 
 \ee
\end{enumerate}
\end{lemma}

Consequently, when there is no jump at $r_0$, that is, $U_L(r) =U_R(r)$, then the term $|U_j - U_{j-1}| $ vanishes, and the approximate solution is, in fact, {\sl exact.}
 
\begin{proof}
By our construction, if $(U_{j-1}, U_j)$ is a shock wave, then we simply connect $U_{j-1}(r),$ and $U_j(r))$ by a jump discontinuity. A Taylor's expansion yields us 
$$
\tildeU(t,r_j^+(t)) -\tildeU(t, r_j^-(t)) =U_j-U_{j-1}+ O(1)|U_j-U_{j-1} | \Delta t,
$$
and, thanks to the Rankine-Hugoniot relations $s_j (U_j-U_{j-1}) =F(r,U_j) -F(r,U_{j-1})$, we arrive at 
$$
\aligned 
& s_j  \Big( \tildeU(t, r_j(t) +) -\tildeU(t, r_j(t) -) \Big) 
\\
&  F(r_j(t),\tildeU(t, r_j(t) +)) -F(r_j(t), \tildeU(t, r_j(t) -))
 + O(1)|U_j-U_{j-1} |\Delta t. 
 \endaligned 
$$
If, now, $(U_{j-1}, U_j)$ is a rarefaction wave, it follows from our construction that 
$$
\aligned
U_j ( r_j^+(t)) -U_{j-1} ( r_j^+(t)) =U_j-U_{j-1}+ O(1)|U_j-U_{j-1} |\Delta t. \\
U_j ( r_j^-(t)) -U_{j-1} ( r_j^-(t)) =U_j-U_{j-1}+ O(1)|U_j-U_{j-1} |\Delta t. 
\endaligned
$$
Moreover, we have 
$r_j^+(t) -r_{j-1}^-(t) =O(1)|U_j-U_{j-1}|\Delta t$ and a Taylor's expansion yields us 
$$
U_{j-1}(r_j^+(t)) -U_{j-1} ( r_j^-(t)) =O(1)|U_j-U_{j-1}|\Delta t. 
$$
Hence, we can compute 
$$
\aligned 
& \tildeU(t,r_j^+(t)) -\tildeU(t, r_j^-(t)) =U_j(r_j^+(t)) -U_{j-1}(r_j^-(t))
\\
&= U_j ( r_j^+(t)) -U_{j-1}(r_j^+(t)) -\big(U_j ( r_j^-(t)) -U_{j-1} ( r_j^-(t)) \big) + U_{j-1}(r_j^+(t)) -U_{j-1} ( r_j^-(t)) 
\\
&= O(1)|U_j-U_{j-1}| \Delta t. \qedhere
\endaligned
$$
\end{proof}

In order to estimate whether the function $\tildeU$ is an ``accurate'' approximate solution, we consider any smooth function $\phi=\phi(t,r)$ and study the integral expression 
\bel{theta}
\Theta(\Delta t, \Delta r;\phi) 
:= \int_{t_0}^{t_0+\Delta t} \int_{r_0-\Delta r}^{r_0+\Delta r}
\Big( \tildeU \, \del_t \phi + F(r,\tildeU) \, \del_r \phi+S(r, \tildeU) \phi \Big) \, dr dt 
\ee
for any $\Delta t, \Delta r>0$ with $r_0-\Delta r>2M$. 
Observe that $\Theta$ would vanish if we would take the exact Riemann solution $U=U(t,r)$ in \eqref{theta} and we would assume that the function is compactly supported in the slab under consideration. The expression \eqref{theta} provides a measure of the discrepancy between the exact and the approximate solutions, and can be expressed from integrals on the boundary of the slab, modulo an error term, as we now show it. 

\begin{proposition}
\label{discrepancy-}
For given  $\Delta t, \Delta r > 0 $ with $r_0-\Delta r>2M$ satisfying the stability condition 
\bel{stability-condition}
{\Delta r\over \Delta t} > \max (-\lambda, \mu),
\ee
and for every smoth function $\phi$ defined on $[t_0, +\infty) \times [{r_0-\Delta r}, {r_0+\Delta r}]$, 
the integral expression in \eqref{theta} satisfies
\bel{discrepancy}
\aligned
\Theta(\Delta t, \Delta r;\phi)
=
& \int_{r_0-\Delta r}^{r_0+\Delta r}\tildeU (t_0+\Delta t, \cdot) \phi (t_0+\Delta t, \cdot) \, dr- \int_{r_0-\Delta r}^{r_0+\Delta r}\tildeU (t_0, \cdot) \phi (t_0, \cdot) \, dr\\
& +\int_{t_0}^{t_0+\Delta t} F(r_0+\Delta r, \tildeU (\cdot, r_0+\Delta r)) \phi ( \cdot, r_0+\Delta r) \, dt \\
& -\int_{t_0}^{t_0+\Delta t} F(r_0-\Delta r, \tildeU (\cdot, r_0-\Delta r)) \phi ( \cdot, r_0-\Delta r) \, dt\\ 
& + O(1)  |U_R(r_0) -U_L(r_0) | \Delta t^2  \| \phi \| _{\mathcal C^1}. 
\endaligned 
\ee
\end{proposition}

\begin{proof}
We decompose the sum under consideration as 
$$
\Theta(\Delta t, \Delta r;\phi) =\sum_j \int_{t_0}^{t_0+\Delta t}\int_{D^1_j} \theta(t,r) \, dr dt + \sum_j \int_{t_0}^{t_0+\Delta t}\int_{D^2_j} \theta (t,r) \, dr dt, 
$$
where 
$$
D_0^1 := (r_0-\Delta r, r_1^ -(t)), 
\quad 
D_1^1= (r_1^ +(t), r_2^ -(t)), 
\quad 
D_0^1= (r_2^ +(t), r_0+\Delta r), 
\quad 
D^2_j= (r_j^ -(t), r_j^ +(t)),
$$
for $j=1,2$ which is used to denote the rarefaction regions. 
We first consider the interval $D_j^1$ where the approximate solution $\tildeU$ is a steady state solution. Therefore, we have $\del_t \tildeU +\del_r \tilde F(r, \tildeU) - S(r, \tildeU) = 0$ in $D_j^1$. Multiplying the equation by the test-function $\phi$ and integrating by parts, we obtain 
$$
\aligned 
\int_{t_0}^{t_0+\Delta t}\int_{D^1_0} \theta(t,r) \, dr dt = 
&  \int_{r_0-\Delta r}^{r_1^-(t)}\tildeU (t_0+\Delta t, r) \phi (t_0+\Delta t, r) \, dr - \int_{r_0-\Delta r}^{r_0}\tildeU (t_0, r) \phi (t_0, r) \, dr
 \\& +  \int_{t_0}^{t_0+\Delta t} \Big(F(r_1^-(t), \tildeU (t,r_1^-(t) -)) -s_1^- \tildeU (t,r_1^-(t) -) \Big) \phi(t,r_1^-(t)) \, dt 
\\ &   -\int_{t_0}^{t_0+\Delta t} F(	r_0-\Delta r, \tildeU (t,r_0-\Delta r)) \phi(t,r_0-\Delta r) \, dt
\endaligned
$$
and 
$$
\aligned 
\int_{t_0}^{t_0+\Delta t}\int_{D^1_2} \theta(t,r) \, dr dt = 
&  \int_{r_2^+(t)}^{r_0+\Delta r}\tildeU (t_0+\Delta t, r) \phi (t_0+\Delta t, r) \, dr - \int_{r_0}^{r_0+\Delta r}\tildeU (t_0, r) \phi (t_0, r) \, dr
\\ &   +\int_{t_0}^{t_0+\Delta t} F(	r_0+\Delta r, \tildeU (t,r_0-\Delta r)) \phi(t,r_0+\Delta r) \, dt
 \\& -  \int_{t_0}^{t_0+\Delta t} \Big(F(	r_2^+(t), \tildeU (t,r_2^+(t) +)) -s_2^+ \tildeU (t,r_2^+(t) +) \Big) \phi(t,r_2^+(t)) \, dt.
\endaligned
$$
A similar calculation for the integration in $D_1^1$ gives us: 
$$
\aligned 
\int_{t_0}^{t_0+\Delta t}\int_{D^1_1} \theta(t,r) \, dr dt = 
&  \int_{r_1^+(t)}^{r_2^-(t)}\tildeU (t_0+\Delta t, r) \phi (t_0+\Delta t, r) \, dr
 \\& +  \int_{t_0}^{t_0+\Delta t} \Big(F(	r_2^-(t), \tildeU (t,r_2^-(t) -)) -s_2^- \tildeU (t,r_2^-(t) -) \Big) \phi(t,r_2^-(t)) \, dt 
  \\& - \int_{t_0}^{t_0+\Delta t} \Big(F(	r_1^+(t), \tildeU (t,r_1^+(t) +)) -s_1^+ \tildeU (t,r_1^+(t) +) \Big) \phi(t,r_1^+(t)) \, dt. 
\endaligned
$$

Next, consider the rarefaction region $D_j^2$. According to the construction in \eqref{approximate-sol} and \eqref{curved-rare}, we have $U(t,r) = V_j(t,\eta_j)$ in $D_j^2$. Performing the change the variable $(t,r) \to (t, \eta_j)$, we have (with the notation $\lambda_1=\lambda$ and $\lambda_2=\mu$) 
$$
\aligned 
\del_t U + \del_r F(U,r) -S (U,r)
= \, & \del_t V_j-  \lambda_j (V_j) \del_{\eta_j} V_j  \del_r \eta_j +  \del_{\eta_j} F \del_r \eta_j - S(V_j)
\\
= \, & \del_r \eta_j ( \del_t V_j  \del_ {\eta_j} r+(\del_U F- \lambda_j) \del_{\eta_j}  V_j-S  \del_ {\eta_j} r) = 0. 
\endaligned 
$$
Multiply the equation by the test function $\phi$, then for the rarefaction region, we have
$$
\aligned 
\int_{t_0}^{t_0+\Delta t}\int_{D^2_j} \theta (t,r) \, dr dt=&  \int_{D^2_j}   \tildeU (t_0+\Delta t, r) \phi (t_0+\Delta t, r) \, dr 
 \\& +  \int_{t_0}^{t_0+\Delta t} \Big(F(	r_j^+(t), \tildeU (t,r_j^+(t) -)) -s_j^+\tildeU (t,r_j^+(t) -) \Big) \phi(t,r_j^+(t)) \, dt 
\\ &   -\int_{t_0}^{t_0+\Delta t} \Big(F(	r_j^-(t), \tildeU (t,r_j^-(t) +)) - s_j^-  \tildeU (t,r_j^-(t) +) \Big) \phi(t,r_j^-(t)) \, dt. 
\endaligned
$$
According our construction of the generalized Riemann problem, if $(U_{j-1}, U_j)$ is a shock, \eqref{approximate-shock} gives 
$$
 \aligned
& s_j  \Big( \tildeU(t, r_j(t) +) -\tildeU(t, r_j(t) -) \Big) 
\\
& =F(r_j(t),\tildeU(t, r_j(t) +)) -F(r_j(t), \tildeU(t, r_j(t) -)) + O(1)|U_j-U_{j-1} | \Delta t. 
 \endaligned 
$$
Hence, we have
$$
\aligned
&  \int_{t_0}^{t_0+\Delta t} \Big(F(	r_j(t), \tildeU (t,r_j(t) +)) -s_j \tildeU (t,r_j(t) +) \Big) \phi(t,r_j(t)) \, dt 
\\ &-  \int_{t_0}^{t_0+\Delta t} \Big(F(	r_j(t), \tildeU (t,r_j (t) -)  -s_j\tildeU (t,r_j(t) -) \Big) \phi(t,r_j(t)) \, dt 
 =O(1)|U_R-U_L | \Delta t^2   \| \phi \| _{C^0}. 
\endaligned
$$

According to \eqref{approximate-rare}, if $(U_{j-1}, U_j)$ is a rarefaction wave, we have 
$$
 \tildeU (t,r_j^+(t) +) -  \tildeU (t,r_j^+(t) -) = O(1)|U_j-U_{j-1} | \Delta t, 
$$
from which we obtain 
$$
\aligned
& \int_{t_0}^{t_0+\Delta t} \Big(F(	r_j(t), \tildeU (t,r_j^+(t) +)) -s_j \tildeU (t,r_j^+ (t) +) \Big) \phi(t,r_j(t)) \, dt 
\\ 
& -  \int_{t_0}^{t_0+\Delta t} \Big(F(	r_j(t), \tildeU (t,r_j^+ (t) -)  -s_j\tildeU (t,r_j^+(t) -) \Big) \phi(t,r_j(t)) \, dt 
\\ 
& = O(1) \big( U (t,r_j^+(t) +) - U (t,r_j^+j(t) -) \big) \Delta t \| \phi \| _{C^0}
 = O(1) |U_R-U_L | \Delta t^2 \| \phi \| _{C^0}. 
\endaligned
$$
Adding all the terms together, we thus estimate the discrepancy as  
$$
\aligned
& \Theta(\Delta t, \Delta r;\phi)
\\
& =\int_{r_0-\Delta r}^{r_0+\Delta r}\tildeU (t_0+\Delta t, \cdot) \phi (t_0+\Delta t, \cdot) \, dr- \int_{r_0-\Delta r}^{r_0+\Delta r}\tildeU (t_0, \cdot) \phi (t_0, \cdot) \, dr
\\
& \quad +\int_{t_0}^{t_0+\Delta t} F(r_0+\Delta r, \tildeU (\cdot, r_0+\Delta r)) \phi ( \cdot, r_0+\Delta r) \, dt 
\\
& \quad -\int_{t_0}^{t_0+\Delta t} F(r_0-\Delta r, \tildeU (\cdot, r_0-\Delta r)) \phi ( \cdot, r_0-\Delta r) \, dt 
+ O(1)  |U_R(r_0) -U_L(r_0) | \Delta t^2  \| \phi \| _{\mathcal C^1}. \qedhere
\endaligned 
$$
\end{proof}


\section{The initial value problem} 
\label{sec:8}

\subsection*{The global existence theory}

We now consider the {\bf initial value problem} for the Euler system on a Schwarzschild background, that is, \eqref{Eulerform}--\eqref{Eulerform24}, with some initial condition at $t_0 \geq 0$ 
\be
(\rho, u)(t_0, \cdot) = (\rho_0, v_0)
\ee
for some prescribed data $\rho_0: (2M, +\infty) \to (0, +\infty)$ and $v_0: (2M, +\infty) \to (-1/\eps, +1/\eps)$. 
Before we introduce our method based on steady states, we first observe that the technique already developed by Grubic and LeFloch \cite{GL} (in a different geometric setup) applies, which is based on a piecewise constant approximation and an ODE solver. This method applies to general initial data and solutions. 

\begin{theorem}[Global existence theory for fluid flows on a Schwarzschild background]
\label{global-existence}
Consider the Euler system describing fluid flows on a Schwarzschild background \eqref{Euler-con} posed in $r>2M$.  Given any initial density $\rho_0=\rho_0(r) > 0$ and velocity  $|v_0| = |v_0(r)|< 1/\eps$ satisfying, for any $\delta> 0$, 
$$
TV_{[2M+\delta, +\infty)} \big(\ln \rho_0 \big) + TV_{[2M+\delta, +\infty)}  \Bigg( \ln {1-\eps v_0 \over 1 + \eps v_0} \Bigg) < + \infty, 
$$  then there exists a weak solution $\rho=\rho(t,r)$ and $v=v(t,r)$ defined on $[t_0, +\infty)$ and satisfying the prescribed initial data at the time $t_0$ and such that, for all finite time $T \geq t_0$ and $\delta >0$, 
$$
\sup_{t \in [t_0, T]} \Bigg( 
TV_{[2M+\delta, +\infty)} \big( \ln \rho(t,\cdot) \big) 
+ TV_{[2M+\delta, +\infty)} \Bigg( \ln {1-\eps v (t, \cdot) \over 1 + \eps v(t, \cdot)} \Bigg) \Bigg) < +\infty.  
$$
\end{theorem}

 For the proof, we only need to observe that no boundary condition is required at $r=2M$, since the wave speeds vanish on the horizon and that we can always ``cut'' an arbitrarily small neighborhood of the horizon and estimate the total variation outside this neighborhood, as explained in the following subsection. We omit the details.


\subsection*{Behavior near the horizon}

In view of Lemma~\ref{Riemann-inva}, the eigenvalues  
$$
\aligned
& \lambda=-\Bigg(1- {2M\over r} \Bigg){v-k\over1-\eps^2 k v},
\qquad 
\mu=\Bigg(1- {2M\over r} \Bigg){v+k\over1 + \eps^2  k v}
\endaligned
$$
are distinct for all $r>2M$ but both of them vanish on the horizon $r=2M$. This indicates that no boundary condition should be required on the horizon. On the other hand, the Euler system \eqref{Euler-con} is not strictly hyperbolic at the horizon $r=2M$. Yet, for any given $\delta> 0$, the system is strictly hyperbolic in the region $r\geq 2M+\delta$.

Furthermore, recall from Section~5 that steady state solutions may ``blow-up'' near the horizon, in the sense that the velocity component $v$ may approach $\pm 1/\eps$, which does correspond to an algebraic singularity for the Euler system.  

It follows that it is natural to study the Cauchy problem, first, away from the horizon within a domain of dependence where the solution is uniquely determined from the prescribed initial data. Observe that, 
according to Lemma~\ref{Riemann-inva}, the eigenvalues are uniformly bounded: 
$$
\aligned
& - {1 \over \eps} < - {1 \over \eps}  \Bigg(1- {2M\over r} \Bigg) < \lambda <  \mu <  {1 \over \eps} \Bigg(1- {2M\over r} \Bigg) 
<  {1 \over \eps}
\qquad &&&\text{ when } r>2M,
\\
& \lambda =  \mu = 0 \qquad &&&\text{ when } r=2M. 
\endaligned
$$
This provides us with a uniform a priori control on the wave speed, so that the stability condition required in the random choice method is automatically satisfied (without having to derive first a uniform sup-norm estimate). 
 
We thus fix $\delta>0$ and consider the curve $r=\barr_\delta(t)$ characterized by 
\be
{d \barr \over dt} (t) =  {1 \over \eps}  \Bigg( 1- {2M\over \barr(t)}\Bigg), 
\qquad \barr(0) = 2M+\delta, 
\ee
which, in the limit of vanishing $\delta$, converges to the line $r=2M$, in the sense that 
\be
\lim_{\delta \to 0} \barr(t) = 2M \quad \text{ uniformly for $t$ in a compact subset of } (2M, +\infty].   
\ee
In the following, we study the initial value problem with data prescribed at some time $t_0 \geq 0$, and we state first our BV estimate within the region $\Omega_\delta(T) = \big\{ t_0 < t < T, \, r > \barr(t)\big\}$. In turn, by letting $\delta \to 0$, we are able to control the total variation in every compact subset in $(t,r)$.  
  

\subsection*{A random choice method based on equilibria} 
\label{Glimm} 

 We are now ready to develop a theory based on steady state solutions as a building blocks, which has the advantage of preserving equilibria and allows to establish the nonlinear stability of equilibria.
Our approach is based on the approximate solver of the generalized Riemann problem provided in Section~\ref{sec:7}. Use $\tildeU(t,r;t_0,r_0,U_L(r),U_R(r))$ to denote the approximate solver of the generalized Riemann problem at $(t_0, r_0)$ with initial steady states $U_L(r)$ and $U_R(r)$ at $t=t_0$ separated at $r=r_0$ provided in Section~\ref{sec:6}. 
Denote the mesh lengths in $r$ and $t$ by $\Delta r$ and $\Delta t$ respectively, and $(t_i, r_j)$ the mesh point of the grid:
$$
t_i =t_0+i\Delta t,
 \qquad r_j=2M+  j\Delta r.
$$
Since $-\lambda, \mu < 1/\eps$, we assume 
${\Delta r \over \Delta t}> {2\over \eps}$
to guarantee the stability condition \eqref{discrepancy-}. Interactions can thus be avoided within one step. 
First of all, we approximate the initial data $U_0$ by a piecewise steady state profile determined from the initial condition at $r=r_ {j +1}$: 
\be
\aligned
& {d\over dr}  F(r,U_\Delta (t_0,r)) = S(r,U_\Delta (t_0,r)), \qquad \text{ $j$ even, \, $r_j<r < r_{j +2}$},\\
& U_\Delta (t_0,r_{j +1}) =U_0(r_ {j +1}). 
\endaligned
\ee

We set 
$$
r_{i,j} =2M+(w_i+j) \Delta r,
$$
where $(w_i)_i$ is a given random sequence in $(-1,1)$. If the approximate solution $U_\Delta  $ has been defined for all $t_{i-1}\leq t< t_i$, we define $U_\Delta (t,r)$ for all $r$ and $t_i\leq t< t_{i+1}$, as follows: 
\begin{enumerate}

\item At the time level $t=t_i$, we define $U_\Delta $ to be the piecewise smooth steady solution given by solving
\bel{piecewise}
\aligned
&{d \over dr} F(r,U_\Delta (t_i,r)) = S(r,U_\Delta (t_i,r)), \quad \text{ $i+j$ even, \quad  $r_j<r < r_{j +2}$,}\\
& U_\Delta (t_i,r_{i, j +1}) =U_\Delta (t_i-,r_{i,j +1}). 
\endaligned
\ee

\item Now define $U_\Delta  $ on  $t_i< t< t_{i+1}$: 
\begin{enumerate}

\item[] For $j\geq 1$, define the solution on $\{t_i< t< t_{i+1},  r_{j-1} < r < r_{j +1}\}$ (with $i+j$ even) by  
$$
U_\Delta (t,r):= \tildeU\Big(t,r; t_i,r_j,\barU_L(r_j), \barU_R(r_j) \Big)
$$
with $\barU_L(r) =U_\Delta (t_i,r)$, $r \in ( r_{j-1}, r_j)$ and  $\barU_R(r) =U_\Delta (t_i,r)$, $r \in (  r_j,r_{j +1})$ the steady state components of $U_\Delta (t_i,r)$. 
 
\end{enumerate}
\end{enumerate}

This completes the definition of the approximate solution $U_\Delta  $ on $ [t_ 0, +\infty) \times (2M, +\infty)$.


\subsection*{Wave interactions of the generalized Riemann problem} 

In Proposition~\ref{interaction}, we studied wave interactions in the context of the classical Riemann problem and established a monotonicity property. For the generalized Riemann problem under consideration now, the initial data is no longer piecewise constant and we need to revisit this issue.  Given a  pattern consisting of three (possibly discontinuous) steady state solutions $U_L=U_L(r)$, $U_M= U_M(r)$, and $U_R= U_R(r)$, we are interested in the solution to the Euler system \eqref{Eulerfo} with Cauchy data (with $r_1<r_0 <r_2$ given)
\bel{three-states} 
U_0(r) =
\begin{cases}
U_L(r),  \qquad & r < r_1, 
\\
U_M(r),            & r_1<r < r_2,
\\
U_R(r),            & r > r_2, 
\end{cases} 
\ee
and we want to compare it with the solution with Cauchy data \eqref{initialgeneral}, that is, 
\bel{initialgeneral-29}
U_0(r) =
\begin{cases} 
U_L(r),  \qquad & r < r_0,
\\
U_R(r),             & r > r_0. 
\end{cases} 
\ee 
The following statement in a generalization of Proposition~\ref{interaction}, which corresponds to the special case $r_1=r_2=r_0$. We restrict attention to continuous steady states. (A generalization to discontinuous steady states could possibly be established too, by including the strength of the steady shock.) 

\begin{proposition}[Diminishing total variation property for the generalized Riemann problem] 
\label{general-interaction}  
Suppose that all steady state under consideration are continuous. 
The wave strengths  
associated with radii  $r_1<r_0 <r_2$ and three steady state solutions $U_L=U_L(r)$, $U_M = U_M(r)$, and $U_R= U_R(r)$ to the Euler system \eqref{Eulerform}. Then, one has 
\be
\str(U_L(r_0-), U_R(r_0+) \leq \Big(\str(U_L(r_1-), U_M(r_1+)) + \str(U_M(r_2-), U_R(r_2+)) \Big) \big(1+ O(1)(r_2-r_1) \big). 
\ee
\end{proposition}

\begin{proof} 
Consider first smooth steady state solutions (which do not contain shocks). Since solutions to an ordinary differential system depend continuously upon their data, it is immediate that 
$$
U_L(r_1) -U_*(r_1) =U_L(r_0) -U_*(r_0) + O(1)(r_0-r_1) |U_L(r_1) - U_*(r_1) |
$$
and, since $|U_R- U_L|=O(1) \str(U_L, U_R)$, we obtain  
$$
\str(U_L(r_1), U_*(r_1)) =\str (U_L(r_0), U_*(r_0)) \Big((1+ O(1)(r_0-r_1) \Big). 
$$
With the same argument, we have 
$$
\str(U_*(r_2), U_R(r_2)) =\str(U_*(r_0), U_R(r_0)) \Big(1+ O(1)(r_2-r_0) \Big)
$$
 and the conclusion follows for smooth equilibrium  solutions. For steady state solutions which are only continuous, we recall the conclusion in Theorem~\ref{theo-constrc}, where we established a Lipschitz continuity property satisfied by global steady state solutions.  
\end{proof}

\subsection*{The existence theory based on equilibria}
 
The existence property below is established under the restriction that only continuous steady states are involved in the scheme. Dealing with discontinuous steady states require a further investigation of the interaction between steady shocks and Riemann solutions (which is outside the scope of the present paper).  

\begin{theorem}[The generalized random method based on equilibria]
\label{global-existence22}
Consider the Euler system describing fluid flows on a Schwarzschild background \eqref{Euler-con} posed in $r>2M$. The generalized random choice scheme above has the following properties: 

1. {\bf Convergence to a weak solution.} Given any initial density $\rho_0=\rho_0(r) > 0$ and velocity  $|v_0| = |v_0(r)|< 1/\eps$ satisfying, for any $\delta> 0$, 
$$
TV_{[2M+\delta, +\infty)} \big(\ln \rho_0 \big) + TV_{[2M+\delta, +\infty)}  \Bigg( \ln {1-\eps v_0 \over 1 + \eps v_0} \Bigg) < + \infty, 
$$
and provided on some (possibly infinite) interval $[t_0, T) \subset [t_0, +\infty)$, the generalized Riemann solver involves continuous steady states, only, then there exists a weak solution $\rho=\rho(t,r)$ and $v=v(t,r)$ defined on $[t_0, T)$ and satisfying the prescribed initial data at the time $t_0$ and such that, for all finite $T' \in [t_0, T)$ and $\delta >0$, 
$$
\sup_{t \in [t_0, T']} \Bigg( 
TV_{[2M+\delta, +\infty)} \big( \ln \rho(t,\cdot) \big) 
+ TV_{[2M+\delta, +\infty)} \Bigg( \ln {1-\eps v (t, \cdot) \over 1 + \eps v(t, \cdot)} \Bigg) \Bigg) < +\infty.  
$$

2. {\bf The well-balanced property for smooth steady states.} When the initial density $\rho_0=\rho_0(r) > 0$ and the initial velocity $|v_0|=|v_0(r)|< 1/\eps$ consist of a smooth steady state solution to \eqref{steady2-con}, the corresponding approximate solution to the Euler system \eqref{Euler-con} constructed by the proposed generalized random choice method (in Section~\ref{sec:7}) coincides with the given solution, so that our method provides the exact solution in this special case. 
 
3. {\bf The well-balanced property for discontinuous steady states.} 
Consider an initial data $U_0 = (\rho_0(r), v_0(r))$ with $\rho_0(r) > 0$ and $|v_0(r)|< 1/\eps$ of the following form  
\bel{eq-244} 
U_0(r) = 
\begin{cases}
U_L(r), \qquad  & r \in (2M, r^\natural),
\\
U_R(r),           & r \in (r^\natural, +\infty), 
\end{cases} 
\ee
where $r^\natural>2M$ is a given radius, 
$U_L=U_L(r)$ and $U_R=U_R(r)$ are global smooth steady solutions such that the states $U_L(r^\natural)$ and $U_R(r^\natural)$ satisfy the equilibrium Rankine-Hugoniot relations \eqref{jump-steady}. Then, the solution constructed by the generalized random choice method has, at each time, the same form \eqref{eq-244}, that is, a discontinuous steady state solution with possibly ``shifted'' location $r^\natural$. 
\end{theorem}

\begin{proof} {\bf Step 1a. Consistency of the method.} With the proposed generalized random method, we obtain a sequence $\{U_\Delta (t,r) \}$. Once the uniform BV bound (established below) is known, it follows from Helly's theorem that there exists a subsequence of  $\{U_\Delta (t,r) \}$ (still denoted by $\{U_\Delta (t,r) \}$) depending on the mesh length $\Delta r \to 0$ and a limit function $U=U(t,r)$ such that $U_\Delta \to U$ pointwise for all times $t$. 
To check that the limit function is a weak solution to the Euler system \eqref{Euler-con}, we consider a compactly supported and smooth function $\phi:   [t_ 0, +\infty) \times (2M, +\infty)  \to \RR$, and from the approximate solution $U_\Delta$ with mesh length $\Delta t, \Delta r$, we define
\bel{Delta}
\Delta(U_\Delta, \phi): = \int_{t_0}^{+\infty} \int_{2M+\Delta r}^{+\infty} \big(U_\Delta  \del_t \phi + F(r,U_\Delta) \del_r \phi +S(r,U_\Delta ) \big) \, dr dt+ \int_{2M+\Delta r}^{+\infty} U_0(r) \phi(t_0,r) \, dr. 
\ee
By definition, $U$ is a weak solution to the Euler system \eqref{Euler-con} with initial data $U(t_0,\cdot) =U_0$ 
if and only if $\Delta(U, \phi) = 0$. 
We write
$
\Delta(U_\Delta, \phi) = \sum_i \Delta_i^1(U_\Delta, \phi) +\Delta_i^2(U_\Delta, \phi)
$
with 
$$
\Delta_i^1(U_\Delta, \phi) = \int_{2M+\Delta r}^{+\infty} \Big(U_\Delta (t_i+,r) -U_\Delta (t_i-,r) \Big) \phi(t_i,r) \, dr, 
$$
$$
\Delta_i^2(U_\Delta, \phi) = \int_{t_i}^{t_{i+1}} \int_{2M+\Delta r}^{+\infty} \Big(U_\Delta  \del_t \phi + F(r,U_\Delta) \del_r \phi +S(r,U_\Delta ) \Big) \, dr dt+ \int_{2M+\Delta r}^{+\infty} U_0(r) \phi(t_0,r) \, dr.
$$
According to Proposition~\ref{discrepancy-}, $\sum_i  \Delta_i^2(U_\Delta, \phi) \to 0$ when $\Delta t \to 0$. 
Furthermore, it a standard matter that, since the sequence  $(w_i)$ is equidistributed and thanks to the approximation result in  Lemma~\ref{approximate-gen}, we have $\sum_i  \Delta_i ^1(U_\Delta, \phi) \to 0$ when $\Delta t \to 0$, and therefore 
$\Delta(U_\Delta, \phi) \to 0$ when $\Delta t,  \Delta r \to 0$.  

\

\noindent {\bf Step 1b. Uniform total variation bound.} 
Next, in order to study globally the total variation of the solution, we introduce the notion of  \emph{mesh curves} $J$, that is, polygonal curves connecting the points $(t_i, r_{i,j +1})$ (with $i+j$ even). Observe that $J$ separates $[t_0, +\infty) \times [2M, +\infty)$ into two parts: the part including the initial time $t=t_0$ denoted by $J-$ and the other part $J+$. We call $J_2$ an immediate successor of $J_1$ if the every point of $J_2$ is either on $J_1$ or in the part $J_1+$. 

For the mesh point, set
$$
U_\Delta (t_i, r_{j +1}) =U_{i,j +1}. 
$$
Denote by $\hat U_{i,j +1}$ as the solution of classical Riemann problem at the mesh point $(t_i, r_{j +1})$. 
We define the total variation $L(J)$ of $J$ as 
\bel{TVofJ}
L(J) =\sum\str (U_{i-1,j}, \hat U_{i,j-1}) + \str (U_{i-1,j},\hat U_{i,j +1}). 
\ee
Observe that we can divide the $(t,r)$ plane as a set of diamonds $\Diamond_{i,j}$ centered at $(t_i, r_j)$, $i+j$ even with vertices $(t_{i-1}, r_{i-1,j})$, $ (t_i, r_{i,j-1})$, $(t_{i}, r_{i,j +1})$. In particular, for $j=1$, $i$ odd, we only have a half diamond cut by the straightline $r=r_1$. 

Now, consider a diamond $\Diamond_{i,j}$, with $i+j$ even, and define 
$$
\str_1 (\Diamond_{i,j}) := \str (U_{i-1,j}, \hat U_{i,j-1}) + \str (U_{i-1,j},\hat U_{i,j +1})
$$
and 
$$
\str_2 (\Diamond_{i,j}) :=\str (U_{i,j-1}, \hat U_{i+1,j}) + \str (U_{i,j +1},\hat U_{i+1,j})
$$
which represent the total strength of waves entering and leaving the diamond $\Diamond_{i,j}$, respectively. We write $\trianglerightNEW_{i,1}$, with $i$ odd, for the right-hand part of the diamond   $\Diamond_{i,1}$ cut by the straightline $r= 2M$. We define similarly 
$$
\str_1 (\trianglerightNEW_{i,1}) := \str (U_{i-1,1},\hat U_{i,2}), 
\qquad 
\str_2 (\trianglerightNEW_{i,1}) := \str ( U_{i,2},\hat U_{i+1,1})
$$
which represent the total wave strength entering and leaving $\trianglerightNEW_{i,1}$, respectively. We now consider the total variation contribution``between'' the mesh curve $J_1$ and its immediate successor $J_2$. 

We now claim that: 
Let $J_1$ and $J_2$ be two mesh curves such that $J_2$ is an immediate successor of $J_1$. Then there exists a constant $C_1> 0$ such that the total variation on the mesh curves satisfies 
$$
L(J_2) -L(J_1) \leq C_1 (\Delta t + \Delta r) L(J_1). 
$$ 
Namely, suppose the mesh curve $J_1$ is sandwiched between the time levels $t_{i-1}$ and $t_i$. In view of \eqref{TVofJ}, we have
$$
L(J_2) -L(J_1) =\str_2 (\trianglerightNEW_{i,1}) -\str_1(\trianglerightNEW_{i,1}) + \sum_{\text {$i+j$ even}}\str_2(\Diamond_{i,j}) -\str_1(\Diamond_{i,j}). 
$$
Now consider the difference $\str_2(\Diamond_{i,j}) -\str_1(\Diamond_{i,j})$: 
$$
\aligned
\str_2(\Diamond_{i,j}) -\str_1(\Diamond_{i,j}) =& \str (U_{i,j-1}, \hat U_{i+1,j}) + \str (U_{i,j +1},\hat U_{i+1,j})
-\str ( \hat U_{i,j-1}, \hat U_{i,j +1}) \\
& + \str ( \hat U_{i,j-1}, \hat U_{i,j +1}) -\str (U_{i-1,j}, \hat U_{i,j-1}) -\str (U_{i-1,j},\hat U_{i,j +1}). 
\endaligned
$$
According to Proposition~\ref{general-interaction}, we have the inequality of the wave strength: 
$$
\str (U_{i-1,j}, \hat U_{i,j-1}) - \str (U_{i-1,j},\hat U_{i,j +1}) -\str ( \hat U_{i,j-1}, \hat U_{i,j +1}) \leq C_1 \Delta r \str ( \hat U_{i,j-1}, \hat U_{i,j +1}). 
$$
Using Lemma~\ref{discrepancy-t}, we have
$$
\aligned
& \str (U_{i,j-1}, \hat U_{i+1,j}) + \str (U_{i,j +1},\hat U_{i+1,j}) -\str ( \hat U_{i,j-1}, \hat U_{i,j +1})
\\
& 
\leq C_1  \str \big( \hat U_{i,j-1}, \hat U_{i,j +1}) 
 (|U_{i+1,j} -\hat U_{i+1,j}|
+|U_{i-1,j} -\hat U_{i-1,j}|) +C_1(|U_{i-1,j} -U_{i+1,j}|+|\hat U_{i-1,j} -\hat U_{i+1,j} \big) \\
& \leq C_1 \Delta t \Big( \str (U_{i-1,j},\hat U_{i,j-1}) + \str ( U_{i,j-1}, \hat U_{i,j +1}) \Big)
\endaligned 
$$
for some constants $C_1$ which need not be the same at each occurence. Therefore, we find 
$$
\str_2(\Diamond_{i,j}) -\str_1(\Diamond_{i,j}) \leq C_1 (\Delta t +\Delta r) \Big( \str (U_{i-1,j},\hat U_{i,j +1}) + \str ( U_{i-1,j}, \hat U_{i+1,j}) \Big), 
$$
and a similar analysis gives 
$\str_2 (\trianglerightNEW_{i,1}) -\str_1(\trianglerightNEW_{i,1}) \leq C_1 (\Delta t +\Delta r) \str ( \hat U_{i,2}, U_{i-1,1})$.

\

\noindent{\bf Step 1c. Convergence property.} Let $T>t_0$ be given, and let $J_0, J_T$ be the mesh curves lying below and above any other mesh curves between $t_0 \leq t \leq T$, respectively.  Thanks to Step~2, there exist uniform constants $C_2, C_3>0$ such that 
$$
L(J_T) \leq C_3e^{C_2(T-t_0)}L(J_0). 
$$
We now claim that for small $\Delta r$, the total variation of the approximate solver $\ln  \rho_ \Delta$ on the mesh curve $J$ can be regarded equivalent as the total wave strength $L(J)$. In fact, according to construction, for the mesh curve between $(t_i, t_{i+1})$, 
$$
 \aligned
& |TV( \ln  \rho_ \Delta(J)) -L(J) | =  \sum_\text{ $i+j$ even}|TV_{(r_j +, r_{j +2} -)}( \ln  \rho_ \Delta(r))|
 \\& =O(1) \sum_\text{ $i+j$ even} \Delta r   |\ln  \rho_ \Delta(t_i, r_{j +2} -) -\ln  \rho_ \Delta(t_i, r_{j}+)| \leq O(1) \Delta r  L(J). 
 \endaligned 
$$
Letting  $\Delta t, \Delta r \to 0$, we see that  $TV_{[2M+ \delta, L]} \big(\ln \rho(T,\cdot) \big) 
\leq
 C_3 TV_{[2M+ \delta, L]} \Big( \ln \rho_0 \Big)e^{C_2(T-t_0)}$ for any given $\delta> 0$ and $L> 0$. We have arrived at our main result stated in Theorem~\ref{global-existence}. 
 \end{proof} 

\

\noindent{\bf Step 2. The well-balanced property for smooth steady states.} 
2. We proceed by induction and assume that the numerical solution coincides with the steady state solution within the time interval $t_{i-1}\leq t< t_i$, and we consider the next interval $t_i\leq t < t_{i+1}$. In our method, the approximate solution is determined in two steps: 
$(\textrm{i})$ First of all, we must solve the steady state problem at the time $t=t_i$; $(\textrm{ii})$ Second, we must solve the generalized Riemann problem on the interval $t_i< t < t_{i+1}$.
Since the initial data is a smooth steady state solution, it is clear that Step $(\textrm{i})$ is exact. On the other hand, Lemma~\ref{approximate-gen} provides us a control of the error associated with the generalized Riemann problem and implies that Step $(\textrm{ii})$ is also exact. This completes our argument.  

\

\noindent{\bf Step 3. The well-balanced property for discontinuous steady states.} We start from the initial data $U_\Delta (t_0, \cdot) = U_0$ at some time $t_0$. 
Writing $U_L(r^\natural) =: U_L^0$ and $U_R(r^\natural) =: U_R^0$, we have either $U_R^0 \in S_1^\rightarrow (U_L^0) $ (if $|v_L^0|> k$ )
or $U_L^0 \in S_2^\leftarrow (U_R^0)$ (if $|v_R^0|> k$). For definiteness, we assume that  $U_L^0 \in S_2^\leftarrow (U_R^0)$. 
Consider the solution for the time interval $t_0 < t< t_1$, and consider the unique even number $j_0$ such that $r^\natural \in (r_{j_0 -1},r_{j_0+1} ]$. We distinguish between two cases:

\vskip.3cm

\noindent{\bf Case $r^\natural \neq r_{j_0}$.} 
The solution is a steady state solution with a shock at $r=r^\natural$. 

\vskip.3cm

\noindent{\bf Case $r^\natural = r_{j_0}$.} Wee solve the generalized Riemann problem at $r= r^\natural$. According to our construction, for all $t_0 < t< t_1$, the solution is defined by  
$$
U_\Delta(t, r) = 
\begin{cases}
U_L(r), \qquad   & r \in (2M, s_2^0(t-t_0) + r^\natural),
\\
U_R(r),               & r \in ( s_2^0(t-t_0) + r^\natural, +\infty),
\end{cases}  
$$
where 
$$
s_2^0 := 
\begin{cases}
s_2( U_L^0, U_R^0),      \qquad  & r^\natural= r_{j_0}, 
\\
0,                  & r^\natural \neq  r_{j_0}.
\end{cases} 
$$
To extend the construction, we solve the differential equation \eqref{steady2-con} iand obtain the steady state solution at the time level $t=t_1$. We write $r_1^\natural := s_2^0 \Delta t+ r^\natural$. Thanks to the stability condition \eqref{stability-condition}, we have $r_1^\natural  \in [r_{j_0},r_{j_0+1} ]$. The definition of the approximate solution depends on the position of $r_{1, j_0} = r_{j_0}+ w_i$. We have 
$$
U_\Delta(t_1, r) = 
\begin{cases} 
U_L(r), \qquad  & r \in (2M, r_{j_1}),
\\
U_R(r),           & r \in ( r_{j_1}, +\infty), 
\end{cases} 
$$
where $r_{j_1} =  r_{j_0-sgn(r_{1, j_0} - r_1^\natural)}$. We then solve the generalized Riemann problem at $r_{j_1} $. By induction, we find the solution defined on the time interval $[t_i, t_{i+1})$:  
\bel{def-on-slab}
U_\Delta(t, r) = 
\begin{cases} 
U_L(r),  \qquad  & r \in (2M, s_2^i(t-t_0) + r^\natural_{i+1}),
\\
U_R(r),               & r \in ( s_2^i(t-t_0) + r^\natural_{i+1}, +\infty), 
\end{cases}
\ee
where $s_2^i$ is (randomly) determined by the sequence $(w_i)$. This completes the proof.


\section{Remarks on special models}
\label{sec:9}

\subsection*{Stiff fluids on a Schwarzschild background}
\label{sec:91}

Consider now the model corresponding to the pressure-law $p = {1\over \eps^2} \rho$, so that the sound speed coincides with the light speed ${1\over \eps}$. That is, consider the Euler model for stiff fluid flows on a Schwarzschild background $\Mscr (\eps,{1\over \eps}, m)$ presented in \eqref{sonetlu}
Recall that it admits two real and distinct eigenvalues $\lambda=-(1-2M/r)/\eps$ and $\mu= (1-2M/r) \ \eps$. 
They satisfy $- {1\over \eps}<\lambda< 0 < \mu< 1/\eps$. and, in the limit $r \to +\infty$, we have $\lambda, \mu \to \pm {1\over \eps}$. 
According to Proposition~\ref{nonlinearpo}, the two characteristics fields are both linearly degenerate. Denote by $D^\rightarrow_1(U_L)$ and $D^\leftarrow_2(U_R)$ the 1- and 2-contact discontinuities (that is, the notions of shock and rarefaction coincide in this case) corresponding to any given constant states $U_L$ and $U_R$ respectively. 

\begin{lemma}[Riemann problem for stiff fluids]
Consider the Euler model  $\Mscr (\eps,{1\over \eps}, m)$ in \eqref{sonetlu}. Given any constant states $U_L, U_R$, there exists a unique intermediate $U_M$, such that $U_L$ can be connected to $U_M$ by a contact discontinuity with the speed $-(1-2M/r) / \eps$, while $U_M$ is connected to $U_R$ by a contact discontinuity with speed $(1-2M/r) /\eps$. 
\end{lemma}

Unlike the case when the sound speed is strictly less than the light speed, in this linearly degenerate regime, 
steady state solutions are always defined globally.  The system for steady state solutions reads 
\be
\label{nonlin}
\aligned
{d\over dr} \Bigg(r(r-2M) {\rho v \over 1 - \eps^2 v^2} \Bigg) &= 0,
\\
{d \over dr} \Bigg((r-2M)^2 {1 + \eps^2 v^2 \over 1 - \eps^2 v^2} \rho \Bigg) 
& = 2M  {r-2M \over  r} {1 + \eps^2 v^2 \over 1 - \eps^2 v^2} \rho  +2{(r-2M)^2 \over  r} \rho.\endaligned
\ee
 
\begin{lemma} By imposing an initial condition $\rho(r_0) = \rho_0$ and $v(r_0) = v_0$, 
the system \eqref{nonlin} has a unique global smooth solution given explicitly by 
\bel{extremeso}
\aligned
& \rho(r) =\Bigg(1- {r_0^4\eps^2 v_0^2\over r^4} \Bigg){(r_0-2 M) r  \over r_0 (r-2M) (1-\eps^2 v_0^2)} \rho_0,
\qquad
 v(r) = {r_0^2 \over  r^2} v_0. 
\endaligned
\ee
\end{lemma}

\begin{proof} By taking $k = 1/\eps$ in \eqref{eq:322}, we obtain 
$$
\aligned
& r(r-2M){\rho v\over 1 - \eps^2 v^2} =r_0(r_0-2M){\rho_0 v_0 \over 1 - \eps^2 v_0^2},\\
&  \Big( 1 - {2M \over r} \Big){\rho \over 1 - \eps^2 v^2} = {r_0-2 M \over r_0}{\rho_0 \over 1 - \eps^2 v_0^2},
\endaligned
$$
which we can solve explicitly for the density and velocity functions. 
\end{proof}

In view of the classical Riemann solver and the explicit form of the steady state solutions, it is now straighforward to follow all the steps of the general proof and check the following result. Our main observation here is that all of our earlier estimates when the sound speed is strictly less than the light speed are {\sl uniform} when the sound speed approaches the light speed.   

\begin{theorem}[Stiff fluid flows on a Schwarzschild background] 
\label{global-existenceSTIFF}
Consider the Euler model $\Mscr (\eps,{1\over \eps}, m)$ for stiff fluids evolving on a Schwarzschild background, as presented in \eqref{sonetlu}.
Given any initial density ${\rho_0=\rho_0(r) > 0}$ and velocity $v=v_0(r)$ defined for $r>2M$ and satisfying 
(for all $\delta> L> 0$)
$$
TV_{[2M+\delta, +\infty)}  \big( \ln \rho_0 \big) + 
TV_{[2M+\delta, +\infty)}  \Bigg( \ln {1-\eps v_0 \over 1 + \eps v_0} \Bigg) < + \infty,
$$
there exists a weak solution $\rho=\rho(t,r)$ and $v = v(t,r)$ satisfying the prescribed initial data at some given time $t_0$, together with the following bound on every time interval $[t_0, T]$ and for all $\delta, L> 0$ 
$$
\sup_{t \in [t_0,T]} \Bigg( 
TV_{[2M+\delta, +\infty)} \Big( \ln \rho(t,\cdot) \Big)  
+ TV_{[2M+\delta, +\infty)} \Bigg( \ln {1-\eps v (t, \cdot) \over 1 + \eps v(t, \cdot)} \Bigg) \Bigg)< + \infty.  
$$
\end{theorem}


\subsection*{Non-relativistic Euler equations on a Schwarzschild background}
\label{sec:92} 

In this section, we state the existence theory for the non-relativistic Euler system \eqref{nonrelativ-con}: 
$$
\aligned
&\del_t(r^2\rho) +\del_r(r^2\rho v) = 0,
\\
&\del_t(r^2\rho v) +\del_r\Big(r^2( v^2+k^2) \rho\Big) -2k^2 \rho r+m \rho= 0.
\endaligned
$$
For \eqref{nonrelativ-con}, we have the eigenvalues
$\lambda= v-k$ and $\mu= v+k$
and a pair of Riemann invariants: 
$w =-v-k \ln \rho$ and $z =-v+k \ln \rho$.
We can also give the form of the $1$-shock and the $2$-shock associated with the constant states $U_L$ and $U_R$ respectively: 
 \bel{Shock-non}
\aligned 
 & S_1^{\rightarrow}(U_L) = \Big\{ v- v_L=  -k \Big( \sqrt{\rho \over \rho_L} -\sqrt{\rho_L\over \rho} \Big), \quad
 \rho>\rho_L \Big\},
 \\
& S_2^{\leftarrow}(U_R) =  \Big\{ v-v_R= k \Big( \sqrt{\rho \over \rho_R} -\sqrt{\rho_R\over \rho} \Big), 
\quad
 \rho<\rho_R \Big\}. 
\endaligned 
\ee
A direct calculation gives the the rarefaction curves issuing from the constant states $U_L$ and $U_R$ respectively: 
\bel{Rarefac-non} 
\aligned
& R_1^{\rightarrow}(U_L) =  \Big\{ {v \over v_L} = \Big({\rho \over \rho_L} \Big)^{-k}, \quad
 \rho>\rho_L \Big\},
\qquad
 R_2^{\leftarrow}(U_R) =\Big\{ {v \over v_R} = \Big({\rho \over \rho_R} \Big)^k,
\quad 
 \rho<\rho_R \Big\}. 
\endaligned
\ee
In view of Proposition~\ref{Riemannpro}, we can solve the Riemann problem of the non-relativistic Euler equations with the help of \eqref{Shock-non} and \eqref{Rarefac-non}. Similarly as the case, the generalized Riemann problem requires a global steady state solution. 

Let $\rho=\rho(r; r_0, \rho_0, v_0)$ and $v=v (r; r_0, \rho_0, v_0)$ be a smooth steady state solution with sonic point of Euler equation \eqref{steady1-con} on $\Xi$. Recall the function $P$ in \eqref{function-P}
which determines the regime of the solutions: 
$$
 P(r_0, v_0) : = {3\over 2}+ \ln {m^2 \over 4 k^3 r_0^2 v_0}+ {1 \over 2 k^2}(v_0^2- {2 m \over r_0}). 
$$
Let $r_1^*$ be the unique point such that $P(r_1^*, {k^2\over v_1^*}) = 0$ where $v_1^*= v(r_1^*; r_0, v_0)$, and introduce the regions  
$$
\Lambda_s := 
\begin{cases} 
[r_1^*, +\infty), \quad  & r_1^* \geq {m\over 2 k^2},
\\
(0, r_1^*),          &r_1^* < {m\over 2 k^2}, 
\end{cases} 
\qquad \qquad
\Lambda_d= 
\begin{cases} 
(0, r_1^*], \quad  & r_1^*\geq {m\over 2 k^2},
\\
(r_1^*, +\infty),             & r_1^* <{m\over 2 k^2}.
\end{cases} 
$$
For this non-relativistic model, we can repeat our construction above. 

\begin{theorem} Consider the family of non-relativistic steady flows on Schwarzschild spacetime with the constant sound speed $k> 0$. Given arbitrary density $\rho_0 > 0$, velocity $v_0 \geq 0 $, and radius $r_0 > 0$, the boundary value problem of the steady Euler system \eqref{steady1-con} with 
$\rho(r_0) =\rho_0$ and $v(r_0) =v_0$, 
admits a global weak solution of \eqref{steady1-con} defined all $r \in (0, +\infty)$. 
\end{theorem}

Observe the the solutions are now defined in the whole half-line and that the eigenvalue $\lambda, \mu$ are not vanishing at $r=0$. By considering a domain $r >r_b$ for a given boundary radius $r_b>0$ and imposing the boundary condition $v=0$ at $r=r_b$, it is conceivable that the following statement could be established with our generalized random choice method.

\begin{theorem}[Non-relativistic fluid flows on a Schwarzschild background]
For the non-relativistic Euler system on a Schwarzschild background \eqref{nonrelativ-con} posed on $r> r_b$ and given any initial data $\rho_0=\rho_0(r) > 0$ and $v_0=v_0(r)$ and any boundary data $\rho_b=\rho_b(t)$ at $r=r_b$, satisfying for any $T>t_0$ 
$$
TV_{[r_b, +\infty)}  \big( \ln \rho_0 \big) +  TV_{[r_b, +\infty)} (v_0) + TV_{[t_0, T)} (\ln \rho_b) < + \infty, 
$$
then there exists a weak solution $\rho=\rho(t,r)$ and $v=v(t,r)$ defined for all $t\geq t_0$ and $r>r_b$ such that for all $T>t_0$ 
$$
\sup_{t \in [t_0,T]} \Bigg( 
TV_{[r_b, +\infty)]} \big( \ln \rho(t,\cdot) \big)   
+ TV_{[r_b, +\infty)} \big( v(t,\cdot) \big)  \Bigg) < + \infty. 
$$ 
\end{theorem}


\subsection*{Fluid flows in Minkowski spacetime}

When the black hole mass $M \to 0$ vanishes, the Schwarzschild metric approaches Minkowski metric and we find the Euler system \eqref{EulerMI}:, that is, 
\bel{eq-885}
\aligned 
& \del_t\Bigg( {1 + \eps^4 k^2 v^2 \over 1 - \eps^2 v^2}\rho \Bigg) +\del_r \Bigg( {1 + \eps^2 k^2 \over 1 - \eps^2 v^2}\rho  v\Bigg) = 0,
\\
& \del_t \Bigg({1 + \eps^2 k^2  \over 1 - \eps^2 v^2}\rho v \Bigg) + \del_r \Bigg( {v^2+ k^2 \over 1 - \eps^2 v^2} \rho \Bigg) = 0.
\endaligned 
\ee
We recover also the standard existence theory \cite{SmollerTemple} for this model. 

\begin{theorem}[Fluid flows in Minkowski spacetime]
Given any initial data $\rho_0=\rho_0(r) > 0$ and $|v_0|=|v_0(r)| \leq 1/ \eps$ defined for $r>0$ and satisfying 
$$
TV  \big( \ln \rho_0 \big) + TV \Bigg( {1-\eps v_0 \over 1 + \eps v_0}\bigg) < + \infty, 
$$
then there exists a corresponding weak solution $\rho=\rho(t,r)$ and $v = v(t,r)$ to \eqref{eq-885}, which is defined for all $t > t_0$ and all $r>0$ with 
$$
\sup_{t \in [t_0,T]} \Bigg( 
TV \big( \ln \rho(t,\cdot) \big)  
+ TV \Bigg( \ln {1-\eps v (t, \cdot) \over 1 + \eps v(t, \cdot)} \Bigg) \Bigg) < +\infty. 
$$ 
\end{theorem}


\section*{Acknowledgments}

This paper was written in April 2015 when the first author (PLF) enjoyed the hospitality of the Courant Institute of Mathematical Sciences at New York University. The authors were partially supported by the Centre National de la Recherche Scientifique (CNRS) and the Agence Nationale de la Recherche through the grant ANR SIMI-1-003-01. 


\end{document}